\documentclass[11pt]{article}

\usepackage{amsmath,epsfig,subfigure,amssymb,color,latexsym,amsfonts,amsthm}
\usepackage{graphicx,graphics,epstopdf}
\usepackage{float}

\title{Forced waves of parabolic-elliptic Keller-Segel models in shifting environments}

\author{
 Wenxian Shen\thanks{Partially supported by the NSF grant DMS--1645673}\,\, and\,\,  Shuwen Xue\\
Department of Mathematics and Statistics\\
Auburn University, AL 36849, USA }

\date{}

\begin{document}

\maketitle

\newtheorem{tm}{Theorem}[section]
\newtheorem{prop}{Proposition}[section]
\newtheorem{defin}{Definition}[section] 
\newtheorem{coro}{Corollary}[section]
\newtheorem{lem}{Lemma}[section]
\newtheorem{assumption}{Assumption}[section]
\newtheorem{rk}{Remark}[section]
\newtheorem{nota}[tm]{Notation}
\numberwithin{equation}{section}

\newcommand{\stk}[2]{\stackrel{#1}{#2}}
\newcommand{\dwn}[1]{{\scriptstyle #1}\downarrow}
\newcommand{\upa}[1]{{\scriptstyle #1}\uparrow}
\newcommand{\nea}[1]{{\scriptstyle #1}\nearrow}
\newcommand{\sea}[1]{\searrow {\scriptstyle #1}}
\newcommand{\csti}[3]{(#1+1) (#2)^{1/ (#1+1)} (#1)^{- #1
 / (#1+1)} (#3)^{ #1 / (#1 +1)}}
\newcommand{\RR}[1]{\mathbb{#1}}

\newcommand{\rd}{{\mathbb R^d}}
\newcommand{\ep}{\varepsilon}
\newcommand{\rr}{{\mathbb R}}
\newcommand{\alert}[1]{\fbox{#1}}
\newcommand{\eqd}{\sim}
\def\p{\partial}
\def\R{{\mathbb R}}
\def\N{{\mathbb N}}
\def\Q{{\mathbb Q}}
\def\C{{\mathbb C}}
\def\l{{\langle}}
\def\r{\rangle}
\def\t{\tau}
\def\k{\kappa}
\def\a{\alpha}
\def\la{\lambda}
\def\De{\Delta}
\def\de{\delta}
\def\ga{\gamma}
\def\Ga{\Gamma}
\def\ep{\varepsilon}
\def\eps{\varepsilon}
\def\si{\sigma}
\def\Re {{\rm Re}\,}
\def\Im {{\rm Im}\,}
\def\E{{\mathbb E}}
\def\P{{\mathbb P}}
\def\Z{{\mathbb Z}}
\def\D{{\mathbb D}}
\newcommand{\ceil}[1]{\lceil{#1}\rceil}

\begin{abstract}
The current paper is concerned with the forced waves of  Keller-Segel chemoattraction systems in shifting environments of the form,
\begin{equation}\label{abstract-eq1}
\begin{cases}
u_t=u_{xx}-\chi(uv_x)_x +u(r(x-ct)-bu),\quad x\in\mathbb{R}\cr
0=v_{xx}- \nu v+\mu u,\quad x\in\mathbb{R},
\end{cases}
\end{equation}
where  $\chi$, $b$, $\nu$, and $\mu$ are positive constants, $c\in\mathbb{R}$, the resource function $r(x)$ is globally  H\"older continuous, bounded, $r^*=\sup_{x\in\R}r(x)>0$, $r(\pm \infty):=\lim_{x\to \pm\infty}r(x)$ exist, and   either $r(-\infty)<0<r(\infty)$, or $r(\pm\infty)<0$. Assume  that  $b>2\chi\mu$. In the case that $r(-\infty)<0<r(\infty)$, it is shown that \eqref{abstract-eq1} has  a forced wave solution connecting $(\frac{r^*}{b},\frac{\mu}{\nu}\frac{r^*}{b})$ and $(0,0)$ with speed $c$  provided that $c>\frac{\chi\mu r^*}{2\sqrt \nu (b-\chi\mu)}- 2\sqrt{\frac{r^*(b-2\chi\mu)}{b-\chi\mu}}$. In the case that $r(\pm\infty)<0$, it is shown that \eqref{abstract-eq1} has a forced wave solution connecting $(0,0)$ and $(0,0)$ with speed $c$ provided that $\chi$ is sufficiently small and $\lambda_\infty>0$, where $\lambda_\infty$ is the generalized principal eigenvalue of the operator $u(\cdot)\mapsto u_{xx}(\cdot)+cu_{x}(\cdot)+r(\cdot)u(\cdot)$  on $\R$ in certain sense.  Some numerical simulations are also carried out. The simulations indicate the existence of forced wave solutions in some parameter regions which are not covered in the theoretical results, induce several problems to be further studied, and also provide some illustration of  the theoretical results.
\end{abstract}



\section{Introduction}

This work is concerned with the forced wave solutions   of the attraction Keller-Segel chemotaxis models in  shifting environments of the form
\begin{equation}\label{Keller-Segel-eq0}
\begin{cases}
u_t= u_{xx}- (\chi u  v_x)_x+u(r(x-ct)-bu),\quad x\in\R\cr
0 =v_{xx}-  \nu v +\mu u,\quad x\in\R,
\end{cases}
\end{equation}
where $ b, \nu,\mu$ and $\chi$ are positive constants, {$c\in\R$} and $u(t,x)$ and $v(t,x)$ represent the densities of a
 mobile species and a chemical substance, respectively.  Biologically, the positive constant $\chi$ measures the effect on the mobile species by  the chemical substance which is produced overtime by the mobile species;   the reaction term $u(r(x-ct)-bu)$ in the first  equation of \eqref{Keller-Segel-eq0} describes the local dynamics of the mobile species which depends on the density $u$ and on the shifting habitat with a fixed speed $c$; $\nu$  represents the degradation rate of the  chemical substance;  and $\mu$ is the rate at which the mobile species produces the chemical substance.

  System \eqref{Keller-Segel-eq0} with $r(\cdot)$ being a constant function is a simplified version of the chemotaxis system proposed by Keller and Segel in their works \cite{KeSe1,KeSe2}.
  Chemotaxis describes the oriented movements of biological cells and organisms in response to chemical gradient which they may produce themselves  over time
   and is crucial for many phenomena such
as the location of food sources, avoidance of predators and attracting mates, slime mold aggregation,
tumor angiogenesis, and primitive streak formation. Chemotaxis is also crucial in
macroscopic process such as population dynamics and gravitational collapse. The reader is referred to  \cite{KJPainter} for  applications of chemotaxis models in a wide range of biological phenomena.
  The reader is also referred to \cite{HiPa, Hor}  for some detailed introduction into the mathematics of Keller-Segel chemotaxis models.

In this paper, we consider \eqref{Keller-Segel-eq0} with $c\not =0$ and $r(\cdot)$ being a sign changing function.
 In particular, we will consider the following two cases:

\medskip

\noindent {\bf Case 1.} {\it Favorable and unfavorable habitats are separated in the sense that  $r(x)$ is globally  {H\"older continuous}, bounded,   the limits $r(\pm\infty):=\lim_{x\to \pm\infty}r(x)$ exist and are finite, and $r(-\infty)<0<r(\infty)$}, $r(-\infty)\leq r(x)\leq r(\infty),\, \forall\, x\in\R$.

\medskip

\noindent {\bf Case 2.} {\it Favorable habitat is surrounded by unfavorable habitat in the sense that {$r(x)$ is globally  H\"older continuous, bounded,}  $\sup_{x\in \R}r(x)>0$,  the limits $r(\pm\infty):=\lim_{x\to \pm\infty}r(x)$ exist and are finite, and
 $r(\pm\infty)<0$, $\min\{r(\infty),r(-\infty)\}\leq r(x),\, \forall\, x\in\R$. }

\medskip

  As described in \cite{ShXu},   in {\bf Case 1}, $r(x-ct)$ divides the spatial domain into two regions: the region with good-quality habitat suitable for growth \{$x\in\R$: $r(x-ct)>0$\} and the region with poor-quality habitat unsuitable for growth \{$x\in\R$: $r(x-ct)<0$\}. The edge of the habitat suitable for species growth is shifting at a speed $c$.
 In {\bf Case 2}, $r(x-ct)$ still divides the spatial domain into two regions: one favorable for growth \{$x\in\R$: $r(x-ct)>0$\} and one unfavorable for growth \{$x\in\R$: $r(x-ct)<0$\}. The favorable habitat is bounded and surrounded by the unfavorable habitat. The favorable habitat is shifting at a speed $c$.

\smallskip

Consider \eqref{Keller-Segel-eq0}. It is important to know whether the species will become extinct or  persist;  whether it spreads into larger and larger  regions, and if so, how fast it spreads; whether the system has so called forced wave solutions.
A positive solution $(u(t,x),v(t,x))$ of \eqref{Keller-Segel-eq0} is called a {\it forced wave solution} if
 it is defined for all $t\in\R$, {$x\in\R$}, and $(u(t,x),v(t,x))=(\phi(x-ct),\psi(x-ct))$ for some one variable functions $\phi(\cdot)$ and $\psi(\cdot)$.

\smallskip

There are many studies on these  problems  for \eqref{Keller-Segel-eq0} in the absence of the chemotaxis (i.e. $\chi=0$).
Note that in the absence of the chemotaxis $(\chi=0)$, the first equation in \eqref{Keller-Segel-eq0} is decoupled from the second equation, and
the dynamics of \eqref{Keller-Segel-eq0} is determined by the first equation of \eqref{Keller-Segel-eq0}, that is,
\begin{equation}\label{fisher-kpp1}
u_t=u_{xx}+u(r(x-ct)-bu), \quad x\in\R.
\end{equation}

For example,
in {\bf  Case 1},
Li et al. \cite{LBSF} studied the spatial dynamics of system \eqref{fisher-kpp1} { for the case $b=1$} and they showed that the persistence and spreading dynamics of \eqref{fisher-kpp1} depend on the speed of the shifting habitat edge $c$ and a number $c^*$, where $c^*=2\sqrt{r(\infty)}$ for \eqref{fisher-kpp1}. More precisely, they proved that if $c>c^*$, then the species will become extinct in the habitat, and if $0<c<c^*$, the species will persist and spread along the shifting habitat at the asymptotic spreading speed $c^*$. Recently, Hu and Zou \cite{HHXZ} demonstrated that { in the case $b=1$, }for any given speed $c>0$ of the shifting habitat edge, \eqref{fisher-kpp1} admits a nondecreasing traveling wave
solution $u(t,x)=\phi(x-ct)$  connecting $0$ and $r(\infty)$ (i.e. $\phi(-\infty)=0$ and $\phi(\infty)=r(\infty)$)  with the speed $c$ agreeing to the habitat shifting speed, which accounts for an extinction wave.
Very recently, Wang and Zhao \cite{WaZh} obtained the uniqueness of the forced wave of \eqref{fisher-kpp1} by using the sliding technique and established the global exponential stability of the  forced wave via the monotone semiflows approach combined with the method of super- and subsolutions (see \cite[Theorem 2.3]{WaZh}).

In {\bf Case 2},
Berestycki et al. \cite{BDNZ} proposed to use the following  reaction-diffusion equation with a forced speed $c>0$ to study the influence of climate change on the population dynamics of biological species:
\begin{equation}\label{B-fisher-kpp1}
u_t= u_{xx}+f(x-ct, u), \quad x\in\R.
\end{equation}
A typical $f$ considered in \cite{BDNZ} is
$$
f(x,u)=\begin{cases}
au(1-\frac{u}{K}),\quad &0\leq x \leq L,\\
-ru,\quad &x<0 \,\ {\rm and }\,\ x>L
\end{cases}
$$
for some positives constants $a, r, K, L$. They first considered this special case and derived an explicit condition for the persistence of species by gluing phase portraits. Then they established a strict qualitative dichotomy for a large class of models by the rigorous PDE methods. More precisely, they showed that if $\lambda_{\infty}$, defined to be the generalized {principal} eigenvalue of the operator $u\to u_{xx}+cu_{x}+f_{u}(x,0)u$ on $\R$, is less than or equal to zero, then \eqref{B-fisher-kpp1} has no forced wave solution and every positive solution of \eqref{B-fisher-kpp1} converges to zero as $t\to\infty$, uniformly in $x$. If $\lambda_{\infty}>0$, \eqref{B-fisher-kpp1} has a unique forced wave solution and every nontrivial positive solution of \eqref{B-fisher-kpp1} converges to this unique forced wave solution as $t\to\infty$, uniformly in $x$.

It should be pointed out that the paper \cite{PoLe}  addresses the same question as in \cite{BDNZ}, but focuses on the effect
of a moving climate on the outcome of competitive interactions between two species.
There are also many studies on forced wave solution of \eqref{B-fisher-kpp1} with different shifting habitats. For example,  Fang, Lou and Wu \cite{FaLoWu} established the existence and nonexistence of forced waves and
pulse waves of \eqref{B-fisher-kpp1} with $f$ of the form $f(x-ct, u)=u(r(x-ct)-u)$ where $r(-\infty)>0>r(\infty)$. Berestycki and Fang \cite{BeFa} obtained the complete existence and multiplicity of forced waves as well as their attractivity except for some critical cases under the condition $f:(s,u)\in\R\times\R_{+}\to\R$ is asymptotically of KPP type. For other related works on climate change problem with different shifting habitats for certain reaction-diffusion equations, nonlocal dispersal equations, lattice differential equations,   as well as integro-difference equations, we refer the readers to \cite{DuWeZh, CHBL, LeShZh,  ChDu, LeMaSh, LiBeBaFa, LiWaZh,  Vo, WaZh, ZhKo} and the references therein.

Observe that, when $r(x)\equiv r$, \eqref{fisher-kpp1} becomes
\begin{equation}\label{fisher-kpp0}
u_t=u_{xx}+u(r-bu), \quad x\in\R.
\end{equation}
 Due to the pioneering works of Fisher
\cite{Fisher} and Kolmogorov, Petrowsky, Piskunov \cite{KPP} on traveling wave solutions and take-over properties of \eqref{fisher-kpp0},
\eqref{fisher-kpp0} is also referred to as  the Fisher-KPP equation.
 The following results are well known about the traveling wave solutions and  spreading speeds of \eqref{fisher-kpp0}.
Equation \eqref{fisher-kpp0} has  traveling wave solutions $u(t,x)=\phi(x-ct)$
connecting $\frac{r}{b}$ and $0$ (i.e. $\phi(-\infty)=\frac{r}{b}$ and $\phi(\infty)=0$) of all speeds $c\geq 2\sqrt r$ and
has no such traveling wave solutions of slower speed. Such a traveling wave solution is unique up to a translation.
For any
nonnegative solution $u(t,x)$ of (\ref{fisher-kpp0}), if at
time $t=0$, $u(0,x)=u_0(x)$ is $\frac{r}{b}$ for $x$ near $-\infty$ and $0$ for $x$ near $ \infty$, then
$$\limsup_{x \ge ct, t\to \infty}u(t,x)=0 \quad \forall \, c>2\sqrt r
$$
and
$$\limsup_{x \le ct, t\to \infty}|u(t,x)-\frac{r}{b}|=0\quad \forall\,  c<2\sqrt r.
$$
In
literature, $c^*_0=2\sqrt r$ is   called the {\it
spreading speed} for \eqref{fisher-kpp0}.  Since the pioneering works by  Fisher \cite{Fisher} and Kolmogorov, Petrowsky,
Piscunov \cite{KPP},  a huge amount of research has been carried out toward the front propagation dynamics of
  reaction diffusion equations of the form,
\begin{equation}
\label{general-fisher-eq}
u_t=\Delta u+u f(t,x,u),\quad x\in\R^N,
\end{equation}
where $f(t,x,u)<0$ for $u\gg 1$,  $\partial_u f(t,x,u)<0$ for $u\ge 0$ (see \cite{ArWe2, BHN,  BeHaNa1, BeHaNa2, Henri1, Fre, FrGa, LiZh, LiZh1, Nad, NoRuXi, NoXi1, She1, She2, Wei1, Wei2, Zla}, etc.).

Comparing with the Fisher-KPP equation, traveling wave solutions and  spreading speeds of chemotaxis models have only been studied recently.
For example, Salako and Shen studied in \cite{SaSh3, SaSh2, SaSh1} the spatial spreading dynamics of \eqref{Keller-Segel-eq0} with constant growth rate $r>0$
 and obtained several fundamental results. Some lower and upper bounds for the propagation speeds of solutions with compactly supported initial functions were derived, and some lower bound for the speeds of traveling wave solutions was also derived. It is proved that all these bounds converge to the spreading speed $c_0^*=2\sqrt r$ of \eqref{fisher-kpp0} as $\chi\to 0$.  Assume that $2\chi\mu<b$. It's also proved that there is a positive constant $c^*(\chi, r, b, \nu, \mu)>c^*_0=2\sqrt{r}$ such that  for any $c>c^*(\chi, a, b, \lambda, \mu)$, \eqref{Keller-Segel-eq0} with constant growth rate $r>0$ has a traveling wave solution $(u(t,x),v(t,x))=(U(x-ct),V(x-ct))$ with speed $c$ connecting the constant solutions $(\frac{r}{b}, \frac{\mu}{\nu}\frac{r}{b})$ and $(0,0) ${(i.e $U(-\infty)=\frac{r}{b}$ and $U(\infty)=0$)}.
 The reader is also referred to  \cite{FhCh} for the lower and upper bounds of propagation speeds of \eqref{Keller-Segel-eq0} with constant growth rate.

Very recently, the authors of the  paper \cite{SaShXu} improved the results in \cite{SaSh2}.
It is proved that in the case of constant growth rate $r>0$, if $b>\chi\mu$ and $\big(1+\frac{1}{2}\frac{(\sqrt{r}-\sqrt{\nu})_+}{(\sqrt{r}+\sqrt{\nu})}\big)\chi\mu { \leq} b$ hold, then $2\sqrt r$ is the spreading speed of the solutions $(u(t,x;u_0),v(t,x;u_0))$ of   \eqref{Keller-Segel-eq0} with nonnegative  compactly supported continuous initial functions $u_0$, that is,
 $$
 \lim_{t\to\infty}\sup_{|x|\ge ct}u(t,x;u_0)=0\quad \forall\, c>2\sqrt r
 $$
 and
 $$
 \liminf_{t\to\infty}\inf_{|x|\le ct} u(t,x;u_0)>0\quad \forall \, 0<c<2\sqrt r.
 $$
 It is also proved that,  if  $b>2\chi\mu$  and $\nu\ge r$ hold, then  $c_0^*=2\sqrt r$ is the minimal speed  of the  traveling wave solutions of \eqref{Keller-Segel-eq0} with constant growth rate $r>0$  connecting $(\frac{r}{b},\frac{\mu}{\nu}\frac{r}{b})$ and $(0,0)$, that is,
  for any $c\ge c_0^*$, \eqref{Keller-Segel-eq0} with constant growth rate $r>0$
 has a  traveling wave solution connecting $(\frac{r}{b},\frac{\mu}{\nu}\frac{r}{b})$ and $(0,0)$ with speed $c$,
 and \eqref{Keller-Segel-eq0} with constant growth rate $r>0$ has no such traveling wave solutions with speed less than $c_0^*$,
 where $(u(t,x;u_0),v(t,x;u_0))$ is the unique global classical solution of   \eqref{Keller-Segel-eq0} with $u(0,x;u_0)=u_0(x)$.

 In the recent work \cite{ShXu}, the authors of the current paper investigated the persistence and spreading speeds of \eqref{Keller-Segel-eq0}
 when $r(x)$ is as in {\bf Case 1} or {\bf Case 2}.
Assume $b>\chi\mu$ and  $b\ge \big(1+\frac{1}{2}\frac{(\sqrt{r^*}-\sqrt{\nu})_+}{(\sqrt{r^*}+\sqrt{\nu})}\big)\chi\mu$,
where $r^*=\sup_{x\in\R} r(x)$.
In the {\bf Case 1}, it is shown that if the moving speed {$c>c^*:=2\sqrt{r^*}$},
then the species becomes extinct in the habitat. If the moving speed $ -c^*\leq c<c^*$, then the species will persist and spread along the shifting habitat  at the asymptotic spreading speed $c^*$. If the moving speed $c<-c^*$, then the species will spread in the both directions at the asymptotic spreading speed $c^*$. In the {\bf Case 2}, it is shown that  if $|c|>c^*$, then the species will become extinct in the habitat. If $\lambda_{\infty}$, defined to be the generalized {principal} eigenvalue of the operator $u\to u_{xx}+cu_{x}+r(x)u$ on $\R$, is negative and the degradation rate $\nu$ of the {chemical substance} is greater than or equal to {the number $\nu^*:={\frac{(\sqrt{8r^*+c^2}+|c|)^2}{4}}$ }, then the species will also become extinct in the habitat. If $\lambda_{\infty}>0$, then the species will persist surrounding the good habitat.

  The objective of the current paper is to investigate  the existence of forced wave solutions  of \eqref{Keller-Segel-eq0}
  both theoretically and numerically.

 In the rest of this introduction, we state the main results of this paper.

\subsection{Statement of  the  main results.}

In order to state our main results, we first introduce some notations and definitions. Let
$$
C^b_{\rm unif}(\R)=\{ u\in C(\R) \, |\, \ u\ \text{is uniformly continuous and bounded on $\R$}\}.
$$
For every $u\in C^b_{\rm unif}(\R)$, we let $\|u\|_{\infty}:=\sup_{x\in\R}|u(x)|$. For each given  $u_0\in C^b_{\rm unif}(\R)$  with $u_0(x)\geq 0$, we denote by $(u(t,x;u_0),v(t,x;u_0))$ the classical solution of \eqref{Keller-Segel-eq0} satisfying $u(0,x;u_0)=u_0(x)$ for every $x\in\R$. Note that, by the comparison principle for parabolic equations, for every  nonnegative initial function $u_0\in C^b_{\rm unif}(\R)$, it always holds that $u(t,x;u_0)\geq 0$ and $v(t,x;u_0)\geq 0$ whenever $(u(t,x;u_0),v(t,x;u_0))$ is defined. In this work we shall only focus on nonnegative classical solutions of \eqref{Keller-Segel-eq0}  since both functions $u(t,x)$ and $v(t,x)$ represent density functions.

The following proposition states the existence and uniqueness of solutions of \eqref{Keller-Segel-eq0} with given initial functions.

\begin{prop}\label{existence-uniqueness}
Suppose that $r(x)$ is globally H\"older continuous and bounded.
For every nonnegative initial function $u_0\in C^b_{\rm unif}(\R)$, there is a unique maximal time $T_{max}>0$, such that $(u(t,x;u_0),v(t,x;u_0))$ is defined for every $x\in\R$ and $0\le t<T_{\max}$. Moreover if $\chi\mu<b$ then $T_{max}=\infty$ and the solution is globally bounded.
\end{prop}

The above proposition can be proved by similar arguments as those in (\cite[Theorem 1.1 and Theorem 1.5]{SaSh1}).

Throughout this paper, we assume that $r(x)$ is as in {\bf Case 1} or {\bf Case 2}. We put
\begin{equation}
\label{r-c-star}
r_*=\inf_{x\in\R} r(x),\quad r^*=\sup_{x\in\R} r(x),\quad c^*=2\sqrt {r^*}.
\end{equation}
Note that, in {\bf Case 1}, $r_*=r(-\infty)$ and $r^*=r(+\infty)$, and in {\bf Case 2}, $r_*=\min\{ r(-\infty), r(+\infty)\}$ and $r^*=\max_{x\in\R} r(x)$.

Let $\lambda_L(r(\cdot),c)$ be the principal eigenvalue of
\begin{equation}
\label{ev-eq0}
\begin{cases}
\phi_{xx}+c\phi_x+ r(x) \phi=\lambda \phi,\quad -L<x<L\cr
\phi(-L)=\phi(L)=0.
\end{cases}
\end{equation}
Note that $\lambda_L(r(\cdot),c)$ is increasing as $L$ increases.
Let $\lambda_\infty(r(\cdot),c)=\lim_{L\to\infty}\lambda_L(r(\cdot),c)$.

For convenience, we make the following standing assumptions.

\medskip

\noindent {\bf (H1)}
$b>2\chi\mu$ and $c>\frac{\chi\mu r^*}{2\sqrt{\nu}(b-\chi\mu)}-2\sqrt{\frac{r^*(b-2\chi\mu)}{b-\chi\mu}}$.

 \medskip

 \noindent {\bf (H2)} $b\geq \frac{3}{2}\chi\mu$ and $\lambda_\infty(r(\cdot),c)>0$.

 \medskip

 Note  that $\lambda_\infty(r(\cdot),c)>0$ implies that $-2\sqrt{r^*}<c<2\sqrt {r^*}$.

A positive solution $(u(t,x),v(t,x))$ of \eqref{Keller-Segel-eq0} is  called a {\it forced wave solution} if
 it is defined for all $t\in\R$, {$x\in\R$} and $(u(t,x),v(t,x))=(\phi(x-ct),\psi(x-ct))$ for some one variable functions $\phi(\cdot)$ and $\psi(\cdot)$.
 It is clear that $(u,v)=(\phi(x),\psi(x))$ is a stationary solution of
 \begin{equation}\label{wave-eq}
\begin{cases}
u_t= u_{xx}+cu_x- (\chi u  v_x)_x+u(r(x)-bu),\quad x\in\R\cr
0 =v_{xx}-  \nu v +\mu u,\quad x\in\R.
\end{cases}
\end{equation}
We say that a positive {\it  forced wave solution $(u(t,x),v(t,x))=(\phi(x-ct),\psi(x-ct))$  of \eqref{Keller-Segel-eq0} connects $(u^*_+,v^*_+)$ and
$(u^*_-,v^*_-)$}  if $(\phi(\pm\infty),\psi(\pm\infty))=(u^*_\pm,v^*_\pm)$.

 \medskip

Our main theoretical results are stated in the following two theorems.

\begin{tm}
\label{forced-wave-thm1}
 Suppose that $r(x)$ is as in {\bf Case 1},  and {\bf (H1)} holds. Then there is a forced wave solution
$(u(t,x),v(t,x))=(\phi(x-ct),\psi(x-ct))$ connecting {$(\frac{r^*}{b},\frac{\mu}{\nu}\frac{r^*}{b})$ and $(0,0)$.}
\end{tm}

\begin{tm}
\label{forced-wave-thm2}
Suppose that $r(x)$ is as in {\bf Case 2}, and {\bf (H2)} holds. Then there exists
a number $\chi_0=\chi_0(r(\cdot),c)>0$ such that for any $0<\chi<\chi_0$,
there is a forced wave solution $(u(t,x),v(t,x))=(\phi(x-ct),\psi(x-ct))$ connecting $(0,0)$ and $(0,0)$, that is,
$\phi(x)>0$ for all $x\in\R$ and $\phi(\pm\infty)=0$.
\end{tm}

\begin{rk}
\begin{itemize}

\item[(1)]   When $\chi=0$,   {\bf (H1)} becomes $c>-2\sqrt {r^*}$. Hence Theorem  \ref{forced-wave-thm1} recovers \cite[Theorem 1.1]{HHXZ}.

\item[(2)] In {\bf Case 1}, for any given $c>-2\sqrt{r^*}$,
there is $\chi_0=\chi_0(c)>0$ such that for any $0<\chi<\chi_0$,  {\bf (H1)} holds. Then
 by Theorem \ref{forced-wave-thm1},    for any $0<\chi<\chi_0$,  there is a forced wave solution
$(u(t,x),v(t,x))=(\phi(x-ct),\psi(x-ct))$ connecting $(\frac{r^*}{b},\frac{\mu}{\nu}\frac{r^*}{b})$ and $(0,0)$.

\item[(3)]    When $\chi=0$, Theorem  \ref{forced-wave-thm2} recovers
     \cite[Theorem 4.8]{BDNZ}.

\item[(4)] Thanks to the presence of chemotaxis, the comparison principle for parabolic equations cannot be applied directly to \eqref{Keller-Segel-eq0} or \eqref{wave-eq}, and the techniques used in the study of \eqref{fisher-kpp1} and \eqref{B-fisher-kpp1}
    are difficult to be applied to  the study of \eqref{Keller-Segel-eq0}.
We used  Schauder's fixed point theorem together with sub- and super-solutions in the study of the existence of forced wave solutions of
\eqref{Keller-Segel-eq0}. We point out that
 the construction of some appropriate sub-solutions is highly nontrivial in both cases.

 \item[(5)] The condition $b>2\chi\mu$ in ${\bf (H1)}$  or $b\ge \frac{3}{2}\chi\mu$ in ${\bf (H2)}$
 indicates that the logistic damping is large relative to the chemotaxis sensitivity. The reader is referred to \cite{TeWi2}
 for the pioneering study on the global existence of classical solutions of
  the chemotaxis models  on bounded domains with relatively large logistic damping.
\end{itemize}
\end{rk}

Observe that the conditions in Theorem \ref{forced-wave-thm1} (resp. in Theorem \ref{forced-wave-thm2}) are sufficient conditions for the existence of forced wave solutions.   To see whether \eqref{Keller-Segel-eq0} still has forced wave solutions when these sufficient conditions are not satisfied,  some numerical simulations are carried out and are described  in the following (see section 5 for detail).

In the case that $r(\cdot)$ is as in {\bf Case 1}, consider
\begin{equation}\label{wave-cut-off-eq1}
\begin{cases}
u_t= u_{xx}+cu_x- (\chi u  v_x)_x+u(r(x)-bu),\quad -L<x<L\cr
0 =v_{xx}-  \nu v +\mu u,\quad -L<x<L\cr
u(t,-L)=v(t,-L)=0\cr
{\frac{\partial u}{\partial x}(t,L)=\frac{\partial v}{\partial x}(t,L)=0.}
\end{cases}
\end{equation}
Note that, if \eqref{wave-cut-off-eq1} has a  positive stationary  solution
$(u_L(x),v_L(x))$ for all $L\gg 1$, then there is $L_k\to\infty$ such that
$\lim_{k\to\infty}(u_{L_k}(x),v_{L_k}(x))$ exists locally uniformly in $x\in\R$ and
$(u^*(x),v^*(x))=\lim_{k\to\infty}(u_{L_k}(x),v_{L_k}(x))$ is a nonnegative stationary solution of \eqref{wave-eq}.
If  $(u^*(x),v^*(x))$ is positive, we then have a positive forced wave solution of \eqref{Keller-Segel-eq0}.
We will numerically {analyze} the existence of positive stationary solutions of \eqref{wave-eq} by
looking at the long time behavior of numerical solutions of \eqref{wave-cut-off-eq1} with a given {nonnegative} initial function for sufficiently large $L$.

We observe from the numerical experiments 1-4 in section 5.1 that the assumptions in Theorem \ref{forced-wave-thm1} can be weakened. We conjecture that if $b>\chi\mu$, and $c>-2\sqrt{r^*}$, there is a forced wave solution
$(u(t,x),v(t,x))=(\phi(x-ct),\psi(x-ct))$ of \eqref{Keller-Segel-eq0} connecting $(\frac{r^*}{b},\frac{\mu}{\nu}\frac{r^*}{b})$ and $(0,0)$. If $b>\chi\mu$ and $c<-2\sqrt{r^*}$, there is no  forced wave solution
$(u(t,x),v(t,x))=(\phi(x-ct),\psi(x-ct))$ connecting $(\frac{r^*}{b},\frac{\mu}{\nu}\frac{r^*}{b})$ and $(0,0)$.

\smallskip

In the case that $r(\cdot)$ is as in {\bf Case 2},
consider the following cut-off problem of \eqref{wave-eq},
\begin{equation}\label{wave-cut-off-eq2}
\begin{cases}
u_t= u_{xx}+cu_x- (\chi u  v_x)_x+u(r(x)-bu),\quad -L<x<L\cr
0 =v_{xx}-  \nu v +\mu u,\quad -L<x<L\cr
u(t,-L)=v(t,-L)=0\cr
u(t,L)=v(t,L)=0.
\end{cases}
\end{equation}
We will numerically {analyze} the existence of positive stationary solutions of \eqref{wave-eq} by
looking at the long time behavior of numerical solutions of \eqref{wave-cut-off-eq2} with a given {nonnegative} initial function for sufficiently large $L$.
We also observe  from the numerical experiments 1-3 in section 5.2 that  $\chi$ is not necessarily small and the assumption $b\geq\frac{3}{2}\chi\mu$ can also be weakened for the existence of forced waves. We conjecture that if $b>\chi\mu$ and $\lambda_{\infty}(r(\cdot))>0$, there is a forced wave solution $(u(t,x),v(t,x))=(\phi(x-ct),\psi(x-ct))$ of \eqref{Keller-Segel-eq0} connecting $(0,0)$ and $(0,0)$. If $b>\chi\mu$ and $|c|>c^*$, there is no forced wave solution $(u(t,x),v(t,x))=(\phi(x-ct),\psi(x-ct))$ connecting $(0,0)$ and $(0,0)$.

We plan to provide some further study on the scenarios observed numerically somewhere else.

The rest of the paper is organized as follows. In section 2, we present some preliminary lemmas to be used in the proofs of the main results.
 In section 3,
we study the existence of forced wave solutions of \eqref{Keller-Segel-eq0} with $r(x)$ being as in {\bf Case 1}
and prove Theorem \ref{forced-wave-thm1}.
In section 4, we study the existence of forced wave solutions of \eqref{Keller-Segel-eq0} with $r(x)$ being as in {\bf Case 2}
and prove Theorem \ref{forced-wave-thm2}. In section 5, we present some numerical simulations.

\section{Preliminary lemmas}

In this section, we present some preliminary lemmas to be used in the proofs of the main theorems in later sections.

Note that, by the second equation in \eqref{Keller-Segel-eq0},
$v_{xx}=\nu v-\mu u.$
Hence the first equation in \eqref{Keller-Segel-eq0} can be written as
$$
u_t=u_{xx}-\chi v_{x} u_{x}+ u\big( r(x-ct)-\chi \nu v-(b-\chi\mu)u\big),\quad x\in\R.
$$
By the comparison principle for parabolic equations, if $b>\chi\mu$, then for any $u_0\in C_{\rm unif}^b (\R)$ with $u_0\ge 0$,
$$
0\le u(t,x;u_0)\le { \max}\{\|u_0\|_\infty, \frac{r^*}{b-\chi\mu}\}\quad \forall\,\ t\ge 0,\,\, x\in\R.
$$

Consider
\begin{equation}\label{Le2.6-eq1}
\begin{cases}
u_t=u_{xx}-\chi(uv_x)_x +u(r^*-bu),\quad x\in\R\cr
0=v_{xx}- \nu v+\mu u,\quad x\in\R.
\end{cases}
\end{equation}

\begin{lem}
\label{lem-001-4}
Assume that $b>2\chi\mu$.
\begin{itemize}
\item[(1)] For any $u_0\in C_{\rm unif}^b(\R)$ with $\inf_{x\in\R}u_0(x)>0$,
$$
\lim_{t\to\infty} \|u(t,\cdot;u_0)-\frac{r^*}{b}\|_\infty=0,
$$
where $(u(t,x;u_0),v(t,x;u_0))$ is the solution of \eqref{Le2.6-eq1} with $u(0,x;u_0)=u_0(x)$.

\item[(2)] If $(u^*(t,x),v^*(t,x))$ is an entire solution of \eqref{Le2.6-eq1} satisfying that
$$0<\inf_{t\in\R,x\in\R}u^*(t,x)\le \sup_{t\in\R,x\in\R}u^*(t,x)<\infty,
$$
 then
$$
(u^*(t,x),v^*(t,x))\equiv\Big (\frac{r^*}{b},\frac{\mu}{\nu}\frac{r^*}{b}\Big).
$$
\end{itemize}

\end{lem}

\begin{proof}
(1) It follows from \cite[Theorem 1.8]{SaSh1}.

(2) { It follows from the arguments in \cite[Theorem 1.8]{SaSh1}.
For completeness, we provide a proof in the following.

First, let
$$
\underline{u}=\inf_{t\in\R,x\in\R}u^*(t,x),\quad \overline{u}=\sup_{t\in\R,x\in\R}u^*(t,x).
$$
Then
$$
0<\underline{u}\le u^*(t,x)\le \overline{u}<\infty\quad \forall\, t\in\R,\,\, x\in\R
$$
and
$$
0<\frac{\mu \underline{u}}{\nu}\le v^*(t,x)\le \frac{\mu \overline{u}}{\nu}<\infty\quad \forall\, t\in\R,\,\, x\in\R.
$$
Note that if $\underline{u}=\overline{u}$, then we must have $(u^*(t,x),v^*(t,x))\equiv \Big (\frac{r^*}{b},\frac{\mu}{\nu}\frac{r^*}{b}\Big)$. It then suffices to prove $\underline{u}=\overline{u}$.

Next, let  ${w}(t)$ be the solution of
$$
\begin{cases}
w_t=w(r^*-(b-\chi\mu)w)\cr
w(0)=\overline{u}.
\end{cases}
$$
Then
$$
\lim_{n\to\infty}{w}(t+n)=\frac{r^*}{b-\chi\mu}\quad\forall\, t\in\R.
$$
Note that
\begin{align}
\label{revise-eq0}
u_t&=u_{xx}-\chi(uv_x)_x +u(r^*-bu)\nonumber\\
&=u_{xx}-\chi v_x u_x+ u(r^*-\chi\nu v-(b-\chi\mu)u),\quad x\in\R.
\end{align}
 By the comparison principle for parabolic equations, we have
$$
 u^*(t,x)=u(t+n,x;u^*(-n,\cdot))\le {w}(t+n)\quad \forall\, t\ge -n,\,\, x\in\R,\,\, n\in \N.
$$
It then follows that
$$
0<\underline{u}\le \overline{u}\le \frac{r^*}{b-\chi\mu}<\frac{r^*}{\chi\mu}.
$$

Now,  we may suppose that $0<r^*-\chi\mu(\overline{u}+\varepsilon) <r^*-\chi\mu(\underline{u}-\varepsilon)$ for $\varepsilon$ very small.
Let $\overline{w}(t)$ denote the solution of the initial value problem
\begin{equation*}\label{supSol}
\begin{cases}
\overline{w}_{t}=\overline{w}\left[r^*-\chi\mu(\underline{u}-\varepsilon) -(b-\chi\mu)\overline{w} \right] \\
\overline{w}(0)=\overline{u}.
\end{cases}
\end{equation*}
By the comparison principle for parabolic equations again, we have that
\begin{equation}\label{A10}
u^*(t,x)=u(t+n,x;u^*(-n,\cdot))\leq \overline{w}(t+n) \quad \forall \, t\ge -n ,\,\, x\in\R,\,\, n\in\N.
\end{equation}
Note that
\begin{equation*}
\overline{w}(t+n) \to \frac{r^*-\chi\mu(\underline{u}-\varepsilon)}{b-\chi\mu} \ \ \text{as} \ \ n\rightarrow \infty.
\end{equation*}
Combining this with inequality \eqref{A10}, we obtain that
\begin{equation*}
\overline{u}\leq \frac{r^*-\chi\mu(\underline{u}-\varepsilon)}{b-\chi\mu}\ \ \ \forall\ \varepsilon>0.
\end{equation*}
By letting $\varepsilon\rightarrow 0,$ we obtain
\begin{equation}
\label{revise-eq1}
\overline{u}\leq \frac{r^*-\chi\mu\underline{u}}{b-\chi\mu}.
\end{equation}

Similarly, let $\underline{w}(t)$ be solution of
\begin{equation*}\label{subSol}
\begin{cases}
\underline{w}_{t}=\underline{w}\left[r^*-\chi\mu(\overline{u}+\varepsilon) -(b-\chi\mu)\underline{w} \right] \\
\underline{w}(0)=\underline{u}.
\end{cases}
\end{equation*}
By  the  comparison principle for parabolic equations, we have
\begin{equation*}
u^*(t,x)=u(t+n,x;u^*(-n,\cdot))\geq \underline{w}(t+n) \quad \forall \, t\ge -n,\, x\in\R, \, n\in \N.
\end{equation*}
Same arguments as in above yield that
\[
\underline{u}\geq  \frac{r^*-\chi\mu(\overline{u}+\varepsilon)}{b-\chi\mu}\ \ \ \forall\ \varepsilon>0.
\]
By letting $\varepsilon\rightarrow 0,$  we have
\begin{equation}
\label{revise-eq2}
\underline{u}\geq  \frac{r^*-\chi\mu \overline{u}}{b-\chi\mu}.
\end{equation}

By \eqref{revise-eq1} and \eqref{revise-eq2}, we have
\begin{align*}
(b-2\chi\mu)\overline{u}&= (b-\chi\mu)\overline{u}-\chi\mu\overline{u}=(b-\chi\mu)\overline{u} +\left(r^* -\chi\mu\overline{u}\right) -r^*\nonumber\\
 &\leq  r^* -\chi\mu\underline{u} +(b-\chi\mu)\underline{u}  -r^*= (b-2\chi\mu)\underline{u}.
\end{align*}
Combining this with the fact that $\underline{u}\leq \overline{u}$ and $ (b-2\chi\mu)>0$ we obtain that
\begin{equation*}
\underline{u}= \overline{u}.
\end{equation*}
The lemma is thus proved.}
\end{proof}

{For every $u\in C^b_{\rm unif}(\R)$, let
\begin{equation}\label{psi-definition}
\Psi(x;u)=\mu\int_{0}^{\infty}\int_{\R}\frac{e^{-\nu s}e^{-\frac{|y-x|^2}{4s}}}{\sqrt{4\pi s}}u(y)dyds.
\end{equation}
It is well known that $\Psi(x;u)\in C^2_{\rm unif}(\R)$ and solves the elliptic equation
$$
\frac{d^2}{dx^2}\Psi(x;u)-\nu\Psi(x;u)+\mu u=0.
$$

\begin{lem}
\label{new-lm1}
\begin{equation}\label{psi-definition-eq2}
\Psi(x;u)=\frac{\mu}{2\sqrt{\nu}}\int_{\R}e^{-\sqrt{\nu}|x-y|}u(y)dy
\end{equation}
and
\begin{equation}\label{space-derivative-of-psi-1}
\frac{d}{dx}\Psi(x;u)=-\frac{\mu}{2}e^{-\sqrt{\nu }x}\int_{-\infty}^xe^{\sqrt{\nu}y}u(y)dy + \frac{\mu}{2}e^{\sqrt{\nu}x}\int_{x}^{\infty}e^{-\sqrt{\nu}y}u(y)dy.
\end{equation}
\end{lem}

\begin{proof}
The lemma is proved in \cite[Lemma 2.1]{SaShXu}.
\end{proof}

\begin{lem}\label{V-estimat}
Suppose that $b>\chi\mu$. For every $u\in C^b_{\rm unif}(\R)$, $0\leq u(x)\leq \frac{r^*}{b-\chi\mu}$, it holds that
$$
\Psi(x;u)\leq \frac{\mu r^*}{\nu(b-\chi\mu)}
\quad {\rm and}\quad
\frac{d}{dx}\Psi(x;u)\leq \frac{\mu r^*}{2\sqrt{\nu}(b-\chi\mu)}
$$
\end{lem}

\begin{proof}
It follows from a direct calculation.
\end{proof}

\section{Forced wave solutions in {\bf Case 1}}

In this section, we study the existence of forced wave solutions of \eqref{Keller-Segel-eq0} with $r(x)$ being as in {\bf Case 1}
and prove Theorem \ref{forced-wave-thm1}.  Throughout this section, we assume that
$r(x)$ is as in {\bf Case 1}.

\smallskip

We first present some lemmas.
Suppose that $b>\chi\mu$. Fix $r_1$ with $r(-\infty)<r_1<0$.
Let $x_1$ be given satisfying that $r(x)\leq r_1$ for any $x\leq x_1$.
Let $\theta_1$ be the positive root of the equation $\theta^2+c\theta+r_1=0$.
Define
\begin{equation}\label{U-Def}
U_1^{+}(x)=\min\{\frac{r^*}{b-\chi\mu}, \frac{r^*}{b-\chi\mu}e^{\theta_1(x-x_1)}\},
\end{equation}
and consider the set
\begin{equation}\label{E-set}
\mathcal{E}_1^+=\{u\in C^b_{\rm unif}(\R)\ :\ 0\leq u(x)\leq U_1^{+}(x), \ \forall\ x\in\R\}.
\end{equation}

For every $u\in\mathcal{E}_1^+$, consider the operator
\begin{equation}\label{A-equ}
\mathcal{A}_{u}(U)(x)=U_{xx}(x)+(c-\chi\Psi_x(x;u))U_{x}(x)+(r(x)-\chi\nu\Psi(x;u)-(b-\chi\mu)U(x))U(x),
\end{equation}
where $\Psi(x;u)$ is given by \eqref{psi-definition}.

\begin{lem}
\label{super-solu-lm}
Suppose that {$b\geq \frac{3}{2}\chi\mu$}. For every $u\in \mathcal{E}_1^+$, it holds that $\mathcal{A}_{u}(\frac{r^*}{b-\chi\mu})(x)\leq 0$ for $x\in\R$ and
$\mathcal{A}_{u}(\frac{r^*}{b-\chi\mu}e^{\theta_1(\cdot-x_1)})(x)\leq 0$ for $x\in (-\infty,x_1)$.
\end{lem}

\begin{proof}
Let $u\in \mathcal{E}_1^+$ be given. First, we have
\begin{align}
\mathcal{A}_{u}(\frac{r^*}{b-\chi\mu})(x)&=\frac{r^*}{b-\chi\mu}(r(x)-\chi\nu\Psi(x;u)-(b-\chi\mu)\frac{r^*}{b-\chi\mu})\cr
&=\frac{r^*}{b-\chi\mu}(r(x)-\chi\nu\Psi(x;u)-r^*)\cr
&\leq 0\quad \forall \, x\in\R.
\end{align}

Next, for  $x\in (-\infty, x_1)$, we have $r(x)\leq r_1$, and hence
\begin{align}\label{sup-eq0}
&\mathcal{A}_{u}(\frac{r^*}{b-\chi\mu}e^{\theta_1(\cdot-x_1)})(x)\cr &=\theta_1^2\frac{r^*}{b-\chi\mu}e^{\theta_1(x-x_1)}+(c-\chi\Psi_x(x;u))\theta_1\frac{r^*}{b-\chi\mu}e^{\theta_1(x-x_1)}\cr
&+\frac{r^*}{b-\chi\mu}e^{\theta_1(x-x_1)}(r(x)-\chi\nu\Psi(x;u)-(b-\chi\mu)\frac{r^*}{b-\chi\mu}e^{\theta_1(x-x_1)})
\cr
&=\frac{r^*}{b-\chi\mu}e^{\theta_1(x-x_1)}\left(\theta_1^2+c\theta_1+r(x)-\chi\theta_1\Psi_x(x;u)-\chi\nu\Psi(x;u)-r^*e^{\theta_1(x-x_1)}\right)\cr
&\leq \frac{r^*}{b-\chi\mu}e^{\theta_1(x-x_1)}\left(\theta_1^2+c\theta_1+r_1-\chi\theta_1\Psi_x(x;u)-\chi\nu\Psi(x;u)-r^*e^{\theta_1(x-x_1)}\right)\cr
&=\frac{r^*}{b-\chi\mu}e^{\theta_1(x-x_1)}\left(-\chi\theta_1\Psi_x(x;u)-\chi\nu\Psi(x;u)-r^*e^{\theta_1(x-x_1)}\right).
\end{align}
It then follows from Lemma \ref{new-lm1} and \eqref{sup-eq0} that
\begin{align*}
\mathcal{A}_u(\frac{r^*}{b-\chi\mu} e^{\theta_1(\cdot-x_1)})(x)
\leq\frac{r^*}{b-\chi\mu}e^{\theta_1(x-x_1)}\left(\frac{\chi\mu}{2}(\theta_1-\sqrt{\nu})e^{-\sqrt{\nu} x}\int_{-\infty}^xe^{\sqrt{\nu}y}u(y)dy
-r^*e^{\theta_1(x-x_1)} \right).
\end{align*}
{If $\theta_1\le \sqrt \nu$, we then have
$$
\mathcal{A}_u(\frac{r^*}{b-\chi\mu} e^{\theta_1(\cdot-x_1)})(x)\le 0\quad \forall\, x\in (-\infty,x_1).
$$
If $\theta_1>\sqrt \nu$, we then have}
\begin{align*}
&\mathcal{A}_u(\frac{r^*}{b-\chi\mu} e^{\theta_1(\cdot-x_1)})(x)\cr
&\leq \frac{r^*}{b-\chi\mu}e^{\theta_1(x-x_1)}\left(\frac{\chi\mu r^*}{2(b-\chi\mu) }(\theta_1-\sqrt{\nu})e^{-\sqrt{\nu} x}\int_{-\infty}^xe^{\sqrt{\nu}y}e^{\theta_1(y-x_1)}dy
-r^*e^{\theta_1(x-x_1)} \right)\cr
&=\frac{r^*{^2}}{b-\chi\mu}e^{2\theta_1(x-x_1)}\left(\frac{\chi\mu (\theta_1-\sqrt{\nu}) }{2(b-\chi\mu)(\theta_1+\sqrt{\nu}) }
-1 \right) \,\  \cr
&\leq 0\quad \forall\,  x<x_1.
\end{align*}
The lemma thus follows.
\end{proof}

Suppose $b>2\chi\mu$. For any $0<\varepsilon\ll 1$, define an ignition nonlinearity by
\begin{equation}\label{ignition-type}
f_{\varepsilon}(u)=\begin{cases}
u\left(r^*-\varepsilon -\frac{\chi\mu r^*}{b-\chi\mu}-(b-\chi\mu)u\right), \,\, &{\rm if} \,\, u\geq 0,\cr
0                     \,\, &{\rm if} -\varepsilon\leq u<0.
\end{cases}
\end{equation}
Consider the equation
\begin{equation}\label{ignition-eq}
u_t=u_{xx}+f_{\varepsilon}(u),\quad x\in \R.
\end{equation}
Equation \eqref{ignition-eq} has a decreasing traveling wave solution $\phi_{\varepsilon}(x-\tilde c_{\varepsilon}t)$ connecting {{$\frac{(r^*-\varepsilon)  (b-\chi\mu)-\chi\mu r^*}{(b-\chi\mu)^2}$ }}
and $-\varepsilon$ with speed $0<\tilde c_{\varepsilon}<2\sqrt{\frac{r^*(b-2\chi\mu)}{b-\chi\mu}}$ and $\lim_{\varepsilon\to 0^+} \tilde c_{\varepsilon}=2\sqrt{\frac{r^*(b-2\chi\mu)}{b-\chi\mu}}$ (see \cite{BNS}), that is, $(\phi_{\varepsilon},\tilde c_{\varepsilon})$ satisfies
\begin{equation}\label{ignit-TV0}
\begin{cases}
-\tilde c_{\varepsilon}\phi_{\varepsilon}^{'}=\phi_{\varepsilon}^{''}+f_{\varepsilon}(\phi_{\varepsilon}),\cr
\phi_{\varepsilon}(-\infty)={
\frac{(r^*-\varepsilon)(b-\chi\mu)-\chi\mu r^*}{(b-\chi\mu)^2} } , \,\, \phi_{\varepsilon}(\infty)=-\varepsilon, \,\, \phi_{\varepsilon}^{'}<0.
\end{cases}
\end{equation}
Let $\psi_{\varepsilon}(x)=\phi_{\varepsilon}(-x)$ for any $x\in\R$. It then follows from \eqref{ignit-TV0} that
\begin{equation}\label{ignit-TV}
\begin{cases}
\tilde c_{\varepsilon}\psi_{\varepsilon}^{'}=\psi_{\varepsilon}^{''}+f_{\varepsilon}(\psi_{\varepsilon}),\cr
\psi_{\varepsilon}(\infty)={
\frac{(r^*-\varepsilon)(b-\chi\mu)-\chi\mu r^*}{(b-\chi\mu)^2} }, \,\, \psi_{\varepsilon}(-\infty)=-\varepsilon, \,\, \psi_{\varepsilon}^{'}>0.
\end{cases}
\end{equation}

Without loss of generality, we can assume that $\psi_{\varepsilon}(x_0)=0$, $r(x)\ge r^*-\varepsilon$ if $x> x_0$, and $x_0>x_1$.
This can be realized by some appropriate translation of $\psi_{\varepsilon}(x)$ if necessary.

\begin{lem}\label{subsolution}
Suppose that {\bf (H1)} holds. For every  $u\in \mathcal{E}_1^+$ and $0<\varepsilon\ll 1$, $U_1^{-}(x)=\max\{\psi_{\varepsilon}(x),0\}$ satisfies that
$\mathcal{A}_u(U_1^{-}(\cdot))(x)\geq 0$ for any $x\not=x_0$.
Moreover, $U_1^{-}(x)<U_1^{+}(x)$ for all $x\in\R$.
\end{lem}

\begin{proof}
Let $\delta =c-\frac{\chi\mu r^*}{2\sqrt{\nu}(b-\chi\mu)}+2\sqrt{\frac{r^*(b-2\chi\mu)}{b-\chi\mu}}$.
By {\bf (H1)}, $\delta>0$.
Since $\lim_{\varepsilon\to 0^+} \tilde c_{\varepsilon}=2\sqrt{\frac{r^*(b-2\chi\mu)}{b-\chi\mu}}$, it then follows that for  $0<\varepsilon\ll 1$, we have $\tilde c_{\varepsilon}>2\sqrt{\frac{r^*(b-2\chi\mu)}{b-\chi\mu}}-\frac{\delta}{2}$.

Fix such $\varepsilon$. For every $u\in \mathcal{E}_1^+$. If $x>x_0$, $U_1^{-}(x)=\psi_{\varepsilon}(x)>0$ and $U_{1x}^-(x)>0$. By Lemma \ref{V-estimat} and \eqref{ignit-TV}, we have
\begin{align}
\mathcal{A}_u(\psi_{\varepsilon}(\cdot))(x)&=\psi_{\varepsilon}^{''}+(c-\chi\Psi_x(x;u))\psi_\varepsilon ^{'}+\psi_{\varepsilon}(r(x)-\chi\nu\Psi(x;u)-(b-\chi\mu)\psi_{\varepsilon})\cr
&\ge \psi_{\varepsilon}^{''}-\tilde c_{\varepsilon}\psi_{\varepsilon}^{'}+(c-\chi\Psi_x(x;u)+\tilde c_{\varepsilon})\psi_{\varepsilon}^{'}+\psi_{\varepsilon}(r^*-\varepsilon-\chi\nu\Psi(x;u)-(b-\chi\mu)\psi_{\varepsilon})\cr
&\geq \psi_{\varepsilon}^{''}-\tilde c_{\varepsilon}\psi_{\varepsilon}^{'}+(c-\chi\frac{\mu r^*}{2\sqrt{\nu}(b-\chi\mu)}+2\sqrt{\frac{r^*(b-2\chi\mu)}{b-\chi\mu}}-\frac{\delta}{2})\psi_{\varepsilon}^{'}\cr
&\,\, +\psi_{\varepsilon}(r^*-\varepsilon-\chi\mu\frac{r^*}{b-\chi\mu}-(b-\chi\mu)\psi_{\varepsilon})\cr
&=(c-\chi\frac{\mu r^*}{2\sqrt{\nu}(b-\chi\mu)}+2\sqrt{\frac{r^*(b-2\chi\mu)}{b-\chi\mu}}-\frac{\delta}{2})\psi_{\varepsilon}^{'}\cr
&\ge  0.
\end{align}
If $x<x_0$, $U_1^{-}(x)=0$. Then $\mathcal{A}_u(U_1^-)(x)=0$.

Since $x_1<x_0$, it is clear that $U_1^{-}(x)<U_1^{+}(x)$ for all  $x\in\R$. The lemma is thus proved.
\end{proof}


 Let
$$
\mathcal{E}_1=\{u\in C^b_{\rm unif}(\R)\ :\  U_1^-(x)\leq u(x)\leq U_1^{+}(x), \ \forall\ x\in\R\}.
$$
For any $u\in \mathcal{E}_1$,
let $U(t,x;u)$ be the solution of the following parabolic equation
\begin{equation}\label{U-def}
\begin{cases}
U_t=\mathcal{A}_{u}(U),\quad\ t>0,\,\,  x\in\R\cr
U(0,x;u)=U_1^{+}(x).
\end{cases}
\end{equation}

\begin{lem}\label{limit-of-U-lm-1}
Suppose that {\bf (H1)} holds.
For any $u\in \mathcal{E}_1$, ${{U_1^*(x;u)}}=\lim_{t\to\infty}U(t,x;u)$ exists  and satisfies the elliptic equation
\begin{equation}\label{U-elliptic-eq}
0=U_{xx} +(c -\chi\Psi_x(x;u))U_x+(r(x)-\chi\nu\Psi(x;u)-(b-\chi\mu)U)U \quad \forall\,\ x\in\R.
\end{equation}
Moreover, ${U_1^*(\cdot;u)}\in\mathcal{E}_1$.
\end{lem}

\begin{proof}
First, thanks to  Lemma \ref{super-solu-lm},  it follows from the comparison principle for parabolic equations that
$$
U(t_2,x;u)\leq U(t_1,x;u)\leq U_1^+(x),\quad \forall \,\, x\in\R,\ 0<t_1<t_2, \ u\in\mathcal{E}_1.
$$
Thus the function
\begin{equation}\label{U^*-main-def}
U_1^*(x;u)=\lim_{t\to\infty}U(t,x;u),\quad \forall\ u\in\mathcal{E}_1
\end{equation}
is well defined.
Moreover, by a priori  estimates for parabolic equations, it is not difficult to see that {$U_1^*(\cdot;u)\in \mathcal{E}_1^+$} and $U_1^*(x;u)$
satisfies \eqref{U-elliptic-eq}.

Next, it follows from Lemma \ref{subsolution} and the comparison principle for parabolic equations that
\begin{equation}\label{lower-bound of U}
U_1^{-}(x)\leq U(t,x;u),\quad \quad \forall \,\,x\in\R,\ t>0, \ u\in\mathcal{E}_1.
\end{equation}
Hence,
\begin{equation}
U_1^{-}(x)\leq { U_1^*(x;u)},\quad \forall\ x\in\R,\ \forall\ u\in\mathcal{E}_1.
\end{equation}
 Therefore, $U_1^*(\cdot;u)\in\mathcal{E}_1$.
The lemma is thus proved.
\end{proof}

\begin{lem}
\label{u-star-star-lm1}
Suppose that {\bf (H1)} holds. For any $u\in\mathcal{E}_1$, suppose that $U_{1*}(x;u)$ is also a solution of \eqref{U-elliptic-eq} in $\mathcal{E}_1$.
Then
\begin{equation}
\label{new-add-eq1}
\lim_{x\to\infty} \frac{U_{1*}(x;u)}{U_1^*(x;u)}=1.
\end{equation}
\end{lem}

\begin{proof}
First of all, {by Lemma \ref{V-estimat}, $\sup_{x\in\R}|\Psi(x;u)|<\infty$ and $\sup_{x\in\R}|\frac{\p}{\p x}\Psi(x;u)|<\infty$. By $\Psi_{xx}(x;u)=\nu \Psi(x;u)-\mu u$, we have $\sup_{x\in\R}|\Psi_{xx}(x;u)|<\infty$.} This implies that
for any $\{x_n\}_{n=1}^{\infty}\subset\R$, there is $\{x_{n_k}\}_{k=1}^{\infty}\subset \{x_n\}_{n=1}^{\infty}$ such that $\lim_{k\to\infty} \Psi(x+x_{n_k};u)$ and $\lim_{k\to\infty}\Psi_x(x+x_{n_k};u)$  exist locally uniformly on $\R$.

Next, note that $\frac{U_{1*}(x;u)}{U_1^*(x;u)}\le 1$ for all $x\in\R$. It then suffices to prove that
$\liminf_{x\to \infty}\frac{U_{1*}(x;u)}{U_1^*(x;u)}\ge 1$. Assume by contraction that
$$\liminf_{x\to \infty}\frac{U_{1*}(x;u)}{U_1^*(x;u)}<1.
$$
Then there are $0<\delta<1$ and  $x_n\to\infty$ such that
$$
\frac{U_{1*}(x_n;u)}{U_1^*(x_n;u)}\le 1-\delta \quad \forall\, \, n=1,2,\cdots.
$$
Let
$$
U_{n,1*}(x;u)=U_{1*}(x+x_n;u),\quad U_{n,1}^*(x;u)=U_1^*(x+x_n;u),\quad \Psi_n(x;u)=\Psi(x+x_n;u).
$$
Without loss of generality, we may assume that there are $\underline{U}_*(x;u)$, $\overline{U}^*(x;u)$, and
$\Psi^*(x;u)$ such that
$$
\lim_{n\to \infty} U_{n,1*}(x;u)=\underline{U}_*(x;u),\quad \lim_{n\to\infty}U_{n,1}^*(x;u)=\overline{U}^*(x;u),\,\, {\rm and}\,\,
\lim_{n\to\infty}\Psi_n(x;u)=\Psi^*(x;u)
$$
locally uniformly on $\R$. This implies that both $\underline{U}_*(x;u)$ and $\overline{U}^*(x;u)$ are solutions of
\begin{equation}\label{U-star-elliptic-eq}
0=U_{xx} +(c -\chi\Psi^*_x(x;u))U_x+(r^*-\chi\nu\Psi^*(x;u)-(b-\chi\mu)U)U.
\end{equation}

{We now claim that $\underline{U}_*(x;u)\equiv\overline{U}^*(x;u)$. Indeed, note that
$$
0<\inf_{x\in\R}\underline{U}_*(x;u)\le \inf_{x\in\R} \overline{U}^*(x;u),\,\, {\rm and}\,\, \sup_{x\in\R}\underline{U}_*(x;u)\le \sup_{x\in\R}\overline{U}^*(x;u)<\infty.
$$
This implies that the following set is not empty,
 $$
 \{\gamma\geq1\,|\, \frac{1}{\gamma} \overline{U}^*(x;u)\le \underline{U}_*(x;u)\le \gamma \overline{U}^*(x;u)\quad \forall\,\,x\in\R\}.
 $$
Hence we can define
 $$
 \rho(\underline{U}_*,\overline{U}^*)=\inf\{\ln \gamma \,|\,  \frac{1}{\gamma} \overline{U}^*(x;u)\le \underline{U}_*(x;u)\le \gamma \overline{U}^*(x;u) \quad \forall\,\,x\in\R \}.
$$
Note that $\rho(\underline{U}_*,\overline{U}^*)$ is the so called {\it part metric} between $\underline{U}_*$ and $\overline{U}^*$.
Assume that { $\rho(\underline{U}_*,\overline{U}^*)>0$}. Then by the arguments of \cite[Proposition 3.4]{KoSh}, there is
$\delta_0>0$ such that
$$
\rho(\underline{U}_*,\overline{U}^*)\le \rho(\underline{U}_*,\overline{U}^*)-\delta_0,
$$
which is a contradiction. Hence {$\rho(\underline{U}_*,\overline{U}^*)=0$}}, and then $\underline{U}_*(x;u)=\overline{U}^*(x;u)$ for all $x\in\R$.
But, by the assumption,
$$
\underline{U}_*(0;u)\not=\overline{U}^*(0;u),
$$
which is a contradiction.
The lemma thus follows.
\end{proof}

\begin{lem}
\label{u-star-star-lm2}
Suppose that {\bf (H1)} holds.
For any $u\in\mathcal{E}_1$, $U_1^*(x;u)$ is the unique positive solution of \eqref{U-elliptic-eq} in $\mathcal{E}_1$.
\end{lem}

\begin{proof}
Suppose that $U_{1*}(x;u)$ is any positive solution of \eqref{U-elliptic-eq} in $\mathcal{E}_1$. It suffices to prove that
$U_{1*}(x;u)\equiv U_1^*(x;u)$.

For any $\epsilon>0$, let
$$
K_\epsilon=\{k\ge 1\,|\, kU_{1*}(x;u)\ge U_1^*(x;u)-\epsilon\quad \forall \, x\in\R\}.
$$
By Lemma \ref{u-star-star-lm1} and the fact $\lim_{x\to -\infty}U_{1*}(x;u)=\lim_{x\to -\infty}U_1^*(x;u)=0$, $K_\epsilon\not =\emptyset$.
Let
$$
k_\epsilon=\inf K_\epsilon.
$$
Then $k_\epsilon\ge 1$ and
\begin{equation}
\label{new-add-eq0}
k_\epsilon U_{1*}(x;u)\ge U_1^*(x;u)-\epsilon\quad \forall\, x\in\R.
\end{equation}
For any $0<\epsilon_1<\epsilon_2$, since
$$
k_{\epsilon_1} U_{1*}(x;u)\ge U_1^*(x;u)-\epsilon_1>U_1^*(x;u)-\epsilon_2\quad \forall\, x\in\R,
$$
it then follows that $k_{\epsilon_1}\geq k_{\epsilon_2}$. Thus, $k_\epsilon$ is nonincreasing in $\epsilon>0$.
If $k_{\epsilon}=1$ for any $\epsilon>0$, clearly, we have that ${U_1^*(x;u)}\equiv {U_{1*}(x;u)}.$

Assume that there exists $\epsilon_0>0$ such that $k_{\epsilon_0}>1$. Then
\begin{equation}
\label{new-add-eq2}
k_{\epsilon}\geq k_{\epsilon_0}>1 \quad \text{for any}\quad  0<\epsilon\leq \epsilon_0.
\end{equation}
For any given $0<\epsilon\leq \epsilon_0$, since $k_{\epsilon}>1$, there exists $\delta>0$, such that $k_{\epsilon}-\delta>1$.
By $$\lim_{x\to\infty} \frac{U_{1*}(x;u)}{U_1^*(x;u)}=1,$$
we have that for such given $\epsilon>0$,
$$
\frac{U_{1*}(x;u)}{U_1^*(x;u)}\geq 1-\frac{\epsilon}{U_1^*(x;u)} \quad x\gg 1.
$$
Hence,
\begin{equation}\label{x>>1}
(k_{\epsilon}-\delta)U_{1*}(x;u)\geq U_1^*(x;u)-\epsilon \quad x\gg 1.
\end{equation}
Since
$\lim_{x\to-\infty}\frac{U_1^*(x;u)-\epsilon}{U_{1*}(x;u)}=-\infty$,
it then clear that
\begin{equation}\label{x<<-1}
(k_{\epsilon}-\delta)U_{1*}(x;u)\geq U_1^*(x;u)-\epsilon \quad x\ll -1.
\end{equation}
It then follows from \eqref{x>>1}, \eqref{x<<-1} and the definition of $k_{\epsilon}$ that
there is $x_\epsilon\in\R$ such that
\begin{equation}
\label{new-add-eq3}
k_\epsilon U_{1*}(x_\epsilon;u)=U_1^*(x_\epsilon;u)-\epsilon.
\end{equation}
 By \eqref{new-add-eq1}, $x_\epsilon$ is bounded above.

We claim that $x_\epsilon$ is bounded from below.  In fact, at $x_\epsilon$, we have
$$
\p_{xx}(k_\epsilon U_{1*}(x_\epsilon;u)-U_1^*(x_\epsilon;u))\ge 0,\quad \p_x (k_\epsilon U_{1*}(x_\epsilon;u)-U_1^*(x_\epsilon;u))=0,
$$
and hence
\begin{align*}
0&\ge  k_\epsilon U_{1*}(x_\epsilon)(r(x_\epsilon )-\chi\nu\Psi(x_\epsilon;u)-(b-\chi\mu)U_{1*}(x_\epsilon))-U_1^*(x_\epsilon)(r(x_\epsilon )-\chi\nu\Psi(x_\epsilon;u)-(b-\chi\mu)U_1^*(x_\epsilon))\\
&\ge k_\epsilon U_{1*}(x_\epsilon)(r(x_\epsilon )-\chi\nu\Psi(x_\epsilon;u)-(b-\chi\mu)U_1^*(x_\epsilon))-U_1^*(x_\epsilon)(r(x_\epsilon )-\chi\nu\Psi(x_\epsilon;u)-(b-\chi\mu)U_1^*(x_\epsilon))\\
&=-\epsilon (r(x_\epsilon )-\chi\nu\Psi(x_\epsilon;u)-(b-\chi\mu)U_1^*(x_\epsilon)).
\end{align*}
This implies that $r(x_\epsilon)\ge 0$ and hence $x_\epsilon$ is bounded from below.

Therefore,  $x_{\epsilon}$ is bounded both from below and above.
By \eqref{new-add-eq3},
$$
k_\epsilon=\frac{U_1^*(x_\epsilon;u)-\epsilon}{U_{1*}(x_\epsilon;u)}.
$$
Hence $k_\epsilon$ is bounded,  and there is $\epsilon_n\to 0$ such that
$x_{\epsilon_n}\to x^*$ and $k_{\epsilon_n}\to k^*(\ge k_{\epsilon_0}>1)$ as $n\to\infty$.
This together with \eqref{new-add-eq3} implies that
$k^* U_{1*}(x^*;u)=U_1^*(x^*;u).$
By \eqref{new-add-eq0},
$$
k^* U_{1*}(x;u)\ge U_1^*(x;u)\quad \forall\, x\in \R.
$$
Since $k^*>1$, by \eqref{new-add-eq1},  $k^* U_{1*}(x;u)\not \equiv U_1^*(x;u)$.
Then by the comparison principle for parabolic equations, we must have
$$
U_1^*(x;u)<k^*U_{1*}(x;u) \quad \forall\,\ x\in\R,
$$
which is a contraction. Therefore, $k_\epsilon= 1$ for $0<\epsilon\ll 1$
and then
$
U_{1*}(x;u)\equiv U_1^*(x;u).
$
\end{proof}

We now prove Theorem \ref{forced-wave-thm1}.

\begin{proof}[Proof of Theorem \ref{forced-wave-thm1}]
Consider the mapping {$U_1^*(\cdot;\cdot): \mathcal{E}_1\ni u \mapsto U_1^*(x;u)\in\mathcal{E}_1$} as defined by \eqref{U^*-main-def}. It follows from the arguments of the proof of \cite[Theorem  3.1]{SaSh2} {and the Lemma
\ref{u-star-star-lm2}} that this function is continuous
and compact in the compact open topology. Hence it has a fixed point $u^*$ by the Schauder's fixed point Theorem. Taking $v^*(x)=\Psi(x;u^*)$, we have from \eqref{U-elliptic-eq}, that $(u(t,x),v(t,x))=(u^*(x-ct),v^*(x-ct))$ is an entire solution of \eqref{Keller-Segel-eq0}. Moreover, since  $U_1^-(x)\leq u^*(x)\leq U_1^{+}(x)$, it follows that
 $ \lim_{x\to-\infty}u^*(x)=0.$

In the following we show that
\begin{equation}\label{kk-0}
\lim_{x\to\infty}u^*(x)=\frac{r^*}{b}.
\end{equation}
Suppose on the contrary that this is false. Then, there is { a constant $\delta>0$ and a sequence} $\{x_n\}_{n\in\N}$  such that
$x_{n}\to\infty$ and
\begin{equation}\label{kkk-1}
|u^*(x_n)-\frac{r^*}{b}|\geq \delta, \,\, \forall\,\, n\geq 1.
\end{equation}
Consider the sequence of functions
$$
u^{n}(t,x)=u(t,x+x_n)\quad \text{and}\quad v^{n}(t,x)=v(t,x+x_n).
$$
By a priori estimate for parabolic equations, without loss of generality, we may assume that there is $(u^{**}(t,x),v^{**}(t,x))\in C^{1,2}(\R\times\R)$ such that $(u^{n},v^{n})(t,x)\to$  $(u^{**}(t,x),v^{**}(t,x))$ {locally uniformly in $C^{1,2}(\R\times\R)$} as $n\to\infty$. Furthermore, the function $(u^{**}(t,x),v^{**}(t,x))$ is an entire solution of the following equation
\begin{equation}\label{final-eq}
\begin{cases}
u_t=u_{xx}-\chi(uv_x)_x +u(r^*-bu),\quad x\in\R\cr
0=v_{xx}- \nu v+\mu u,\quad x\in\R,
\end{cases}
\end{equation}
Note that
$$
u^{**}(t,x)=\lim_{n\to\infty}u^{n}(t,x)\geq \lim_{n\to\infty}U_1^{-}(x+x_n-ct)=U_1^-(\infty)>0,\quad \forall\ x\in\R, \ t\in\R.
$$
So $\inf_{(t,x)\in\R\times\R}u^{**}(t,x)>0$.

Therefore, since $\chi\mu<\frac{b}{2}$, it follows from  Lemma \ref{lem-001-4} that $u^{**}(t,x)=\frac{r^*}{b}$ for every $(t,x)\in \R\times\R$. In particular, $\frac{r^*}{b}= u^{**}(0,0)=\lim_{n\to\infty}u^{n}(0,0) =\lim_{n\to\infty}u(0,x_n)=\lim_{n\to\infty}u^*(x_n) $, which contradicts to \eqref{kkk-1}. Therefore, \eqref{kk-0} must hold.
\end{proof}

\section{Forced wave solutions in {\bf Case 2}}

In this section, we study the existence of forced wave solutions of \eqref{Keller-Segel-eq0} with $r(x)$ being as in {\bf Case 2}
and prove Theorem \ref{forced-wave-thm2}. We first present some lemmas. Throughout this section, we assume that
 $r(x)$ is as in {\bf Case 2} and {\bf (H2)} holds.

 Fix a $\bar r$ with $\max\{r(-\infty),r(\infty)\}<\bar r<0$. Choose $\bar x$ such that  the inequality $r(x)<\bar r$ holds for all $x<\bar x$.
Let $\bar \theta$ be the positive solution of $\bar \theta^2+c\bar \theta+\bar r=0$.
Choose  $\tilde x$ such  that  the inequality $r(x)<\bar r$ holds for all $x>\tilde x$.
Let $\tilde \theta$ be the positive solution of $\tilde \theta^2-c\tilde \theta+\bar  r=0$.

\smallskip

Define
$$
U_2^+(x)=\begin{cases}
\frac{r^*}{b-\chi\mu}e^{\bar\theta(x-\bar x)} \quad  &{\rm if}\,\,  x<\bar x,\cr
\frac{r^*}{b-\chi\mu}                     \quad &{\rm if}\,\, \bar x\leq x\leq \tilde x,\cr
\frac{r^*}{b-\chi\mu}e^{-\tilde\theta(x-\tilde x)} \quad  &{\rm if}\,\,  x>\tilde x,\cr
\end{cases}
$$
and consider the set
$$
\mathcal{E}_2^+=\{u\in C^b_{\rm unif}(\R)\ :\ 0\leq u(x)\leq U_2^{+}(x), \ \forall\ x\in\R\}.
$$
For every $u\in\mathcal{E}_2^+$, consider the operator
$$
\mathcal{A}_{u}(U)(x)=U_{xx}(x)+(c-\chi\Psi_x(x;u))U_{x}(x)+(r(x)-\chi\nu\Psi(x;u)-(b-\chi\mu)U(x))U(x),
$$
where $\Psi(x;u)$ is given by \eqref{psi-definition}.

\begin{lem}
\label{super-solu-lm2}
Suppose that $b\ge \frac{3\chi\mu}{2}$. For every $u\in \mathcal{E}_2^+$, it holds that
$\mathcal{A}_{u}(\frac{r^*}{b-\chi\mu}e^{\bar\theta(\cdot-\bar x)})(x)\leq 0$ for $x\in (-\infty,\bar x)$,
$\mathcal{A}_{u}(\frac{r^*}{b-\chi\mu})(x)\leq 0$ for $x\in\R$ and
$\mathcal{A}_{u}(\frac{r^*}{b-\chi\mu}e^{-\tilde\theta(\cdot-\tilde x)})(x)\leq 0$ for $x\in (\tilde x,\infty)$.
\end{lem}
\begin{proof}
It can be proved by the similar arguments as those used in the proof of Lemma \ref{super-solu-lm}.
\end{proof}

Consider
\begin{equation}
\label{aux-new-new-new-eq01}
\begin{cases}
u_t=u_{xx}+c u_x-A(t,x)u_x+u (r(x) -B(t,x)-(\bar b-\chi\mu)u), \quad -L<x<L\cr
u(t,-L)=u(t,L)=0,
\end{cases}
\end{equation}
where $A(t,x)$ and $B(t,x)$ are both globally H\"older continuous in $t\in\R$ and
$x\in [-L,L]$ with H\"older exponent $0<\alpha<1$ and $\|A(\cdot,\cdot)\|_\infty<\infty$, $\|B(\cdot,\cdot)\|_\infty<\infty$.

\smallskip


\begin{lem}
\label{new-new-lm1}
Suppose that {\bf (H2)} holds.  Then  there are $L^*>0$ and $\eta=\eta(r(\cdot),c)>0$  such that for any $L\ge L^*$, any $A(\cdot,\cdot)$, $B(\cdot,\cdot)$ with
 $\|A(\cdot,\cdot)\|_\infty<\eta$, $\|B(\cdot,\cdot)\|_\infty<\eta$, and any $\bar b>\chi\mu$,
 \eqref{aux-new-new-new-eq01} has a  unique positive bounded entire solution
$u^*(t,x; \bar b, A(\cdot,\cdot), B(\cdot,\cdot))$ with
\begin{equation}
\label{aux-new-new-new-eq2}
\inf_{t\in\R, -L+\delta\le x\le L-\delta, \|A(\cdot,\cdot)\|_\infty<\eta, \|B(\cdot,\cdot)\|_\infty<\eta} u^*(t,x; \bar b, A(\cdot,\cdot), B(\cdot,\cdot))>0\quad \forall\,\, 0<\delta<L.
\end{equation}
\end{lem}
\begin{proof}
It can be proved by the similar arguments as those used in the proof of \cite[Lemma 3.2]{ShXu}.
\end{proof}

\smallskip

Let $L=L^*$ and $\eta$ be fixed. For every $u\in\mathcal{E}_2^+$, let $A(t,x)=\chi\Psi_{x}(x;u)$, $B(t,x)=\chi\nu\Psi(x;u)$. By Lemma \ref{V-estimat}, $\Psi_{x}(x;u)$ and $\Psi_{xx}(x;u)$ are both bounded for any $u\in\mathcal{E}_2^+$, we then have $A(t,x)$ and $B(t,x)$ are both globally H\"older continuous in $x\in [-L,L]$.

  In the following, we assume that
  \begin{equation}
  \label{chi-cond-eq2}
  0<\chi<\chi_0=\chi_0(r(\cdot),c):=\min\{\frac{2\sqrt{\nu}\eta b}{\mu r^*+2\sqrt{\nu}\eta\mu}, \frac{b\eta}{\mu r^*+\mu\eta}\}.
  \end{equation}
  Then by Lemma \ref{V-estimat}, we have
$$
\|A(\cdot,\cdot)\|_\infty\leq \frac{\chi\mu r^*}{2\sqrt{\nu}(b-\chi\mu)}<\eta
\quad {\rm and}\quad
\|B(\cdot,\cdot)\|_\infty\leq \frac{\chi\mu r^*}{b-\chi\mu}<\eta.
$$
It then follows from Lemma \ref{new-new-lm1} that \eqref{aux-new-new-new-eq01} with $A(t,x)=\chi\Psi_{x}(x;u)$ and $B(t,x)=\chi\nu\Psi(x;u)$ has a unique positive bounded entire solution $u^*(t,x; \bar b, u){:=u^*(t,x;  \bar b, \chi\Psi_{x}(\cdot;u), \chi\nu\Psi(\cdot;u))}$ with
\begin{equation}
\label{aux-new-new-new-eq3}
\inf_{t\in\R, -L+\delta\le x\le L-\delta} u^*(t,x; \bar b, u)>0\quad \forall\,\, 0<\delta<L.
\end{equation}
{Note that, by the comparison principle for parabolic equations,
$$
u^*(t,x;\bar b,u)\le \frac{r^*}{\bar b-\chi\mu}\quad \forall\, t\in\R,\,\, -L\le x\le L.
$$}
Fix $\bar b \gg b$ such that $u^*(t,x;\bar b, u)<U_2^+(x)$ for any $-L\leq x\leq L$, any $t\in\R$, {any $u\in
\mathcal{E}_2^+$}.
By the proof of \cite[Lemma 3.2]{ShXu}, we have that
\begin{equation}
\label{aux-new-new-new-eq4-1}
\inf_{t\in\R, -L+\delta\le x\le L-\delta, u\in\mathcal{E}_2^+} u^*(t,x; \bar b, u)>0\quad \forall\,\, 0<\delta<L.
\end{equation}

Define
$$
U_2^{-}(x)=\begin{cases}
\inf_{t\in\R, u\in\mathcal{E}_2^+}u^*(t,x; \bar b,  u) \quad  &{\rm if}\,\, -L< x<L,\cr
0 \quad  &{\rm if}\,\,  x\ge L,\, x\le -L\cr
\end{cases}
$$
Then, $U_2^{-}(x)\not\equiv 0$  and $\inf_{-L+\delta\le x\le L-\delta}U_2^{-}(x)>0$, and  $U_2^{-}(x)<U_2^{+}(x)$ for any $x\in\R$.

\begin{lem}
\label{new-add-lm2}
For any  $u\in\mathcal{E}_2^+$,
$$
U_2^-(x)<U(t,x;u) \quad \forall \,\ x\in\R,\ t>0,
$$
where $U(t,x;u)$ is the solution of the following parabolic equation
\begin{equation}\label{U-def3}
\begin{cases}
U_t=\mathcal{A}_{u}(U),\quad\ t>0,\,\,  x\in\R\cr
U(0,x;u)=U_2^{+}(x), \,\,  x\in\R.
\end{cases}
\end{equation}
\end{lem}

\begin{proof} Observe that
\begin{align}
&u^*_{xx} +(c -\chi\Psi_x(x;u))u^*_x+(r(x)-\chi\nu\Psi(x;u)-(b-\chi\mu)u^*)u^*-u^*_t\cr
&=(\bar b-b)u{^*}^2>0 \quad \forall\, -L<x<L.
\end{align}
Hence, by the comparison principle for parabolic equations, we get that
\begin{equation}\label{lower-bound of U}
u^*(t,x;\bar b, u)< U(t,x;u),\quad \quad \forall \,\,-L<x<L,\ t>0, \ u\in\mathcal{E}_2^+.
\end{equation}
This implies that
$$
\inf_{t\in\R, u\in\mathcal{E}_2^+}u^*(t,x;\bar b, u)<U(t,x;u),\quad \forall \,\,-L<x<L,\ t>0, \ u\in\mathcal{E}_2^+.
$$
The lemma then follows.
\end{proof}

Note that, by Lemma \ref{super-solu-lm2} and the comparison principle for parabolic equations,
$$
U(t_2,x;u)\leq U(t_1,x;u)\leq U_2^+(x),\quad \forall \,\, x\in\R,\ 0<t_1<t_2, \ u\in\mathcal{E}_2^{+}.
$$
Thus the function
\begin{equation}\label{U-main-def2}
U_2^*(x;u)=\lim_{t\to\infty}U(t,x;u),\quad \forall\ u\in\mathcal{E}_2^{+}
\end{equation}
is well defined, and
\begin{equation}\label{U-upper}
U_2^-(x) \leq U_2^*(x;u)\leq U_2^+(x),\quad \forall \,\, x\in\R,\ u\in\mathcal{E}_2^{+}.
\end{equation}

Let
$$
\mathcal{E}_2=\{u\in C^b_{\rm unif}(\R)\ :\  U_2^-(x)\leq u(x)\leq U_2^{+}(x), \ \forall\ x\in\R\}.
$$
For any $u\in \mathcal{E}_2$, it follows from \eqref{U-upper} that $U_2^*(\cdot;u)\in \mathcal{E}_2$.
Moreover, by a priori estimates for parabolic equation, we have that $U_2^*(x;u)$
satisfies
\begin{equation}\label{U-elliptic-eq2}
0=U_{xx} +(c -\chi\Psi_x(x;u))U_x+(r(x)-\chi\nu\Psi(x;u)-(b-\chi\mu)U)U\quad \forall\,\ x\in\R.
\end{equation}
{Since $U_2^-(x)\ge 0$ for any $x\in\R$ and $U_2^-(x)\not\equiv 0$ for any $x\in\R$, it follows from the comparison principle for parabolic equations that
\begin{equation}\label{U-positive}
U_2^*(x;u)>0 \quad \forall\,\ x\in\R, \,\ u\in\mathcal{E}_2^+.
\end{equation}}

\begin{lem}
\label{uniqueness-of-U-lm2}
For any given $u\in\mathcal{E}_2$,
\eqref{U-elliptic-eq2} has a unique solution $U_2^*(\cdot;u)\in\mathcal{E}_2$.
\end{lem}

\begin{proof}
Let $U_1(x;u)$, $U_2(x;u)$ be two solutions of \eqref{U-elliptic-eq2} in $\mathcal{E}_2$. Note that $U_i(x;u)>0$ for $x\in\R$ and every $i=1,2$. For any $\epsilon>0$, let
$$
K_\epsilon=\{k\ge 1\,|\, kU_2(x;u)\ge U_1(x;u)-\epsilon\quad \forall \, x\in\R\}.
$$
$K_\epsilon$ is not empty because
\begin{equation}\label{quotient-lim}
\lim_{x\to\pm \infty}\frac{U_1(x;u)-\epsilon}{U_2(x;u)}=-\infty.
\end{equation}
Let
$k_\epsilon=\inf K_\epsilon.
$
Then $k_\epsilon\ge 1$ and
$$
k_\epsilon U_2(x;u)\ge U_1(x;u)-\epsilon\quad \forall\, x\in\R.
$$
Note that $k_\epsilon$ is nonincreasing in $\epsilon>0$. Following the similar arguments as those used in the proof of Lemma \ref{u-star-star-lm2}, we have $k_\epsilon=1$ for $0<\epsilon\ll 1$. Therefore,
$${
U_2(x;u)\ge U_1(x,u)\quad \forall\, x\in\R. }
$$
Similarly, we have
$${
U_1(x;u)\ge U_2(x,u)\quad \forall\, x\in\R.}
$$
The lemma thus follows.
\end{proof}

We now prove Theorem \ref{forced-wave-thm2}.

\begin{proof}[Proof of Theorem \ref{forced-wave-thm2}]
Consider the mapping  $U_2^*(\cdot;\cdot): \mathcal{E}_2\ni u \mapsto U_2^*(x;u)\in\mathcal{E}_2$ as defined by \eqref{U-main-def2}. It follows from the arguments of  \cite[Theorem  3.1 ]{SaSh2} {and  Lemma
\ref{uniqueness-of-U-lm2} }that this function is continuous
and compact in the compact open topology. Hence it has a fixed point $u^*$ by the Schauder's fixed point Theorem. Taking $v^*(x)=\Psi(x;u^*)$, we have from \eqref{U-elliptic-eq2}, that $(u(t,x),v(t,x))=(u^*(x-ct),v^*(x-ct))$ is an entire solution of \eqref{Keller-Segel-eq0}. Moreover, {by \eqref{U-positive}, $u^*(x)>0$ for all $x\in\R$.} Since  $U_2^{-}(x) \leq u^*(x)\leq U_2^{+}(x)$,  it follows that
 $\lim_{x\to \pm\infty}u^*(x)=0.$
 The theorem is thus proved.
\end{proof}

\section{Numerical Simulations}

In this section, we present some numerical investigation on the existence of forced wave solutions. All the numerical simulations were conducted using programming software MATLAB.

\subsection{Numerical simulations  in Case 1}

In this subsection, we present some numerical simulations in Case 1 by the finite difference method. It should be pointed out that
the authors in
  \cite{YaBa} provided some numerical study for the vanishing and spreading dynamics of chemotaxis systems with logistic source and a free boundary by the finite difference method.

First, we describe the numerical scheme. To numerically investigate the existence of forced wave solutions in Case 1, we use the finite difference method to  compute the solution of
\begin{equation}\label{wave-cut-off-eq1-1}
\begin{cases}
u_t= u_{xx}+cu_x- (\chi u  v_x)_x+u(r(x)-bu),\quad -L<x<L\cr
0 =v_{xx}-  \nu v +\mu u,\quad -L<x<L\cr
u(0,x)=u_0(x), \quad -L\leq x\leq L\cr
u(t,-L)=v(t,-L)=0\cr
\frac{\partial u}{\partial x}(t,L)=\frac{\partial v}{\partial x}(t,L)=0
\end{cases}
\end{equation}
for reasonable  large $L>1$,
where $u_0(x)$ is piece-wise linear,
$
u_0(x)=\begin{cases}
0 \quad  &{\rm if}\,\,  x\leq -1,\cr
\frac{r^*}{2b}x+\frac{r^*}{2b}  \quad  &{\rm if}\,\,  -1<x<1,\cr
\frac{r^*}{b}                   \quad &{\rm if}\,\,  x\geq 1.
\end{cases}
$

\smallskip

Let $(u_L(t,x;u_0),v_L(t,x;u_0))$ be the solution of \eqref{wave-cut-off-eq1-1} with $u_L(0,x;u_0)=u_0(x)$.
Observe that if $(u_L(x),v_L(x)):=\lim_{t\to \infty}(u_L(t,x;u_0),v_L(t,x;u_0))$ exists, then $(u_L(x),v_L(x))$ is a stationary solution of
\eqref{wave-cut-off-eq1-1}. If $(u_\infty(x),v_\infty(x)):=\lim_{L\to\infty} (u_L(x),v_L(x))$ exists, and $u_\infty(-\infty)=0$ and $u_\infty(\infty)=\frac{r^*}{b}$, then $(u_\infty(x),v_\infty(x))$ is a forced wave solution of \eqref{Keller-Segel-eq0}
connecting $(\frac{r^*}{b},\frac{\mu r^*}{\nu b})$ and $(0,0)$.
We will then compute the numerical solution of \eqref{wave-cut-off-eq1-1}  on a reasonable large time interval $[0,T]$ for several choices of $L$.

\smallskip

Note that, by the second equation in \eqref{wave-cut-off-eq1-1},
$v_{xx}=\nu v-\mu u.$
Hence the first equation in \eqref{wave-cut-off-eq1-1} can be written as
\begin{equation}\label{wave-cut-off-eq1-2-0}
u_t=u_{xx}+(c-\chi v_x)u_x +u(r(x)-\chi \nu v-(b-\chi\mu)u),\quad -L<x<L.
\end{equation}

To find the numerical solution of  \eqref{wave-cut-off-eq1-1} on the interval $[0,T]$,
we divide the space interval  $[-L,L]$ into $M$ subintervals with equal length and divide the time interval $[0,T]$ into $N$ subintervals with equal length. Then the space step size is $h=\frac{2L}{M}$ and the time step size is $\tau=\frac{T}{N}$. For simplicity, we denote the approximate value of $u(t_j, x_i)$ by $u(j,i)$, $r(x_i)$ by $r(i)$ and $v(t_j, x_i)$ by $v(j,i)$ with $t_{j}=(j-1)\tau$, $1 \leq j\leq N+1$, and $x_{i}=-L+(i-1)h$, $1 \leq i\leq M+1$.

Using  the central approximation for the second spatial derivative $v_{xx}(t_j,x_i)$,
$$
v_{xx}(t_{j}, x_{i}) \approx \frac{v(j,i-1)-2v(j,i)+v(j,i+1)}{h^2},
$$
the second equation in \eqref{wave-cut-off-eq1-1} can be discretized as
\begin{equation}\label{v-discre}
\frac{v(j,i-1)-2v(j,i)+v(j,i+1)}{h^2}-\nu v(j,i)+\mu u(j,i)=0, \quad 2\leq i\leq M.
\end{equation}
We use the backward approximation for the spatial derivative $\frac{\partial v}{\partial x}(t_j,x_{M+1})$,
$$
\frac{\partial v}{\partial x}(t_j, x_{M+1})\approx \frac{v(j,M+1)-v(j,M)}{h}.
$$
By $v(t,-L)=0$ and $\frac{\partial v}{\partial x}(t,L)=0$, we set
$
v(j,1)=0\quad {\rm and}\quad v(j,M+1)=v(j,M).
$

By \eqref{v-discre}, for each $j$, we have $M-1$ equations which form a system of linear algebraic equations for $(v(j,2), \cdots, v(j,M))$. It is nonsingular and there exists a unique solution $(v(j,2), \cdots, v(j,M))$.

By the forward approximation of the time derivative
$$
u_{t}(t_{j}, x_{i}) \approx \frac{u(j+1,i)-u(j,i)}{\tau},
$$
and the central approximation of the spatial derivative
$$
u_{xx}(t_{j}, x_{i}) \approx \frac{u(j,i-1)-2u(j,i)+u(j,i+1)}{h^2},
$$
$$
u_{x}(t_{j}, x_{i}) \approx \frac{u(j,i+1)-u(j,i-1)}{2h},
$$
equation \eqref{wave-cut-off-eq1-2-0} can be discretized as
\begin{align*}
\frac{u(j+1,i)-u(j,i)}{\tau}&=\frac{u(j,i-1)-2u(j,i)+u(j,i+1)}{h^2}\cr
&+(c-\chi \frac{v(j,i+1)-v(j,i-1)}{2h})\cdot\frac{u(j,i+1)-u(j,i-1)}{2h}\cr
&+u(j,i)\big(r(i)-\chi \nu v(j,i)-(b-\chi\mu)u(j,i)\big), \,\ 1\leq j\leq N,\quad 2\leq i\leq M.
\end{align*}
Simplify and reorder the above equations, we  get the following explicit formulas for $u(j+1,i)$ ($1\le j\le N$, $2\le i\le M$),
\begin{align*}
u(j+1,i)&=\big(\frac{\tau}{h^2}-\frac{\tau}{2h}(c-\chi \frac{v(j,i+1)-v(j,i-1)}{2h})\big)
u(j,i-1)\cr
&+\big(1-\frac{2\tau}{h^2}+\tau r(i)-\tau\chi\nu v(i)\big)u(j,i)-\tau(b-\chi\mu)u(j,i)^2\cr
&+\big(\frac{\tau}{h^2}+\frac{\tau}{2h}(c-\chi \frac{v(j,i+1)-v(j,i-1)}{2h})\big)u(j,i+1),\,\ 1\leq j\leq N,\quad 2\leq i\leq M.
\end{align*}
Similarly, by $u(t,-L)$ and $\frac{\partial u}{\partial x}(t,L)=0$, we set
$
u(j+1,1)=0 \quad {\rm and}\quad
 u(j+1,M+1)=u(j+1,M) \quad 1\leq j\leq N.
$
Thus, for each $j$, the values $u(j+1,i), 1 \leq i\leq M+1$ are obtained.

\smallskip

Next, we present our numerical simulations. We fix  the parameter values $\chi=0.1$, $\mu=1$, $\nu=0.05$,  and choose $r(x)$ to be the piece-wise linear function
$$
r(x)=\begin{cases}
-1 \quad  &{\rm if}\,\,  x\leq -8,\cr
11x+87 &{\rm if}\,\,  -8<x<-7,\cr
10                   \quad &{\rm if}\,\,  x\geq -7.\cr
\end{cases}
$$
For this choice of $r(x)$, $r^*=10$ and  $-c^*=-2\sqrt{r^*}\approx -6.325$.
We will do four numerical experiments for different values of $b$ and $c$.
{In these four numerical experiments, we will use the same space step size $h=0.1$ and the same time step size $\tau=0.002$. }

\smallskip

\noindent{\bf Numerical Experiment 1.}
Let $b=1$ and $c=1$. In this case,  $c>\frac{\chi\mu r^*}{2\sqrt{\nu}(b-\chi\mu)}-2\sqrt{\frac{r^*(b-2\chi\mu)}{b-\chi\mu}}$ becomes $c>\frac{5}{9\sqrt{0.05}}-2\sqrt{\frac{8}{0.9}}\approx -3.478$. So for these choices of $b$ and $c$, the assumption
{\bf (H1)} holds.

 We compute the numerical solution of \eqref{wave-cut-off-eq1-1} with $L=15, 20, 25, 30$, and $40$
on the time interval $[0,10]$. For all the choices of $L$, we observe that the numerical solution of \eqref{wave-cut-off-eq1-1}
changes very little after $t=3$, which indicates that the numerical solution converges to a stationary solution of \eqref{wave-cut-off-eq1-1}
as $t\to\infty$. We also observe that the numerical solution $u(t,x)$ at $t=10$ changes very little  and $u(10,L)$ is very close to
 $\frac{r^*}{b}=10$ as $L$ increases,
which indicates the stationary solution of \eqref{wave-cut-off-eq1-1} converges to a stationary solution of \eqref{wave-eq} connecting
$(\frac{r^*}{b},\frac{\mu}{\nu}\frac{r^*}{b})$ and $(0,0)$ or a forced wave solution of \eqref{Keller-Segel-eq0}  connecting
$(\frac{r^*}{b},\frac{\mu}{\nu}\frac{r^*}{b})$ and $(0,0)$ as $L\to\infty$, whose existence is proved in Theorem \ref{forced-wave-thm1}. Hence the numerical results for the choices $b=1$ and $c=1$ match the theoretical results.

We demonstrate the numerical solutions of \eqref{wave-cut-off-eq1-1} for the cases $L=20$ and $L=40$ in Figure \ref{Case1c1L20-1-1}  and
Figure \ref{Case1c1L40-1-1} , respectively.
Figure \ref{Case1c1L20-1} is the surface plot of the numerical solution of the system \eqref{wave-cut-off-eq1-1} on the interval $[-20,20]$ as time evolves. We plot the profile of the numerical solution at time $t=0,1,2,3,7,10$ in Figure \ref{Case1c1L20}.
Figure \ref{Case1c1L40-1} is the surface plot of the numerical solution of the system \eqref{wave-cut-off-eq1-1} on the interval $[-40,40]$ as time evolves. We plot the profile of the numerical solution at time $t=0,1,2,3,7,10$ in Figure \ref{Case1c1L40}.

\begin{figure}[H]
\centering
\subfigure[]
{\begin{minipage}{7cm}
	\centering
	\includegraphics[scale=0.17]{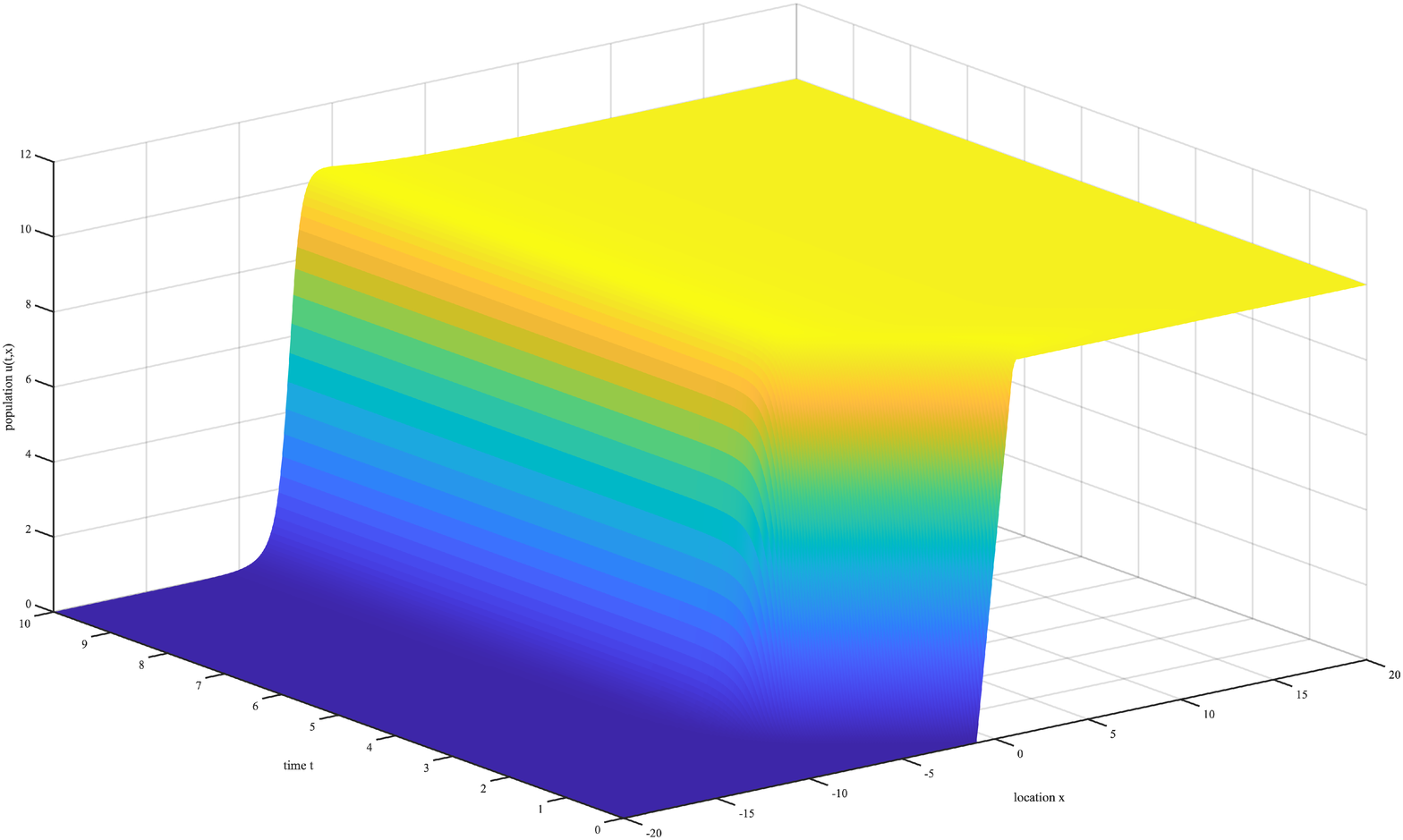}\label{Case1c1L20-1}
        \end{minipage}} \quad
	\subfigure[]
{\begin{minipage}{7cm}
	\centering
	\includegraphics[scale=0.30]{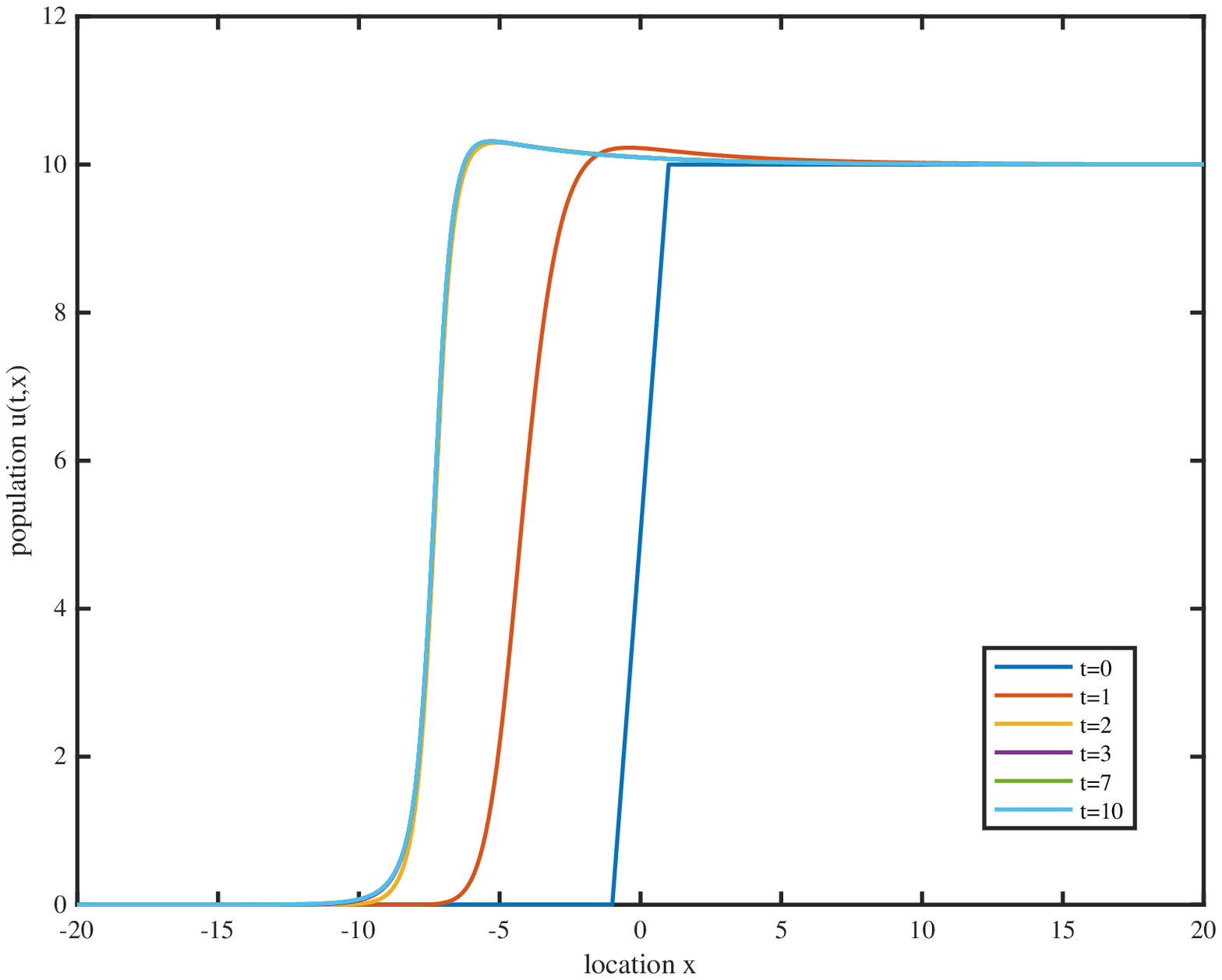}\label{Case1c1L20}
	\end{minipage}}
 \caption{{\bf (a)} Evolution of numerical solution of \eqref{wave-cut-off-eq1-1} on the interval $[-20,20]$ with $b=1$ and  $c=1$. {\bf (b)} numerical solution of \eqref{wave-cut-off-eq1-1} on the interval $[-20,20]$ at time $t=0,1,2,3,7,10$ with $b=1$ and $c=1$.}
\label{Case1c1L20-1-1}
\end{figure}

\begin{figure}[htbp]
\centering
\subfigure[]
{\begin{minipage}{7cm}
	\centering
	\includegraphics[scale=0.17]{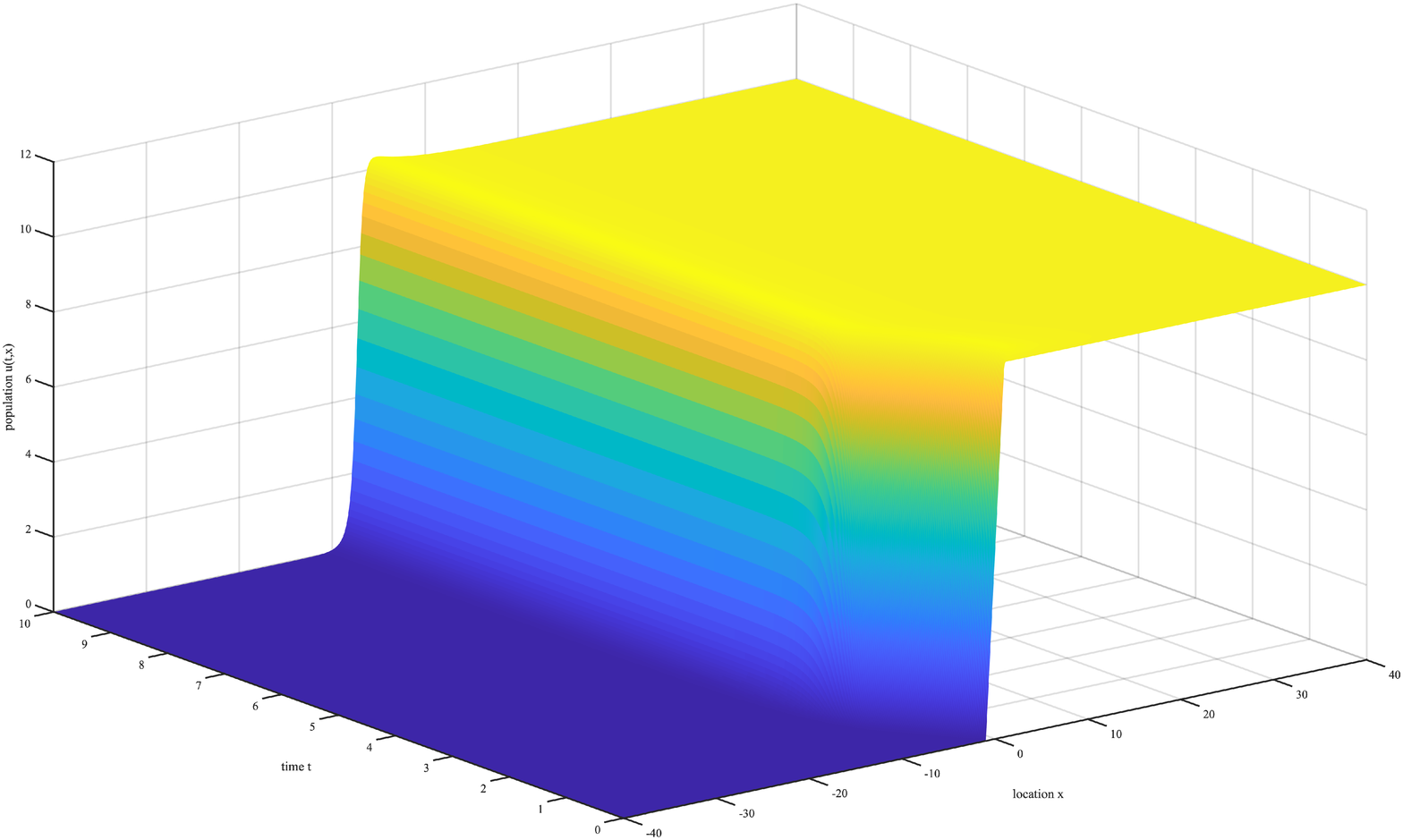}\label{Case1c1L40-1}
         \end{minipage}} \quad
	\subfigure[]
{\begin{minipage}{7cm}
	\centering
	\includegraphics[scale=0.30]{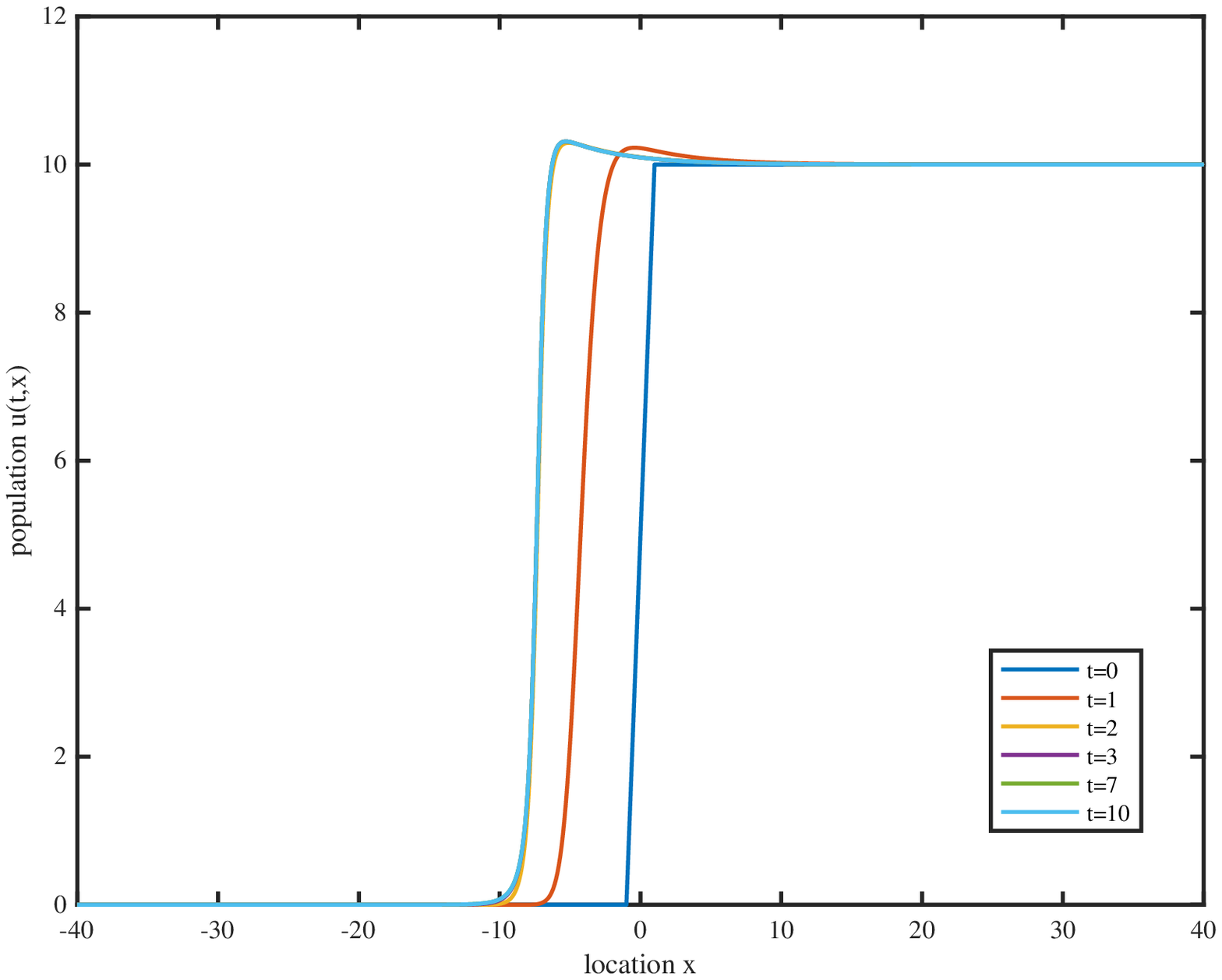}\label{Case1c1L40}
	\end{minipage}}
 \caption{{\bf (a)} Evolution of numerical solution of \eqref{wave-cut-off-eq1-1} on the interval $[-40,40]$ with $b=1$ and  $c=1$. {\bf (b)} numerical solution of \eqref{wave-cut-off-eq1-1} on the interval $[-40,40]$ at time $t=0,1,2,3,7,10$ with $b=1$ and $c=1$.}
\label{Case1c1L40-1-1}
\end{figure}

\noindent{\bf Numerical Experiment 2.}
Let $b=1$, $c=-6$. For these choices of $b$ and $c$, $b$ and $c$ satisfy $b>2\chi \mu$ and $c>-c^*$, but the assumption
$c>\frac{\chi\mu r^*}{2\sqrt{\nu}(b-\chi\mu)}-2\sqrt{\frac{r^*(b-2\chi\mu)}{b-\chi\mu}}$ does not hold.

 We compute the numerical solution of \eqref{wave-cut-off-eq1-1} with $L=15, 20, 25, 30$, and $40$
on the time interval $[0,15]$. For all the choices of $L$, we observe that the numerical solution of \eqref{wave-cut-off-eq1-1}
changes very little after $t=7$, which indicates that the numerical solution converges to a stationary solution of \eqref{wave-cut-off-eq1-1}
as $t\to\infty$.
We also observe that the numerical solution $u(t,x)$ at $t=15$ changes very little and $u(15,L)$ is very close to
 $\frac{r^*}{b}=10$ as $L$ increases,
which indicates the stationary solution of \eqref{wave-cut-off-eq1-1} converges to a stationary solution of \eqref{wave-eq} connecting
$(\frac{r^*}{b},\frac{\mu}{\nu}\frac{r^*}{b})$ and $(0,0)$ or a forced wave solution of \eqref{Keller-Segel-eq0}  connecting
$(\frac{r^*}{b},\frac{\mu}{\nu}\frac{r^*}{b})$ and $(0,0)$ as $L\to\infty$. The numerical results indicate that when
 $c>-c^*$ and $b>2\chi\mu$,   \eqref{Keller-Segel-eq0} still
 has a forced wave solution.
We demonstrate the numerical solutions of \eqref{wave-cut-off-eq1-1} for the case $L=20$ and $L=40$ in Figure \ref{Case1c6L20-1-1}  and
Figure \ref{Case1c6L40-1-1} , respectively.

\begin{figure}[H]
\centering
\subfigure[]
{\begin{minipage}{7cm}
	\centering
	\includegraphics[scale=0.17]{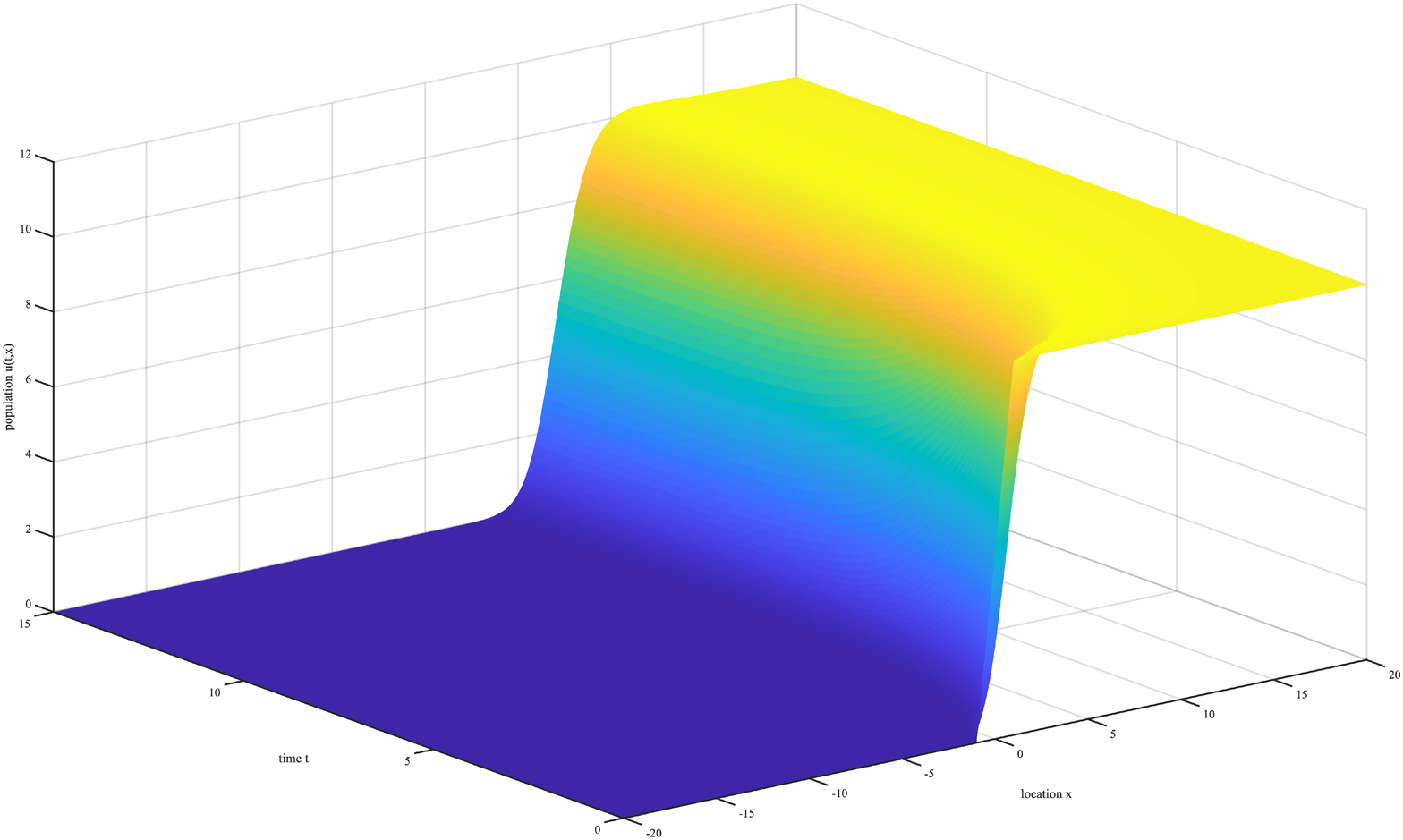}\label{Case1c6L20-1}
	\end{minipage}} \quad
	\subfigure[]
{\begin{minipage}{7cm}
	\centering
	\includegraphics[scale=0.30]{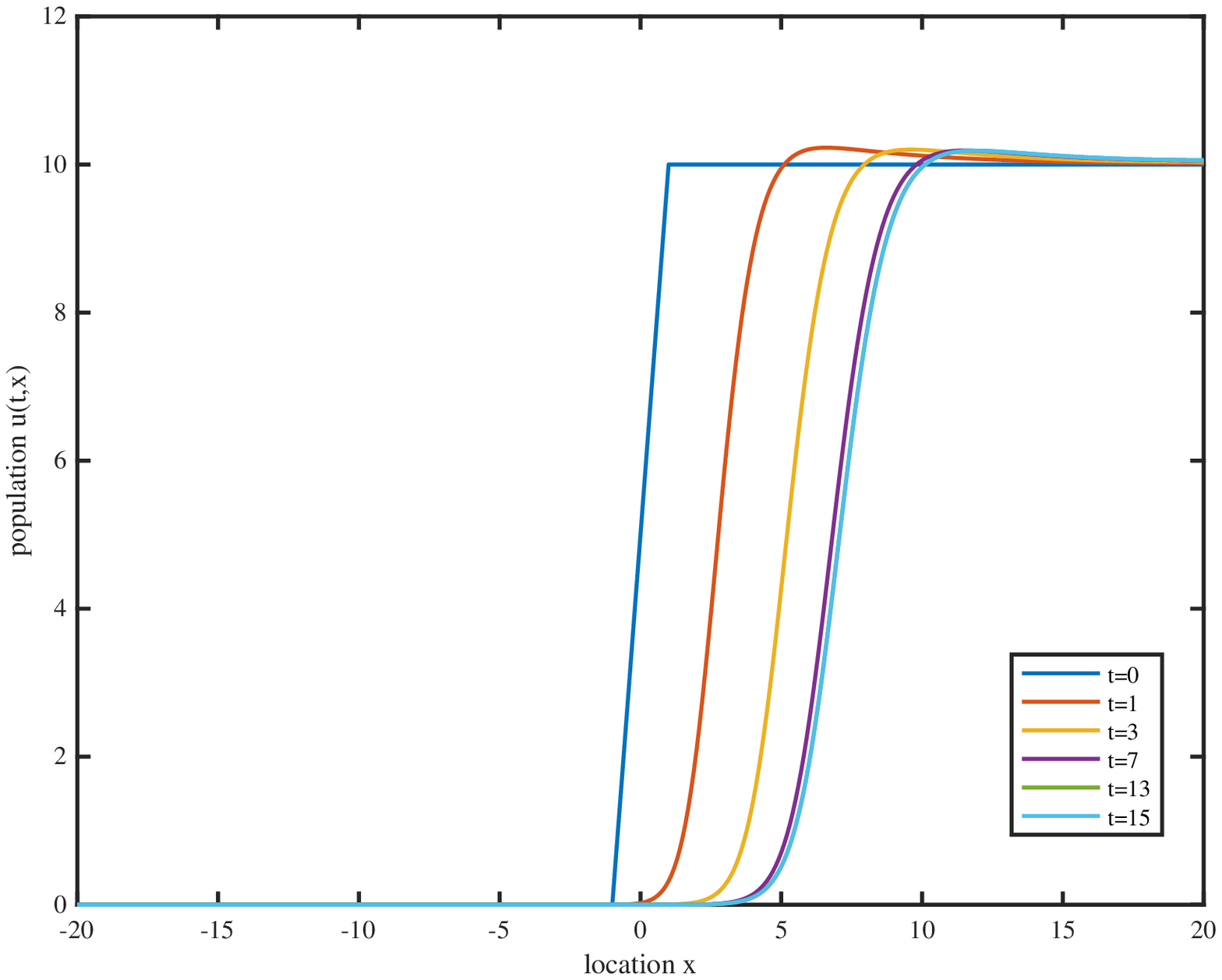}\label{Case1c6L20}
	\end{minipage}}
 \caption{{\bf (a)} Evolution of numerical solution of \eqref{wave-cut-off-eq1-1} on the interval $[-20,20]$ with $b=1$ and $c=-6$. {\bf (b)} numerical solution of \eqref{wave-cut-off-eq1-1} on the interval $[-20,20]$ at time $t=0,1,3,7,13,15$ with $b=1$ and $c=-6$.}
\label{Case1c6L20-1-1}
\end{figure}

\begin{figure}[H]
\centering
\subfigure[]
{\begin{minipage}{7cm}
	\centering
	\includegraphics[scale=0.17]{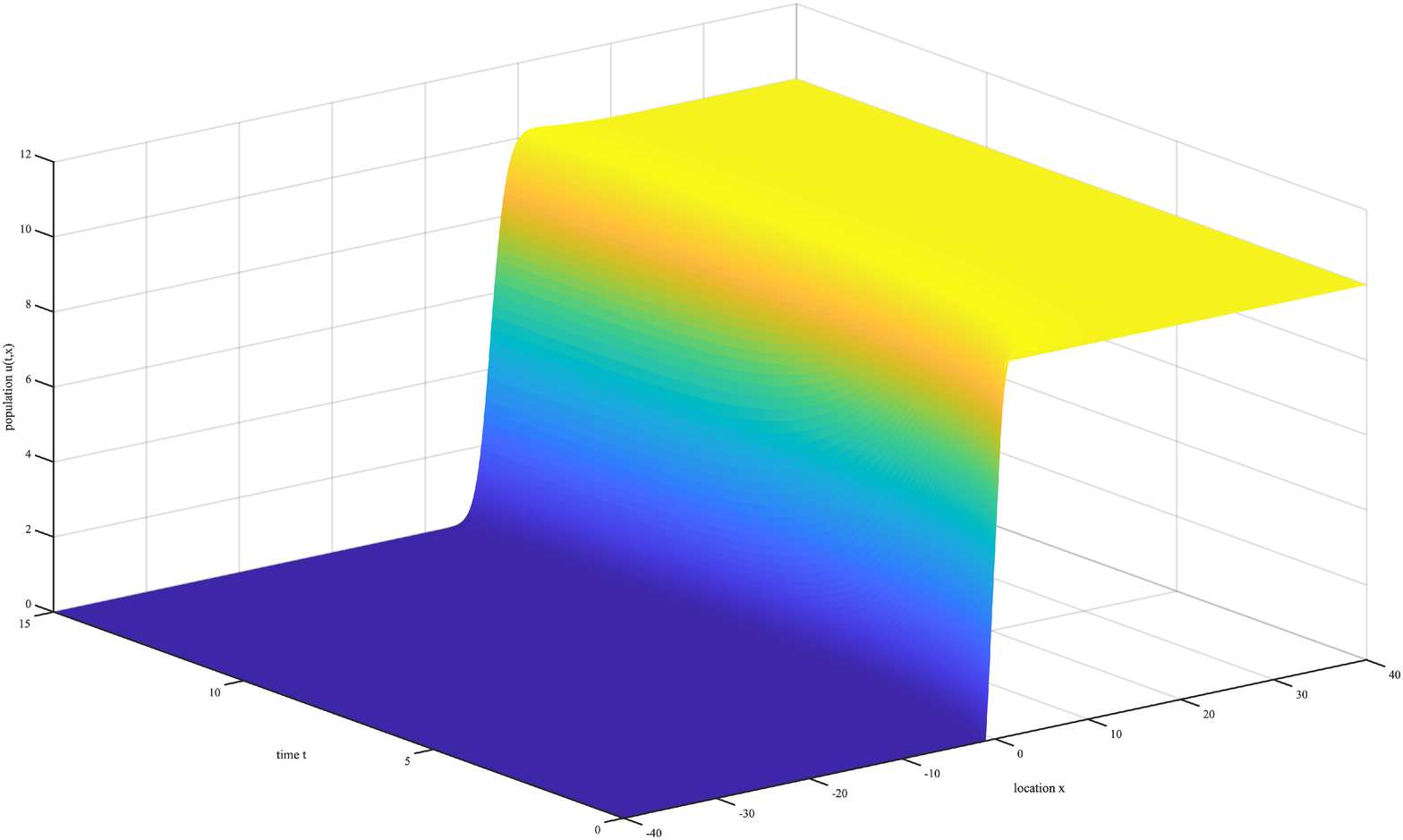}\label{Case1c6L40-1}
	\end{minipage}} \quad
	\subfigure[]
{\begin{minipage}{7cm}
	\centering
	\includegraphics[scale=0.30]{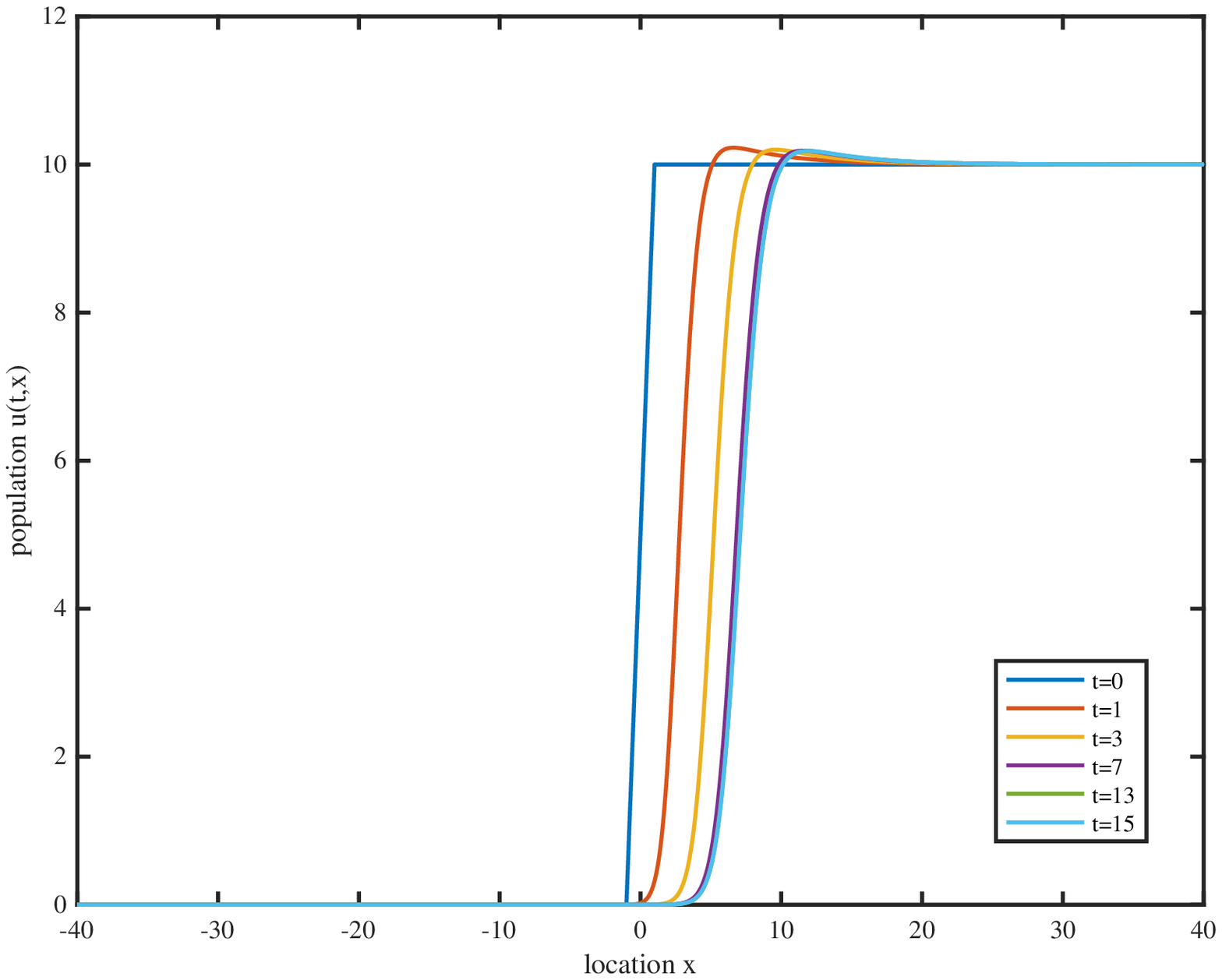}\label{Case1c6L40}
	\end{minipage}}
 \caption{{\bf (a)} Evolution of numerical solution of \eqref{wave-cut-off-eq1-1} on the interval $[-40,40]$ with $b=1$ and  $c=-6$. {\bf (b)} numerical solution of \eqref{wave-cut-off-eq1-1} on the interval $[-40,40]$ at time $t=0,1,3,7,13,15$ with $b=1$ and $c=-6$.}
\label{Case1c6L40-1-1}
\end{figure}

\noindent{\bf Numerical Experiment 3.}
Let $b=0.15$, $c=-6$. For these choices of $b$ and $c$, $b$ and $c$ satisfy $b>\chi \mu$ and $c>-c^*$.

 We compute the numerical solution of \eqref{wave-cut-off-eq1-1} with $L=35, 40, 45, 50$, and $60$
on the time interval $[0,60]$. For all the choices of $L$, we observe that the numerical solution of \eqref{wave-cut-off-eq1-1}
changes very little after $t=50$, which indicates that the numerical solution converges to a stationary solution of \eqref{wave-cut-off-eq1-1}
as $t\to\infty$.
We also observe that the numerical solution $u(t,x)$ at $t=60$ changes very little and $u(60,L)$ is very close to
 $\frac{r^*}{b}=\frac{10}{0.15}\approx 66.67$ as $L$ increases,
which indicates the stationary solution of \eqref{wave-cut-off-eq1-1} converges to a stationary solution of \eqref{wave-eq} connecting
 $(\frac{r^*}{b},\frac{\mu}{\nu}\frac{r^*}{b})$ and $(0,0)$ or a forced wave solution of \eqref{Keller-Segel-eq0}  connecting
$(\frac{r^*}{b},\frac{\mu}{\nu}\frac{r^*}{b})$ and $(0,0)$ as $L\to\infty$.   The numerical results indicate that when
 $c>-c^*$ and $b>\chi\mu$,   \eqref{Keller-Segel-eq0} still
 has a forced wave solution.

We demonstrate the numerical solutions of \eqref{wave-cut-off-eq1-1} for the case $L=40$ and $L=60$ in Figure \ref{Case1b015c-6L40-1-1}  and
Figure \ref{Case1b015c-6L60-1-1} , respectively.

\begin{figure}[H]
\centering
\subfigure[]
{\begin{minipage}{7cm}
	\centering
	\includegraphics[scale=0.17]{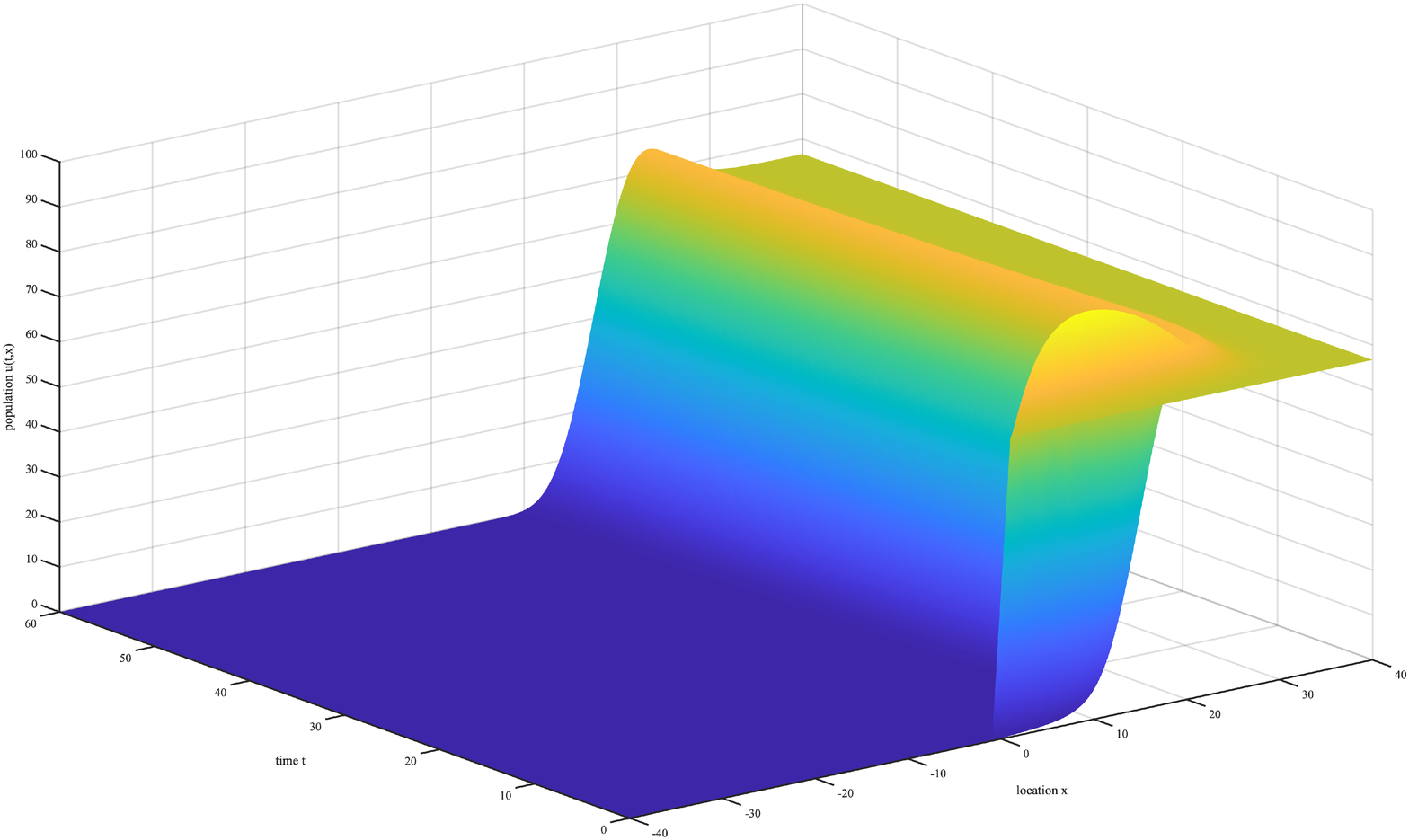}\label{Case1b015c-6L40-1}
	\end{minipage}} \quad
	\subfigure[]
{\begin{minipage}{7cm}
	\centering
	\includegraphics[scale=0.30]{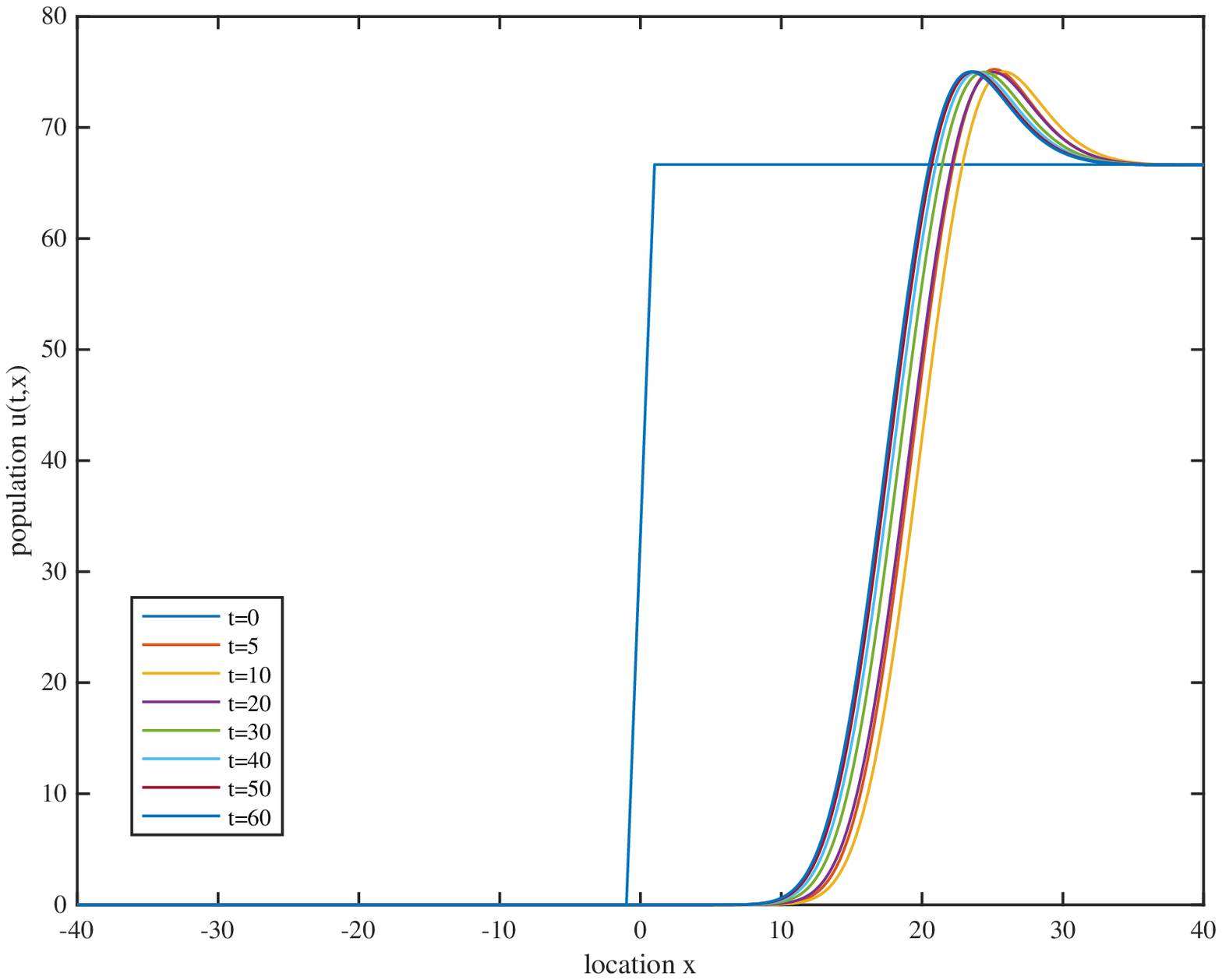}\label{Case1b015c-6L40}
	\end{minipage}}
 \caption{{\bf (a)} Evolution of numerical solution of \eqref{wave-cut-off-eq1-1} on the interval $[-40, 40]$ with $b=0.15$ and  $c=-6$. {\bf (b)} numerical solution of \eqref{wave-cut-off-eq1-1} on the interval $[-40, 40]$ at time $t=0, 5, 10, 20, 30, 40, 50, 60$ with $b=0.15$ and $c=-6$.}
\label{Case1b015c-6L40-1-1}
\end{figure}

\begin{figure}[H]
\centering
\subfigure[]
{\begin{minipage}{7cm}
	\centering
	\includegraphics[scale=0.17]{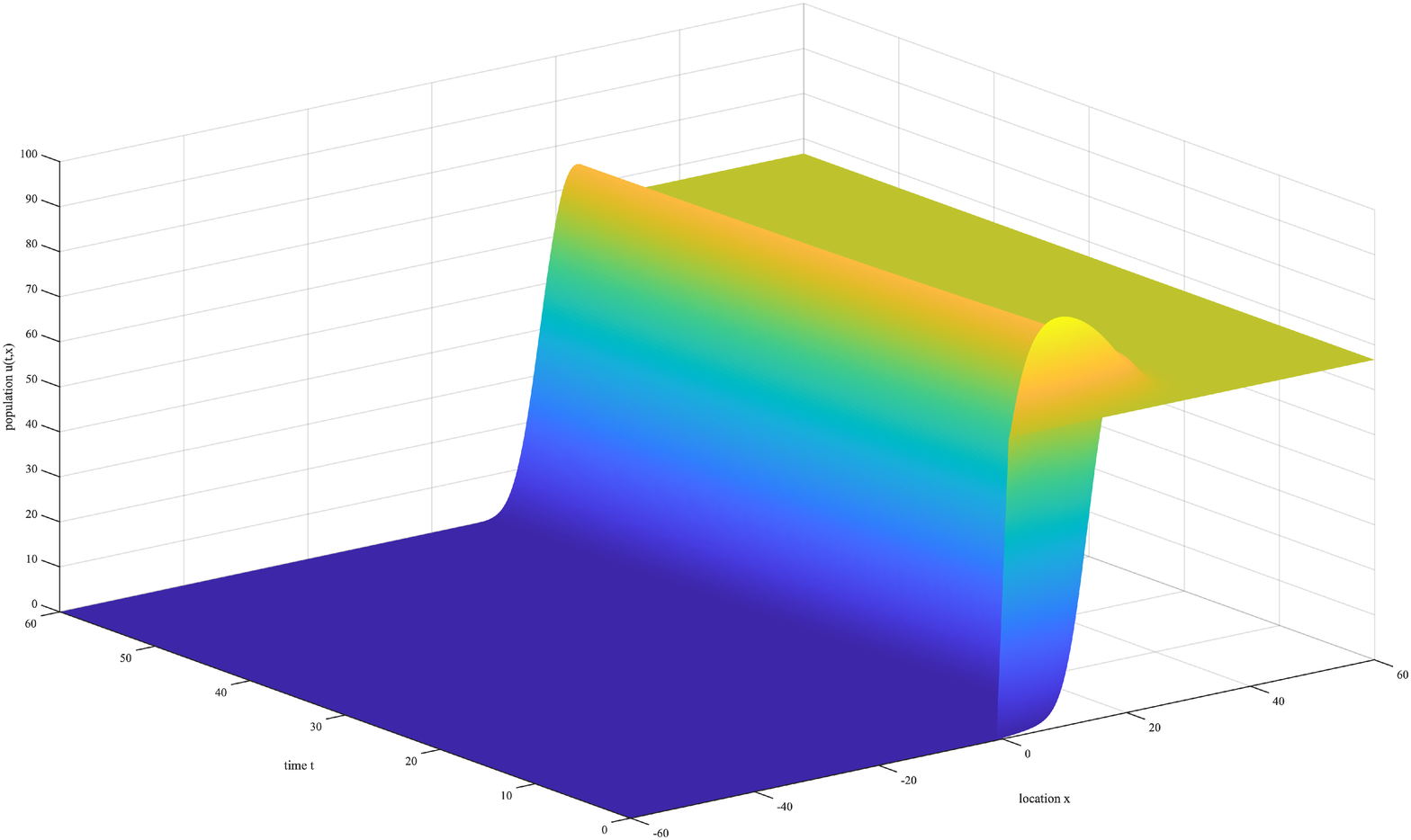}\label{Case1b015c-6L60-1}
	\end{minipage}} \quad
	\subfigure[]
{\begin{minipage}{7cm}
	\centering
	\includegraphics[scale=0.30]{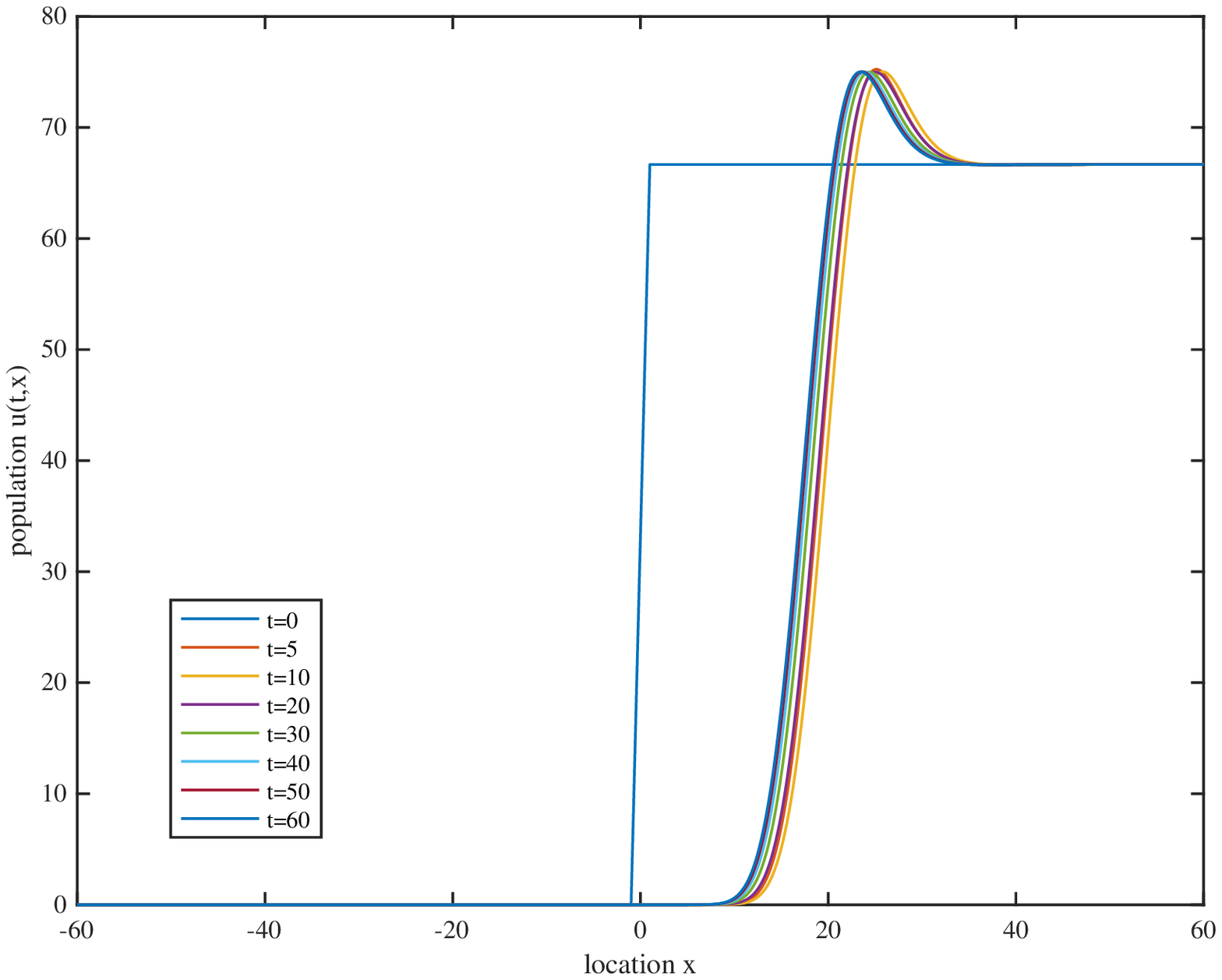}\label{Case1b015c-6L60}
	\end{minipage}}
 \caption{{\bf (a)} Evolution of numerical solution of \eqref{wave-cut-off-eq1-1} on the interval $[-60, 60]$ with $b=0.15$ and  $c=-6$. {\bf (b)} numerical solution of \eqref{wave-cut-off-eq1-1} on the interval $[-60, 60]$ at time $t=0, 5, 10, 20, 30, 40, 50, 60$ with $b=0.15$ and $c=-6$.}
\label{Case1b015c-6L60-1-1}
\end{figure}

\medskip

\noindent{\bf Numerical Experiment 4.}
Let $b=1$, $c=-6.5$. For these choices of $b$ and $c$, $b$ and $c$ satisfy $b>2\chi \mu$ and $c<-c^*$.

 We compute the numerical solution of \eqref{wave-cut-off-eq1-1} with $L=15$ on the time interval $[0,50]$, with $L= 20$  on the time interval $[0,60]$, with $L=25$ on the time interval $[0,70]$, with $L=30$ on the time interval $[0,90]$, and with $L=40$ on the time interval $[0,140]$.
For all the choices of $L$, we observe that the numerical solution of \eqref{wave-cut-off-eq1-1}
becomes very small after certain time, which indicates that the numerical solution converges to the zero solution of \eqref{wave-cut-off-eq1-1}
as $t\to\infty$, and also indicates that
 \eqref{wave-eq} has no positive stationary solutions  or \eqref{Keller-Segel-eq0}  has  no forced wave solutions in the case that $c<-c^*$.

We demonstrate the numerical solutions of \eqref{wave-cut-off-eq1-1} for the case $L=20$ and $L=40$ in Figure \ref{Case1c65L20-1-1}  and
Figure \ref{Case1c65L40-1-1} , respectively.

\begin{figure}[H]
\centering
\subfigure[]
{\begin{minipage}{7cm}
	\centering
	\includegraphics[scale=0.17]{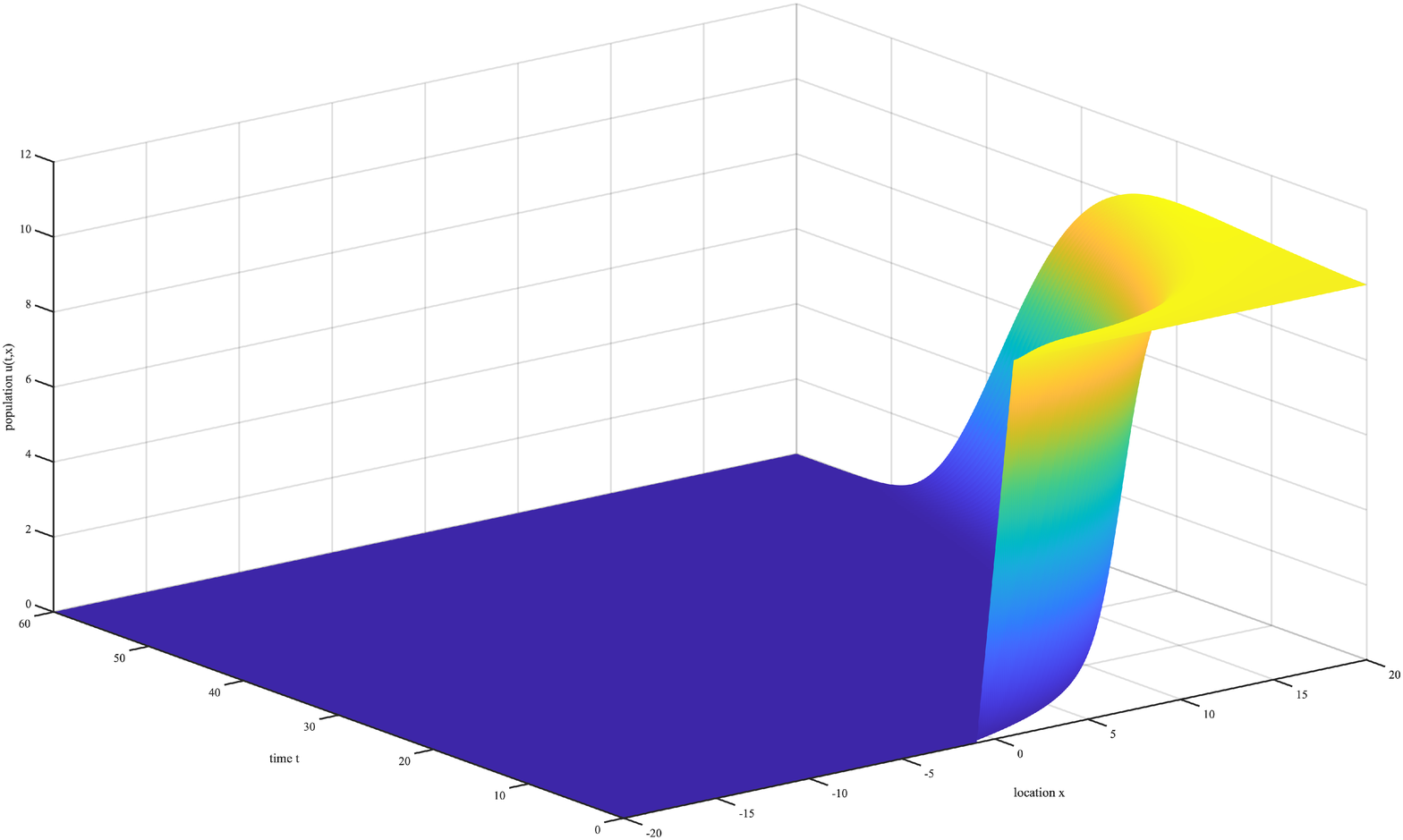}\label{Case1c65L20-1}
	\end{minipage}} \quad
	\subfigure[]
{\begin{minipage}{7cm}
	\centering
	\includegraphics[scale=0.30]{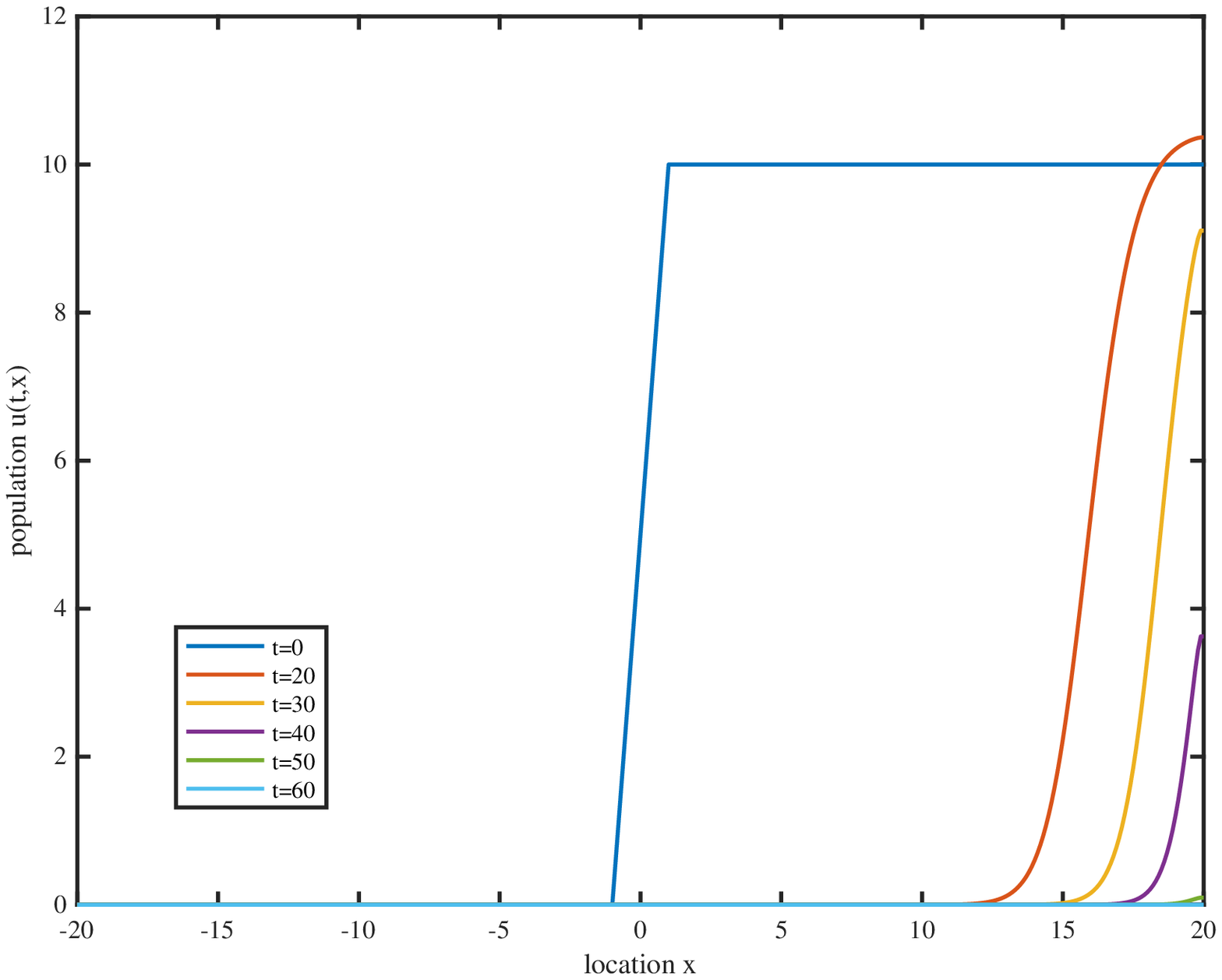}\label{Case1c65L20}
	\end{minipage}}
 \caption{{\bf (a)} Evolution of numerical solution of \eqref{wave-cut-off-eq1-1} on the interval $[-20,20]$ with $b=1$ and  $c=-6.5$. {\bf (b)} numerical solution of \eqref{wave-cut-off-eq1-1} on the interval $[-20,20]$ at time $t=0, 20, 30, 40, 50, 60$ with $b=1$ and $c=-6.5$.}
\label{Case1c65L20-1-1}
\end{figure}

\begin{figure}[H]
\centering
\subfigure[]
{\begin{minipage}{7cm}
	\centering
	\includegraphics[scale=0.17]{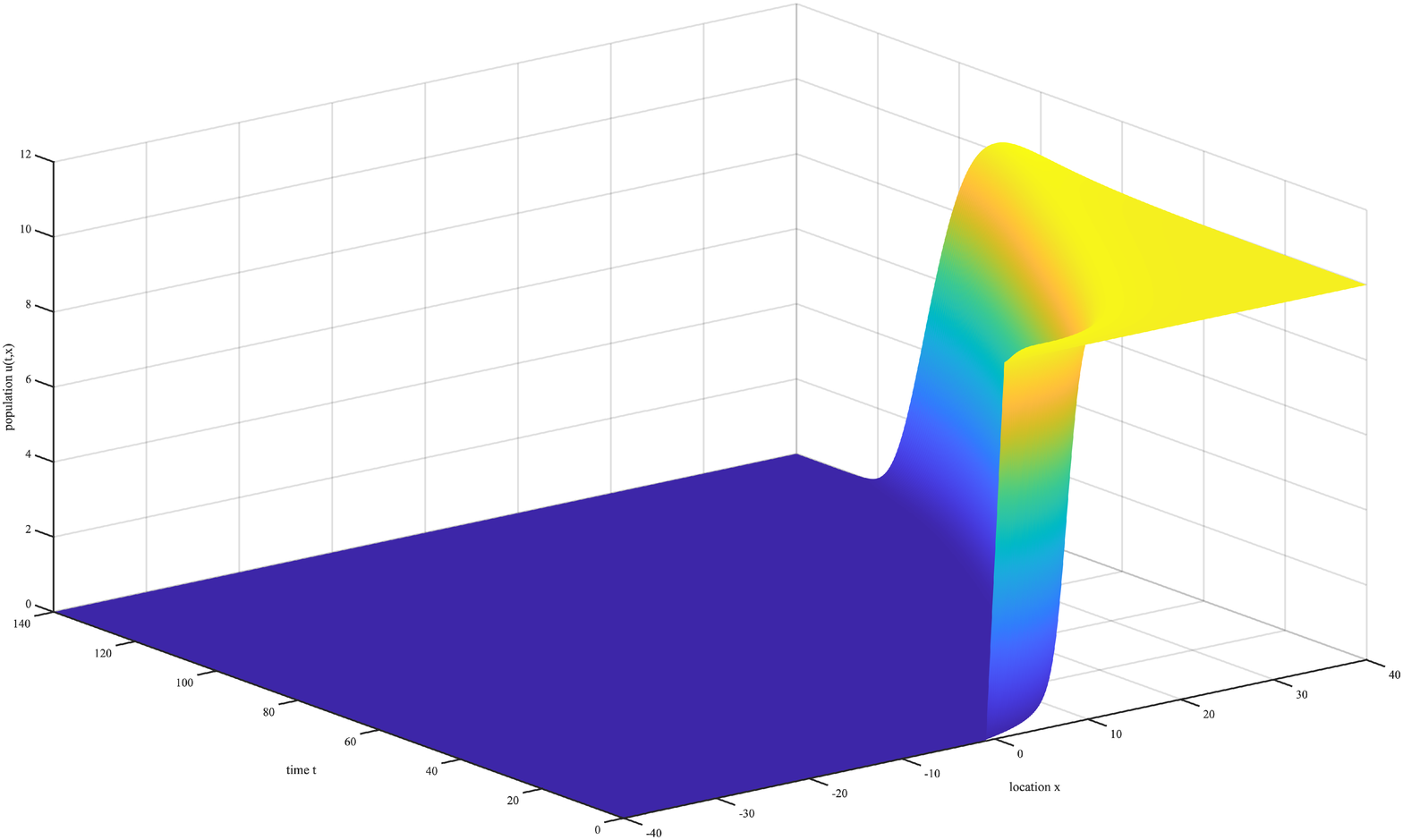}\label{Case1c65L40-1}
	\end{minipage}} \quad
	\subfigure[]
{\begin{minipage}{7cm}
	\centering
	\includegraphics[scale=0.30]{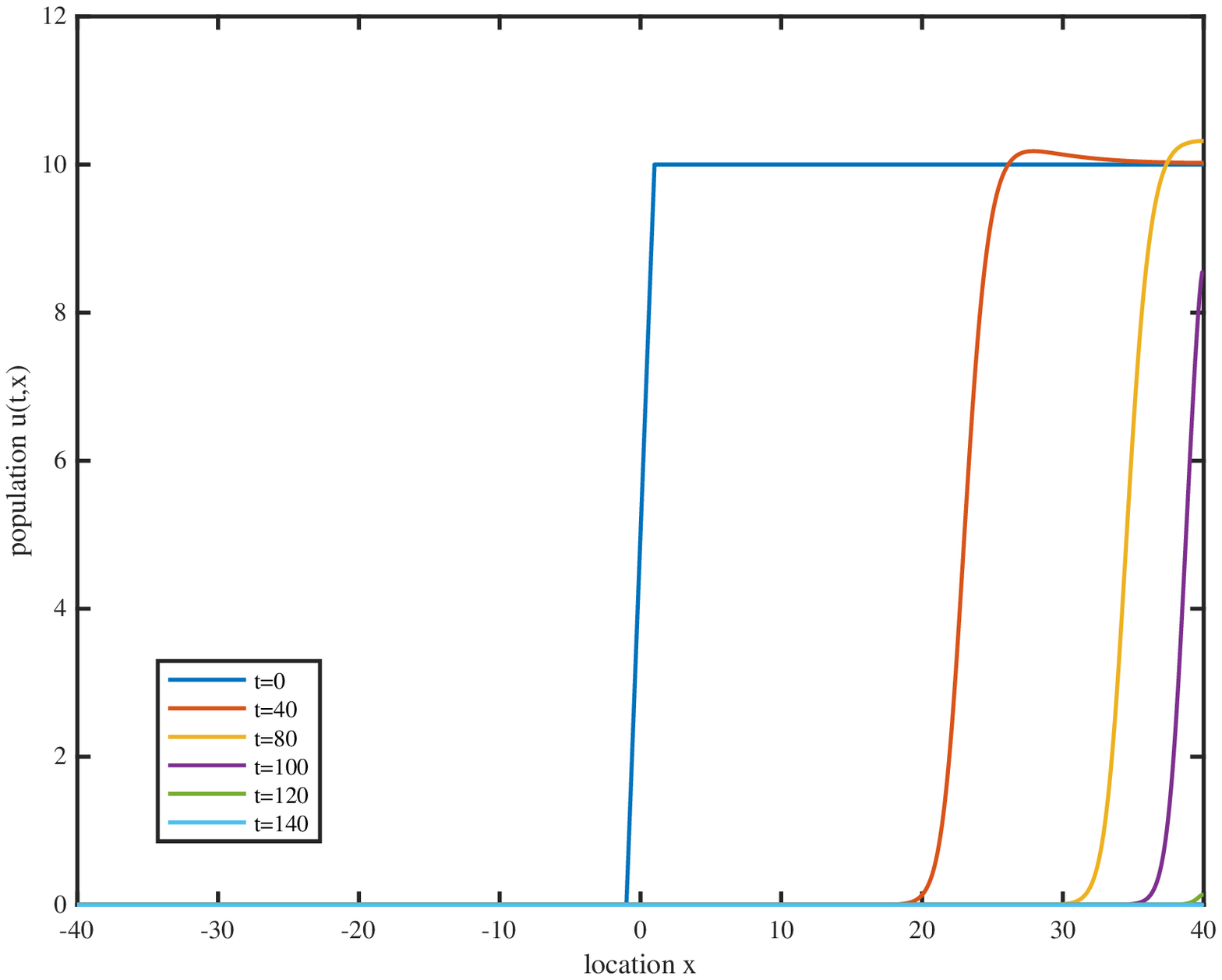}\label{Case1c65L40}
	\end{minipage}}
 \caption{{\bf (a)} Evolution of numerical solution of \eqref{wave-cut-off-eq1-1} on the interval $[-40,40]$ with $b=1$ and  $c=-6.5$. {\bf (b)} numerical solution of \eqref{wave-cut-off-eq1-1} on the interval $[-40,40]$ at time $t=0, 40, 80, 100, 120, 140$ with $b=1$ and $c=-6.5$.}
\label{Case1c65L40-1-1}
\end{figure}

\begin{rk}
\begin{itemize}
\item[(1)]
The numerical simulations above illustrate our Theorem \ref{forced-wave-thm1} and also shows that the assumptions in
Theorem \ref{forced-wave-thm1} can be weakened. Based on these numerical simulations, we conjecture that if $b>\chi\mu$, and $c>-2\sqrt{r^*}$, there is a forced wave solution
$(u(t,x),v(t,x))=(\phi(x-ct),\psi(x-ct))$ connecting $(\frac{r^*}{b},\frac{\mu}{\nu}\frac{r^*}{b})$ and $(0,0)$, that is,
$\phi(\infty)=\frac{r^*}{b}$ and $\phi(-\infty)=0$. If $b>\chi\mu$ and $c<-2\sqrt{r^*}$, there is no  forced wave solution
$(u(t,x),v(t,x))=(\phi(x-ct),\psi(x-ct))$ connecting $(\frac{r^*}{b},\frac{\mu}{\nu}\frac{r^*}{b})$ and $(0,0)$, that is,
$\phi(\infty)=\frac{r^*}{b}$ and $\phi(-\infty)=0$.

\item[(2)] {We used the explicit  scheme  in the above simulations.  It should be pointed out that
the explicit scheme and  implicit scheme for parabolic-type equations have the same order of accuracy,
but the implicit scheme  yields a weaker CFL condition for the stability. For example,
 it  is known that  the explicit scheme for the heat equation
  $$
  u_t=u_{xx}
  $$
   is conditionally stable under the following CFL condition:
  $$
  \frac{\tau}{h^2} \le \frac{1}{2},
  $$
  where $\tau$ is the time step size and $h$ is the space step size, while the implicit scheme is unconditionally stable for the heat equation.
  }

\item[(3)]
In the above four numerical experiments, we used the same space step size $h=0.1$ and the same time step size $\tau=0.002$.
They satisfy the numerical stable condition $\frac{\tau}{h^2}<\frac{1}{2}$. We do not give the accuracy analysis of the simulations
in this paper. To see the  reliability of the numerical results,
we also tried different values of  $h$ and  $\tau$ to simulate the existence of forced wave solutions.
For the above four experiments, let $h=0.1$ be fixed, choose $\tau=0.001, 0.002, 0.004$ respectively, the graphs we got do not have a big difference. Fix $h=0.05$, let $\tau=0.001, 0.0005, 0.00025$ respectively, the graphs we got also do not have a big difference.

\item[(4)]
We also tried to use different initial conditions to simulate the forced wave. For example, let
$$
u_0(x)=\begin{cases}
0 \quad  &{\rm if}\,\,  x<-1,\cr
x+1                    \quad &{\rm if}\,\,  -1\leq x\leq 1,\cr
2 \quad  &{\rm if}\,\,  x>1.
\end{cases}
$$
We see similar dynamical scenarios. We then conjecture that the forced wave solution of \eqref{Keller-Segel-eq0}
is unique and stable in certain parameter region.

\end{itemize}
\end{rk}

\subsection{Numerical simulations in Case 2}

In this subsection, we study the numerical simulations of the forced wave solutions in Case 2.
To this end, we consider
\begin{equation}\label{wave-cut-off-eq2-1}
\begin{cases}
u_t= u_{xx}+cu_x- (\chi u  v_x)_x+u(r(x)-bu),\quad -L<x<L\cr
0 =v_{xx}-  \nu v +\mu u,\quad -L<x<L\cr
u(0,x)=u_0(x), \quad -L\leq x\leq L,\cr
u(t,-L)=v(t,-L)=0\cr
u(t,L)=v(t,L)=0
\end{cases}
\end{equation}
for reasonable large $L>1$,
where
$$
u_0(x)=\begin{cases}
0 \quad  &{\rm if}\,\,  |x|>1,\cr
(x+1)(1-x)                    \quad &{\rm if}\,\,  -1\leq x\leq 1.
\end{cases}
$$

\medskip

Similarly, to compute the solution of \eqref{wave-cut-off-eq2-1} on a time interval $[0,T]$  numerically, we divide the space interval $[-L,L]$ into $M$ subintervals with equal length and divide the time interval $[0,T]$ into $N$ subintervals with equal length. The space step size is $h=\frac{2L}{M}$ and the time step size is $\tau=\frac{T}{N}$. For simplicity, we denote the approximate value of $u(t_j, x_i)$ by $u(j,i)$, $r(x_i)$ by $r(i)$ and $v(t_j, x_i)$ by $v(j,i)$ with $t_{j}=(j-1)\tau$, $1 \leq j\leq N+1$, and $x_{i}=-L+(i-1)h$, $1 \leq i\leq  M+1$.

By the same procedure as in Case 1,
the second equation in \eqref{wave-cut-off-eq2-1} can be discretized as
\begin{equation}
\frac{v(j,i-1)-2v(j,i)+v(j,i+1)}{h^2}-\nu v(j,i)+\mu u(j,i)=0, \quad 2\leq i\leq M.
\end{equation}
Together with the boundary conditions $$v(j,1)=v(j,M+1)=0,$$ we can solve for the solution $(v(j,1), v(j,2), \cdots, v(j,M+1))$ for each $j$.
By the same numerical discretization for the first equation of \eqref{wave-cut-off-eq2-1} as in Case 1,
we have an explicit formula for $u(j+1,i)$,
\begin{align*}
u(j+1,i)&=\big(\frac{\tau}{h^2}-\frac{\tau}{2h}(c-\chi \frac{v(j,i+1)-v(j,i-1)}{2h})\big)
u(j,i-1)\cr
&+\big(1-\frac{2\tau}{h^2}+\tau r(i)-\tau\chi\nu v(i)\big)u(j,i)-\tau(b-\chi\mu)u(j,i)^2\cr
&+\big(\frac{\tau}{h^2}+\frac{\tau}{2h}(c-\chi \frac{v(j,i+1)-v(j,i-1)}{2h})\big)u(j,i+1),\,\ 1\leq j\leq N,\quad 2\leq i\leq M.
\end{align*}
By the boundary condition, we set
$$
u(j+1,1)=u(j+1, M+1)=0 \quad 1\leq j\leq N.
$$
Thus, the values $u(j+1,i), 1 \leq i\leq M+1$ are obtained.

We choose $\mu=1$, $\nu=1$ and the following growth rate function
$$
r(x)=\begin{cases}
-1 \quad  &{\rm if}\,\,  |x| \ge 8,\cr
11x+87 &{\rm if}\,\,  -8<x< -7,\cr
10                   \quad &{\rm if}\,\,  -7\leq x\leq 7,\cr
-11x+87 &{\rm if}\,\,  7<x<8.
\end{cases}
$$
For this choice of $r(x)$, $r^*=10$ and  $c^*:=2\sqrt{r^*}\approx 6.325$.
We will do three  numerical experiments for different values of $b$, $c$ and $\chi$.
In these three numerical experiments, we will use the same space step size $h=0.1$ and the same time step size $\tau=0.002$.

\medskip

\noindent{\bf Numerical Experiment 1.}
Choose $c=1$, then $\lambda_{-7}(r(\cdot))=\frac{40-1-\frac{\pi^2}{49}}{4}>0$. Since $\lambda_{L}(r(\cdot))>\lambda_{-7}(r(\cdot))$ for $L>-7$, we have $\lambda_\infty(r(\cdot))\geq\lambda_{-7}(r(\cdot))>0$. Choose $b=1$ and $\chi=0.6$. Then  $b\ge \frac{3\chi\mu}{2}$.

\smallskip

We compute the numerical solution of \eqref{wave-cut-off-eq2-1} with $L=15, 20, 25, 30$, and $40$ on the time interval $[0,10]$. In all the cases,
we observe that the numerical solution changes very little after $t=3$ and stays away from $0$ on some fixed interval, which indicates
that the numerical solution converges to a positive stationary solution of \eqref{wave-cut-off-eq2-1} as $t\to\infty$. We also observe that
the numerical solution $u(t,x)$ at $t=10$ changes very little as $L$ increases, which indicates that the stationary solution of \eqref{wave-cut-off-eq2-1} converges as $L\to\infty$ to a stationary solution of \eqref{wave-eq}  or a forced wave solution of \eqref{Keller-Segel-eq0}  connecting
 $(0,0)$ and $(0,0)$.
We demonstrate the numerical solutions of \eqref{wave-cut-off-eq2-1} for the cases $L=20$ and $L=40$ in Figure \ref{Case2b1chi06L20-1-1}  and
Figure \ref{Case2b1chi06L40-1-1}, respectively.


\begin{figure}[H]
\centering
\subfigure[]
{\begin{minipage}{7cm}
	\centering
	\includegraphics[scale=0.17]{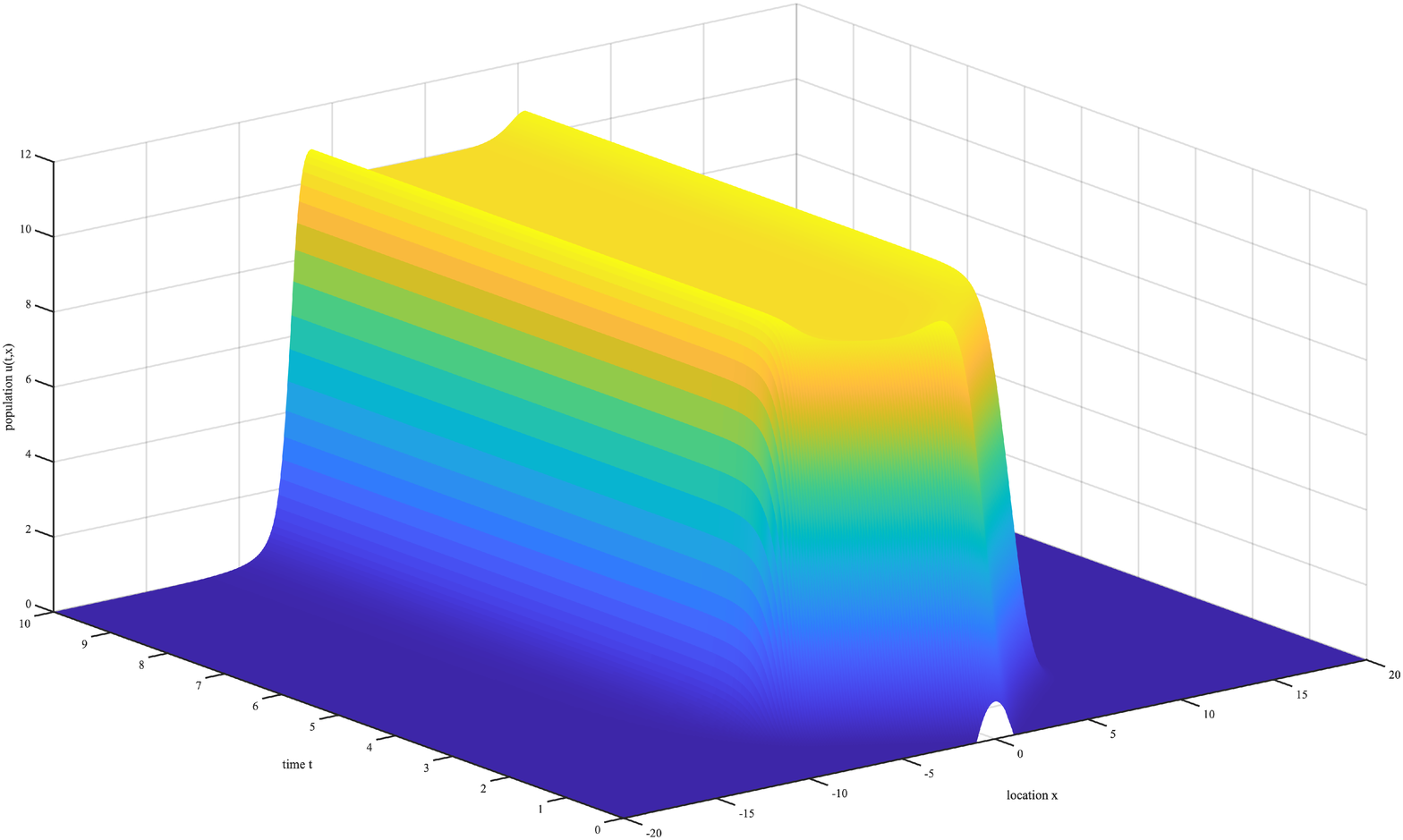}\label{Case2b1chi06L20-1}
	\end{minipage}} \quad
	\subfigure[]
{\begin{minipage}{7cm}
	\centering
	\includegraphics[scale=0.30]{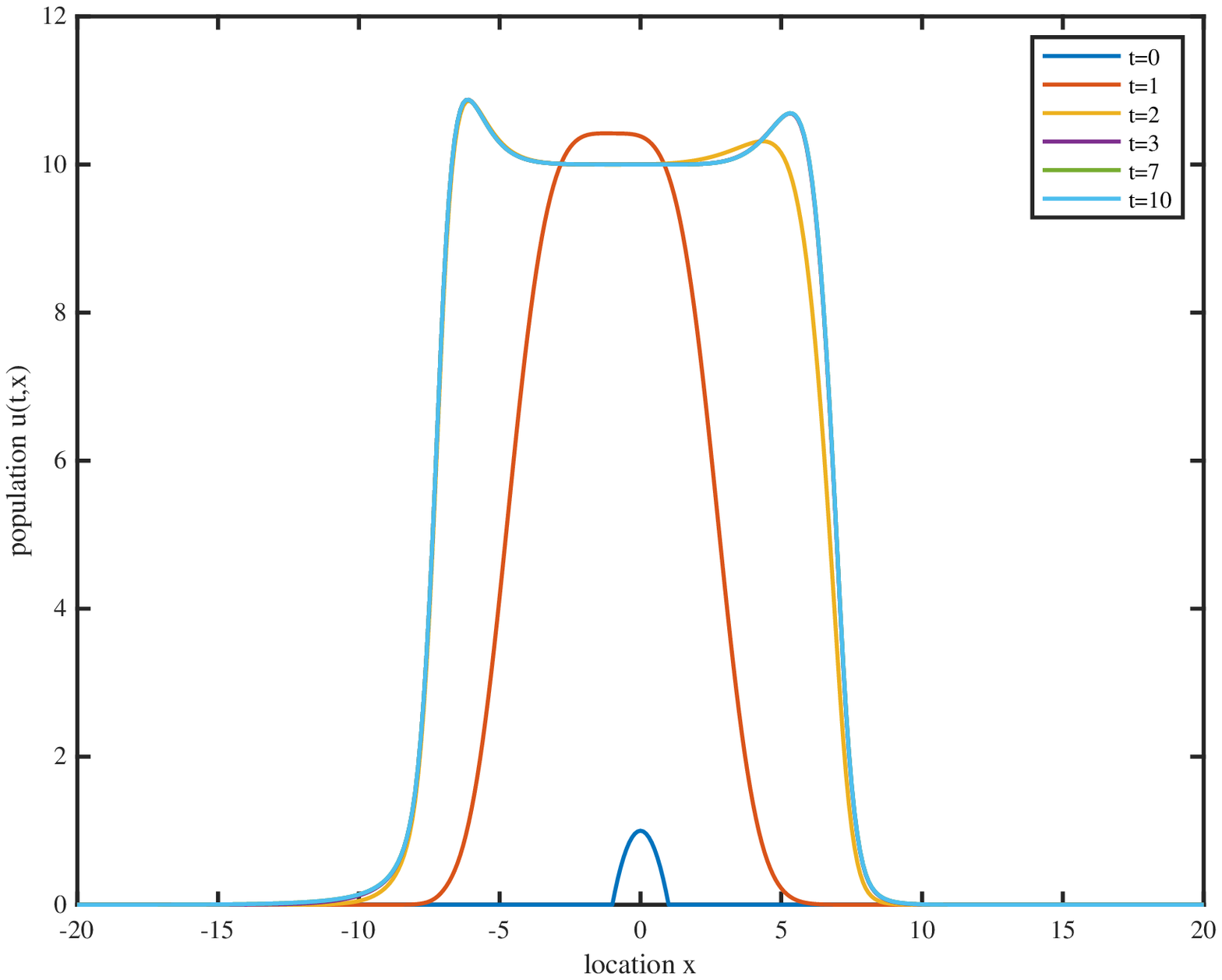} \label{Case2b1chi06L20}
	\end{minipage}}
 \caption{{\bf (a)} Evolution of numerical solution of \eqref{wave-cut-off-eq2-1} on the interval $[-20, 20]$ with $c=1$, $b=1$ and  $\chi=0.6$. {\bf (b)} numerical solution of \eqref{wave-cut-off-eq2-1} on the interval $[-20, 20]$ at time $t=0, 1, 2, 3, 7, 10$ with $c=1$, $b=1$ and $\chi=0.6$.}
\label{Case2b1chi06L20-1-1}
\end{figure}


\begin{figure}[H]
\centering
\subfigure[]
{\begin{minipage}{7cm}
	\centering
	\includegraphics[scale=0.17]{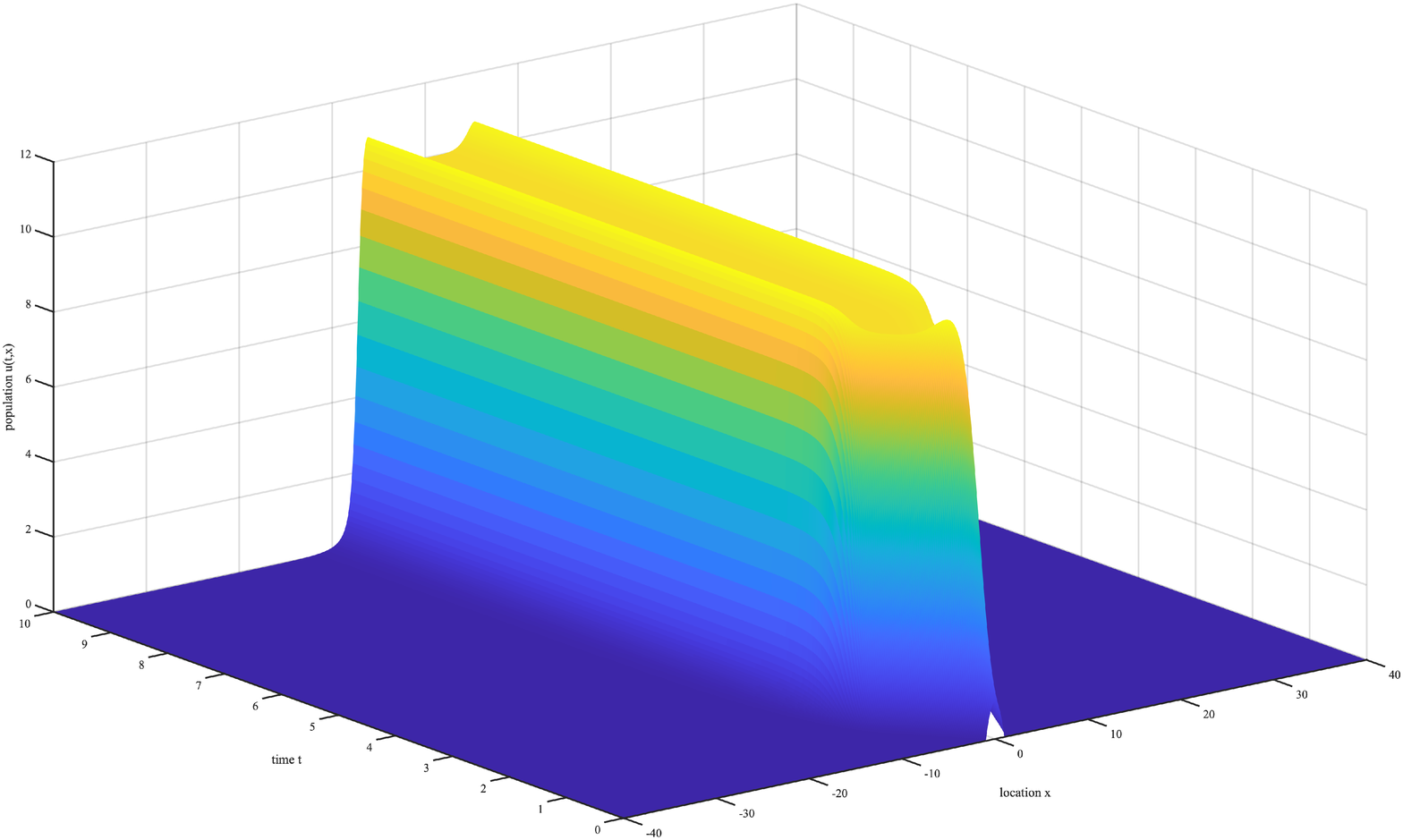} \label{Case2b1chi06L40-1}
	\end{minipage}} \quad
	\subfigure[]
{\begin{minipage}{7cm}
	\centering
	\includegraphics[scale=0.30]{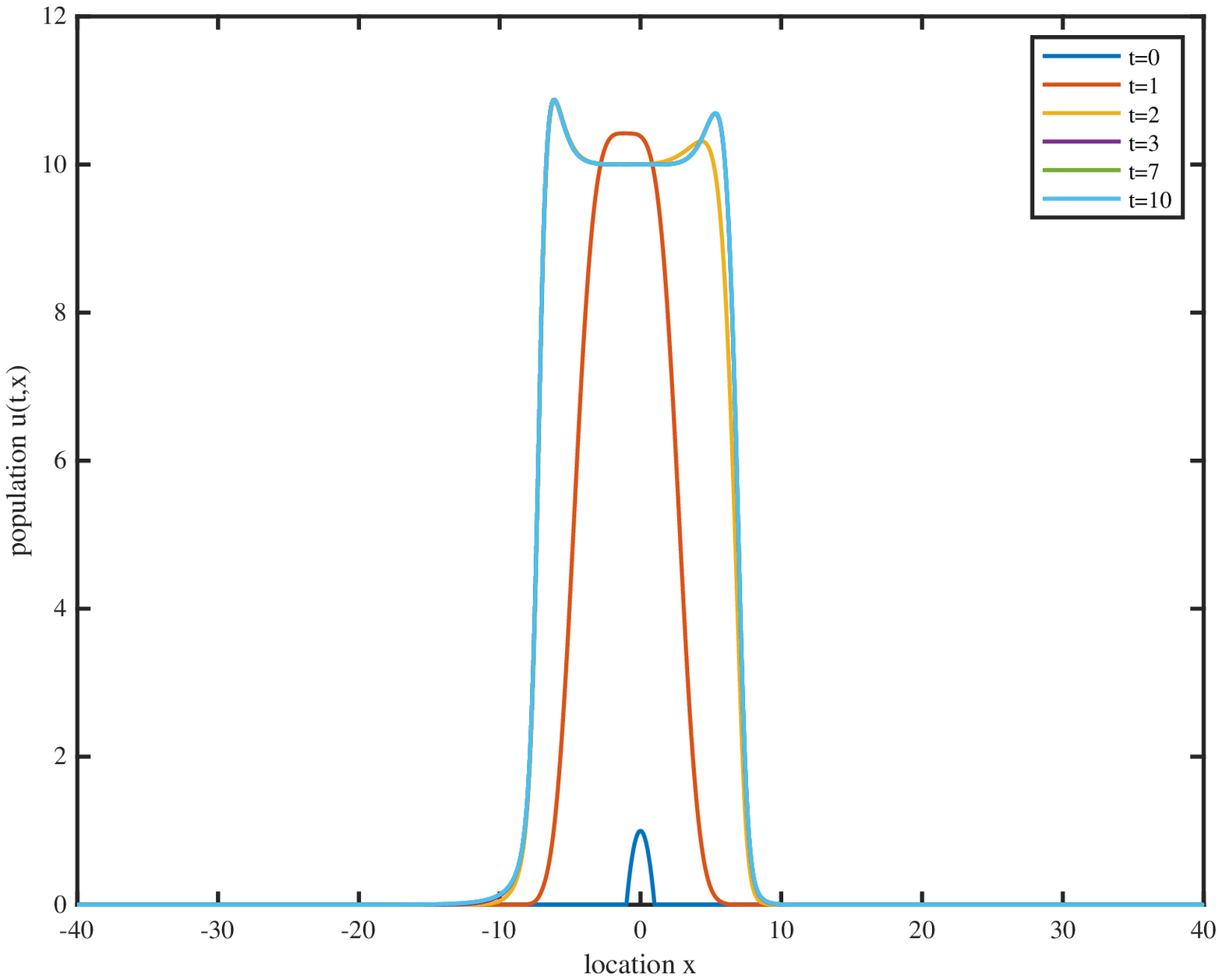} \label{Case2b1chi06L40}
	\end{minipage}}
 \caption{{\bf (a)} Evolution of numerical solution of \eqref{wave-cut-off-eq2-1} on the interval $[-40, 40]$ with $c=1$, $b=1$ and  $\chi=0.6$. {\bf (b)} numerical solution of \eqref{wave-cut-off-eq2-1} on the interval $[-40, 40]$ at time $t=0, 1, 2, 3, 7, 10$ with $c=1$, $b=1$ and $\chi=0.6$.}
\label{Case2b1chi06L40-1-1}
\end{figure}

\medskip

\noindent{\bf Numerical Experiment 2.}
Choose $c=1$ (then $\lambda_\infty(r(\cdot))>0$). Choose $b=0.7$ and $\chi=0.6$ (then $\chi\mu<b< \frac{3\chi\mu}{2}$).

We compute the numerical solution of \eqref{wave-cut-off-eq2-1} with $L=15, 20, 25, 30$, and $40$ on the time interval $[0,10]$. In all the cases,
we observe that the numerical solution changes very little after $t=3$ and stays away from $0$ on some fixed interval, which indicates
that the numerical solution converges to a positive stationary solution of \eqref{wave-cut-off-eq2-1} as $t\to\infty$. We also observe that
the numerical solution $u(t,x)$ at $t=10$ changes very little as $L$ increases, which indicates that the stationary solution of \eqref{wave-cut-off-eq2-1} converges as $L\to\infty$  to a stationary solution of \eqref{wave-eq}  or a forced wave solution of \eqref{Keller-Segel-eq0}  connecting
 $(0,0)$ and $(0,0)$.
We demonstrate the numerical solutions of \eqref{wave-cut-off-eq2-1} for the cases $L=20$ and $L=40$ in Figure \ref{Case2b07chi06L20-1-1}  and
Figure \ref{Case2b07chi06L40-1-1}, respectively.

\begin{figure}[H]
\centering
\subfigure[]
{\begin{minipage}{7cm}
	\centering
	\includegraphics[scale=0.17]{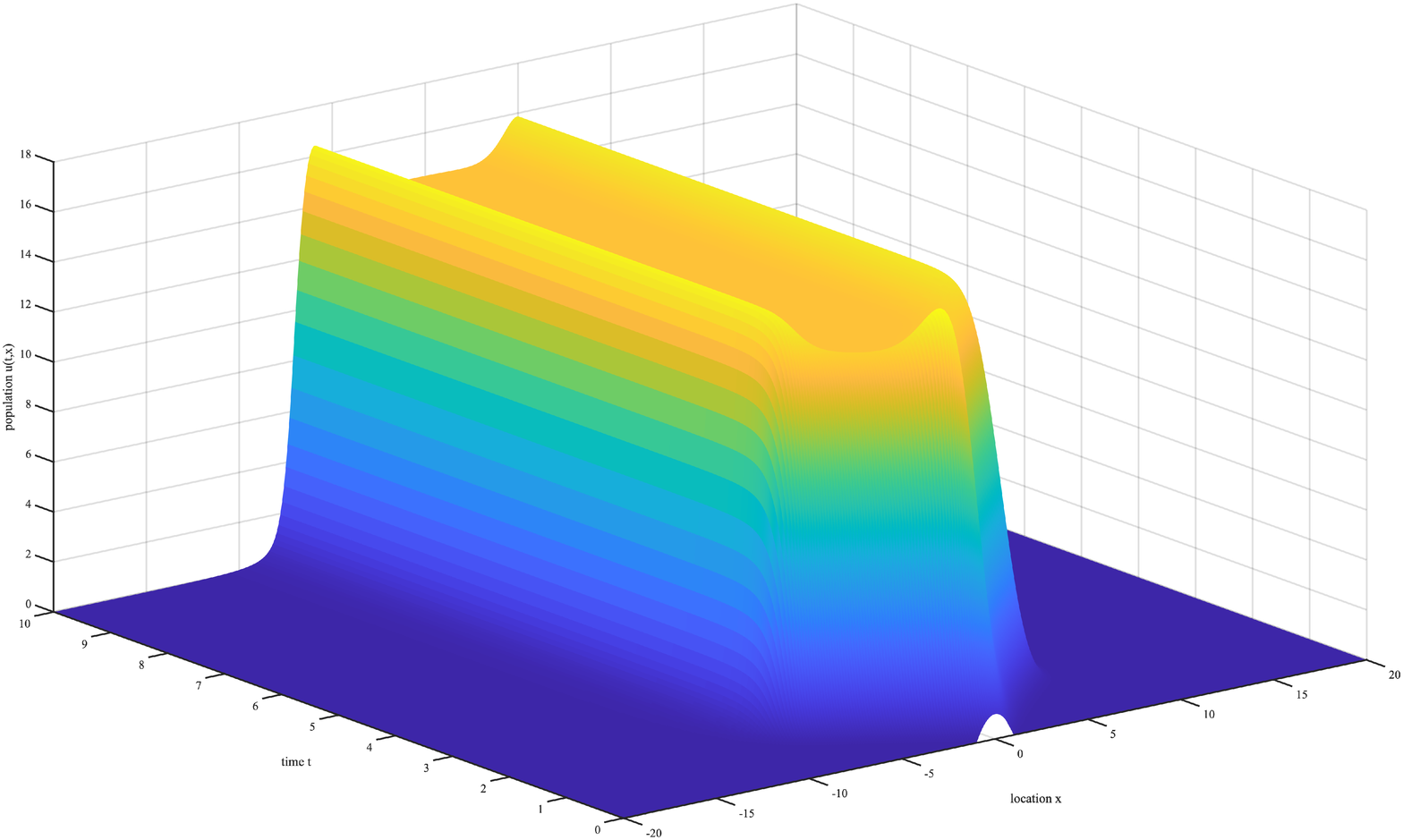} \label{Case2b07chi06L20-1}
	\end{minipage}} \quad
	\subfigure[]
{\begin{minipage}{7cm}
	\centering
	\includegraphics[scale=0.30]{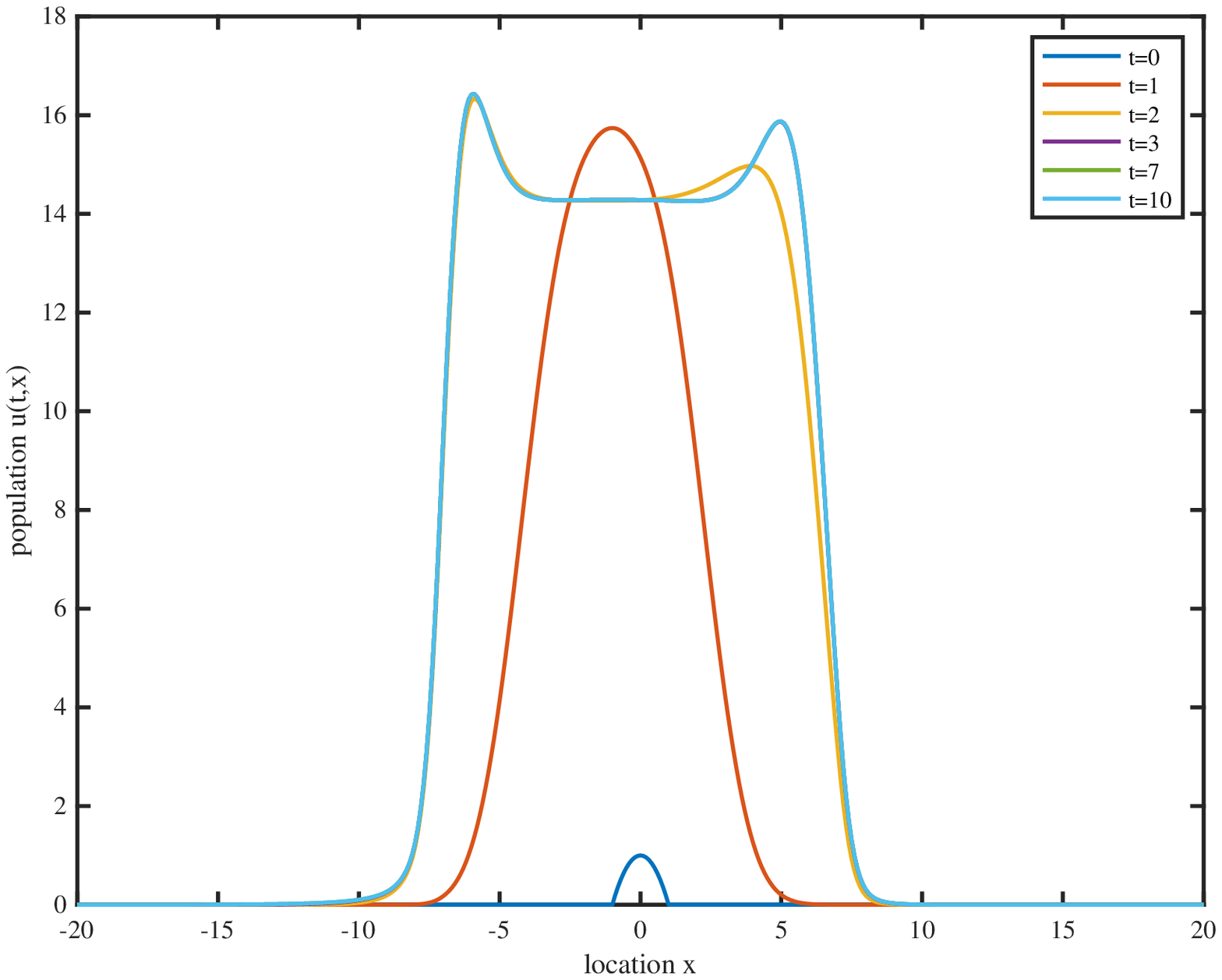} \label{Case2b07chi06L20}
	\end{minipage}}
 \caption{{\bf (a)} Evolution of numerical solution of \eqref{wave-cut-off-eq2-1} on the interval $[-20, 20]$ with $c=1$,  $b=0.7$ and  $\chi=0.6$. {\bf (b)} numerical solution of \eqref{wave-cut-off-eq2-1} on the interval $[-20, 20]$ at time $t=0, 1, 2, 3, 7, 10$ with $c=1$, $b=0.7$ and $\chi=0.6$.}
\label{Case2b07chi06L20-1-1}
\end{figure}

\begin{figure}[H]
\centering
\subfigure[]
{\begin{minipage}{7cm}
	\centering
	\includegraphics[scale=0.17]{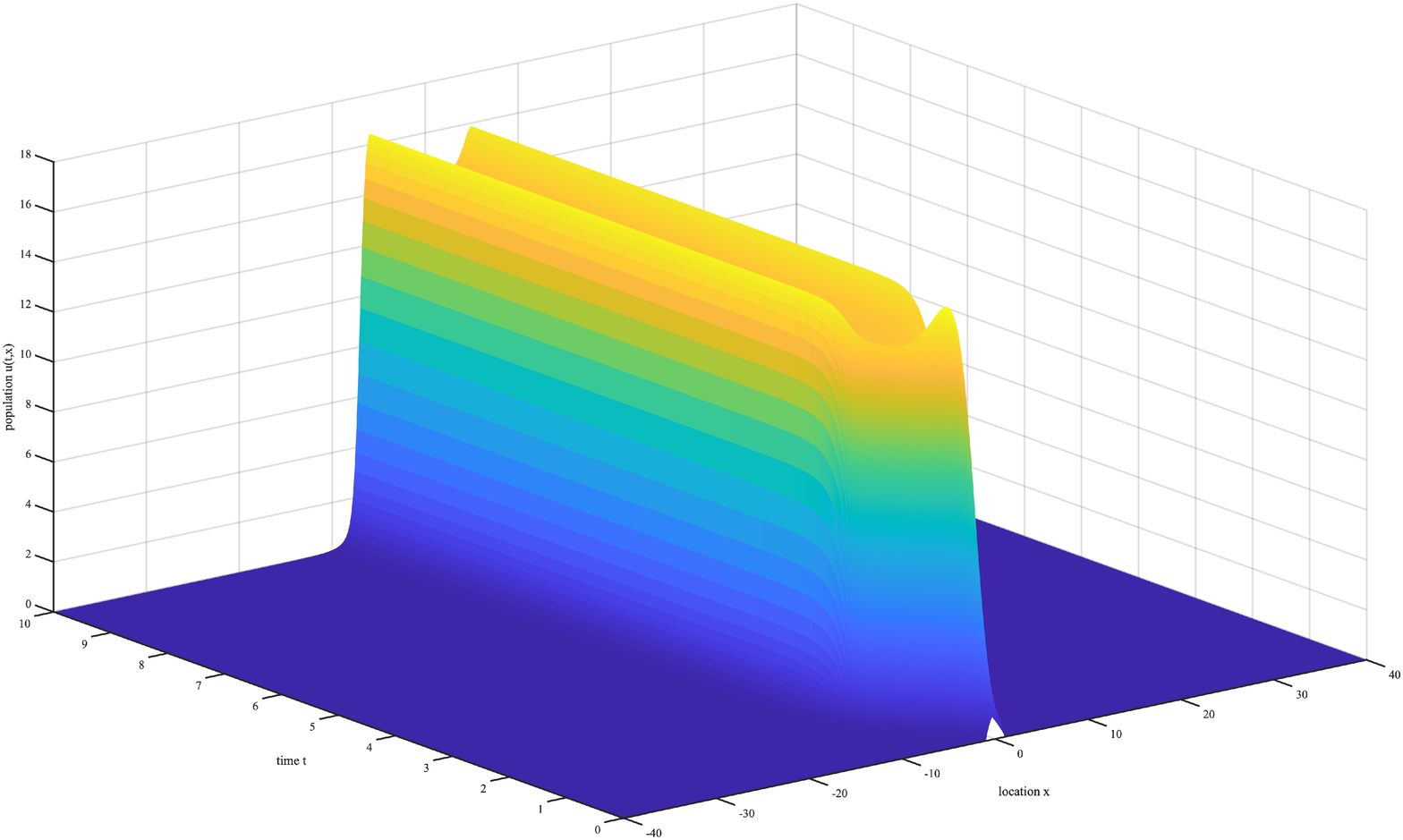} \label{Case2b07chi06L40-1}
	\end{minipage}} \quad
	\subfigure[]
{\begin{minipage}{7cm}
	\centering
	\includegraphics[scale=0.30]{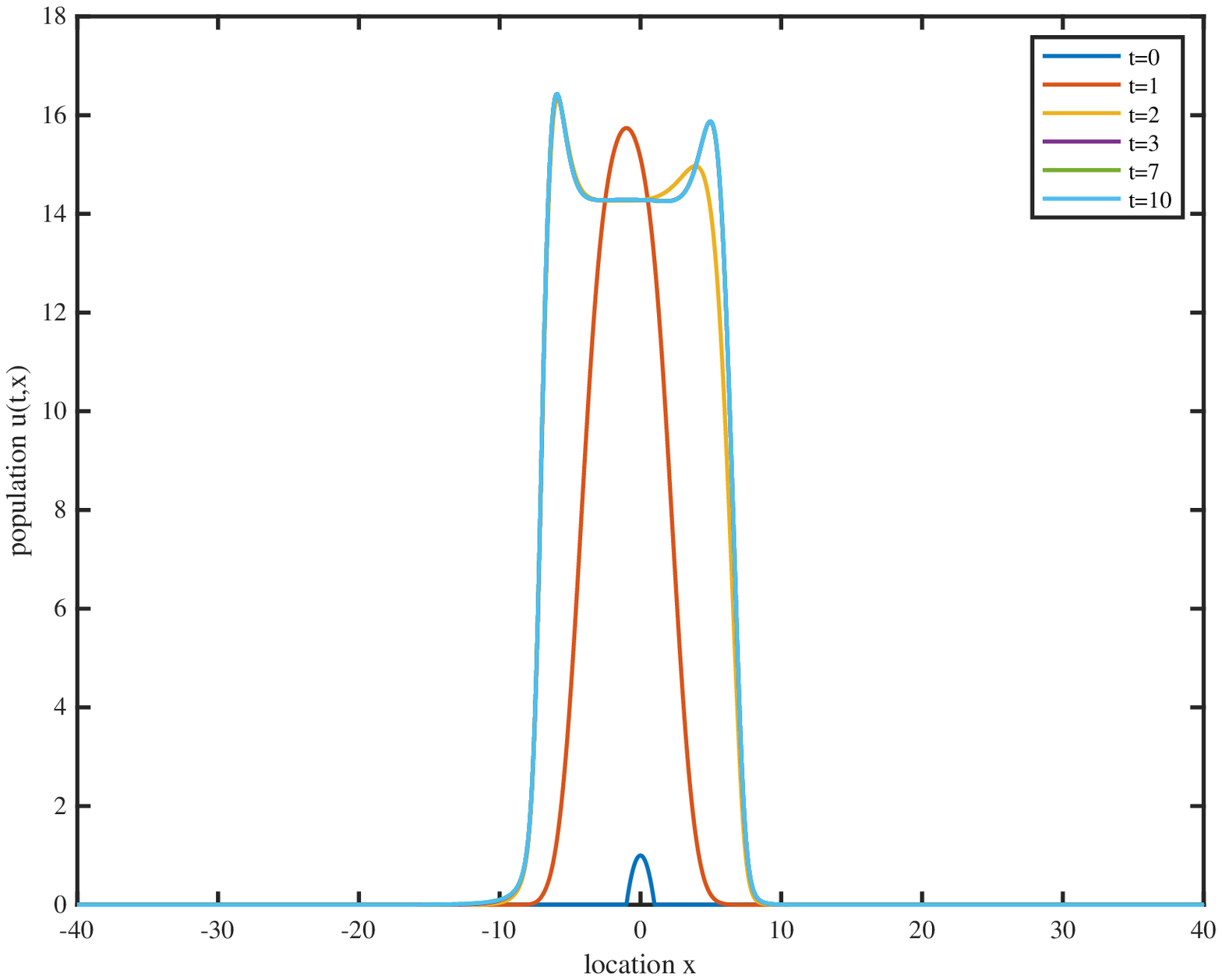} \label{Case2b07chi06L40}
	\end{minipage}}
 \caption{{\bf (a)} Evolution of numerical solution of \eqref{wave-cut-off-eq2-1} on the interval $[-40, 40]$ with $c=1$, $b=0.7$ and  $\chi=0.6$. {\bf (b)} numerical solution of \eqref{wave-cut-off-eq2-1} on the interval $[-40, 40]$ at time $t=0, 1, 2, 3, 7, 10$ with $c=1$, $b=0.7$ and $\chi=0.6$.}
\label{Case2b07chi06L40-1-1}
\end{figure}

\medskip

\noindent{\bf Numerical Experiment 3.}
Let $c=6.5$ (hence $c>c^*$).  Let $b=1$ and $\chi=0.6$ (hence $b> \frac{3\chi\mu}{2}$).

\smallskip

We compute the numerical solution of \eqref{wave-cut-off-eq2-1} with $L=15, 20, 25, 30, 40$ on the time interval $[0,30]$.
For all the choices of $L$, we observe that the numerical solution of \eqref{wave-cut-off-eq2-1}
becomes very small after $t=20$, which indicates that the numerical solution converges to the zero solution of \eqref{wave-cut-off-eq2-1}
as $t\to\infty$, and also indicates that
 \eqref{wave-eq} has no positive stationary solutions  or \eqref{Keller-Segel-eq0}  has  no forced wave solutions in the case that $c>c^*$ and $b>\frac{3}{2} \chi\mu$ which matches the theoretical result \cite[Theorem 1.3(1)]{ShXu}.
We demonstrate the numerical solutions of \eqref{wave-cut-off-eq2-1} for the case $L=20$ and $L=40$ in Figure \ref{Case2c65b1chi06L20-1-1}  and
Figure \ref{Case2c65b1chi06L40-1-1} , respectively.


\begin{figure}[H]
\centering
\subfigure[]
{\begin{minipage}{7cm}
	\centering
	\includegraphics[scale=0.17]{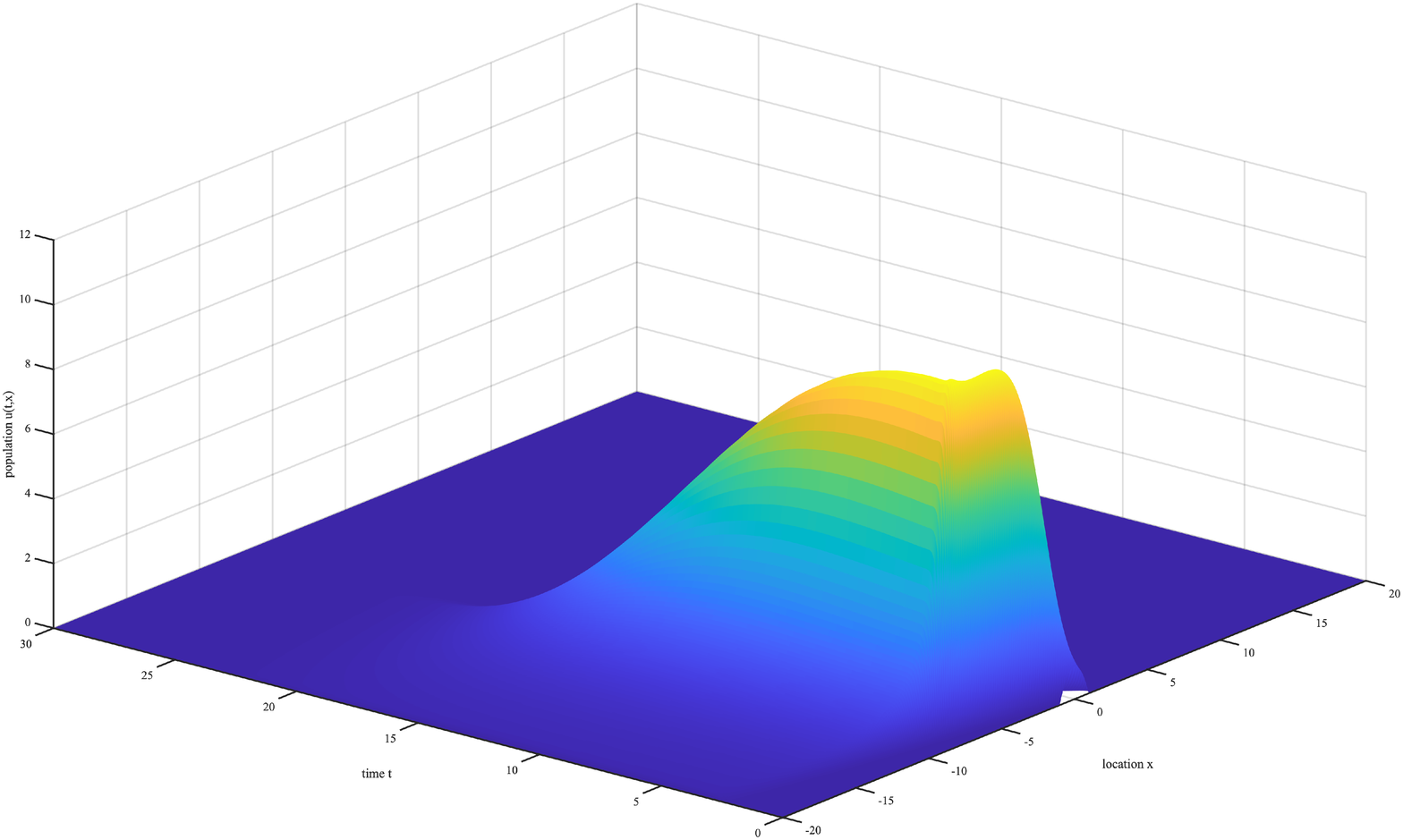} \label{Case2c65b1chi06L20-1}
	\end{minipage}} \quad
	\subfigure[]
{\begin{minipage}{7cm}
	\centering
	\includegraphics[scale=0.30]{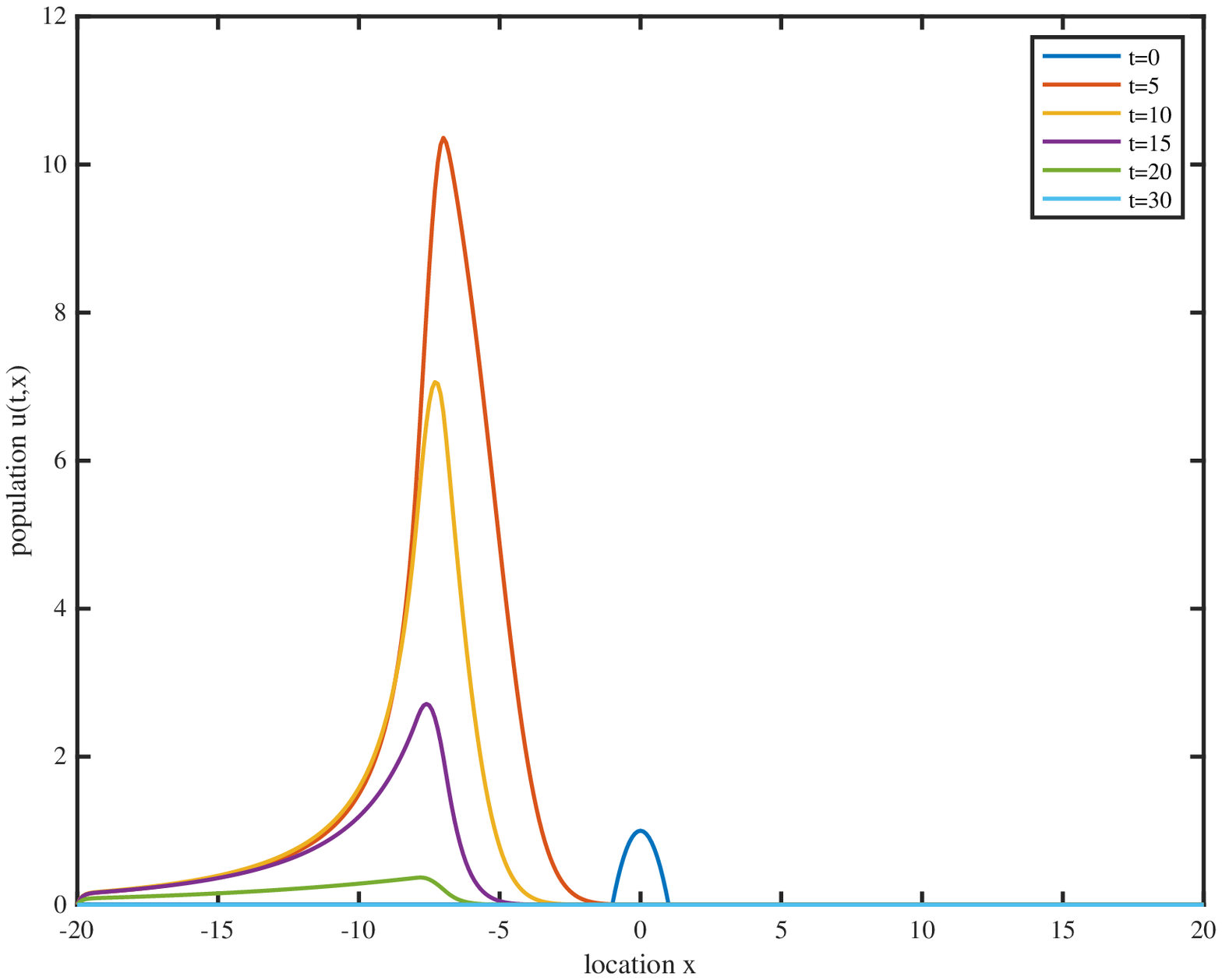} \label{Case2c65b1chi06L20}
	\end{minipage}}
 \caption{{\bf (a)} Evolution of numerical solution of \eqref{wave-cut-off-eq2-1} on the interval $[-20, 20]$ with $c=6.5$, $b=1$ and  $\chi=0.6$. {\bf (b)} numerical solution of \eqref{wave-cut-off-eq2-1} on the interval $[-20, 20]$ at time $t=0, 5, 10, 15, 20, 30$ with $c=6.5$, $b=1$ and $\chi=0.6$.}
\label{Case2c65b1chi06L20-1-1}
\end{figure}


\begin{figure}[H]
\centering
\subfigure[]
{\begin{minipage}{7cm}
	\centering
	\includegraphics[scale=0.17]{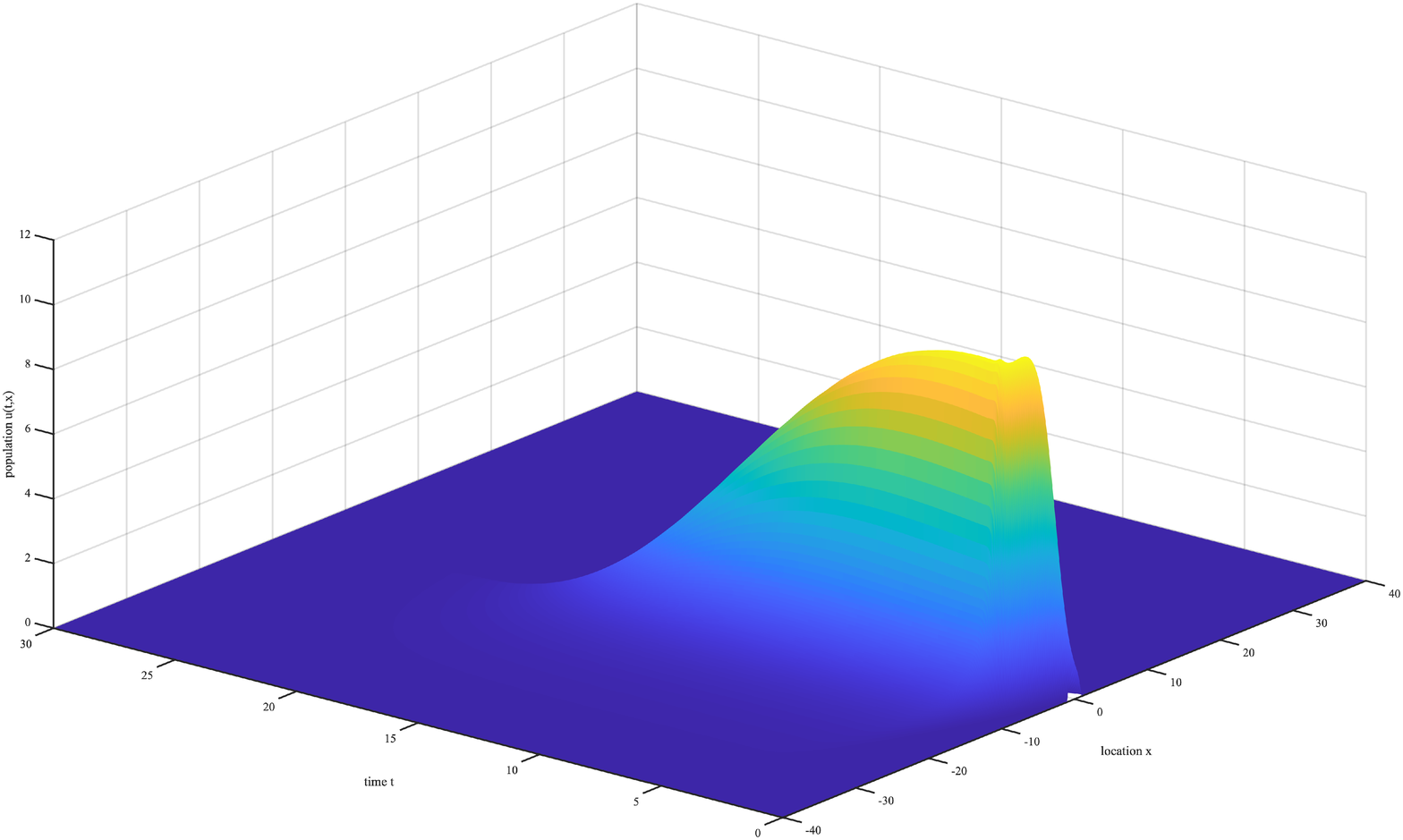} \label{Case2c65b1chi06L40-1}
	\end{minipage}} \quad
	\subfigure[]
{\begin{minipage}{7cm}
	\centering
	\includegraphics[scale=0.30]{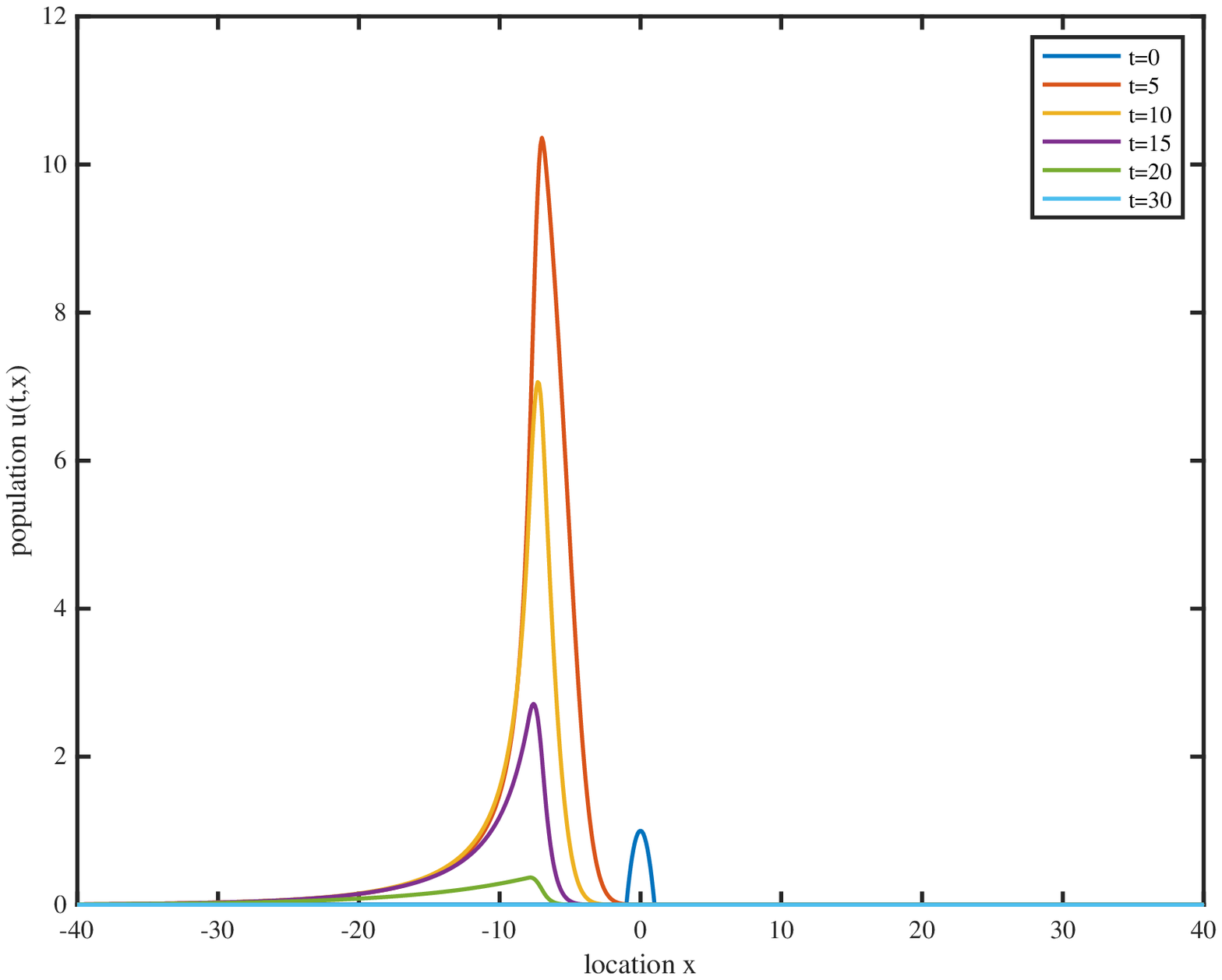} \label{Case2c65b1chi06L40}
	\end{minipage}}
 \caption{{\bf (a)} Evolution of numerical solution of \eqref{wave-cut-off-eq2-1} on the interval $[-40, 40]$ with $c=6.5$, $b=1$ and  $\chi=0.6$. {\bf (b)} numerical solution of \eqref{wave-cut-off-eq2-1} on the interval $[-40, 40]$ at time $t=0, 5, 10, 15, 20, 30$ with $c=6.5$, $b=1$ and $\chi=0.6$.}
\label{Case2c65b1chi06L40-1-1}
\end{figure}

Similarly, if $c=-6.5$, $b=1$ and $\chi=0.6$, we observe that the numerical solution of \eqref{wave-cut-off-eq2-1}
becomes very small after certain time, which indicates that the numerical solution converges to the zero solution of \eqref{wave-cut-off-eq2-1}
as $t\to\infty$, and also indicates that
 \eqref{wave-eq} has no positive stationary solutions  or \eqref{Keller-Segel-eq0}  has  no forced wave solutions in the case that $c<-c^*$ and $b>\frac{3}{2}\chi\mu$ which matches the theoretical result \cite[Theorem 1.3(1)]{ShXu}.

\begin{rk}
\begin{itemize}
\item[(1)]
The numerical simulations above supports our Theorem \ref{forced-wave-thm2} and also tells us that the assumptions in
Theorem \ref{forced-wave-thm2} may be weakened. Based on these numerical simulations, we conjecture that if $b>\chi\mu$ and $\lambda_{\infty}(r(\cdot))>0$, there is a forced wave solution $(u(t,x),v(t,x))=(\phi(x-ct),\psi(x-ct))$ connecting $(0,0)$ and $(0,0)$, that is,
$\phi(x)>0$ for all $x\in\R$ and $\phi(\pm\infty)=0$. If $b>\chi\mu$ and $|c|>c^*$, there is no forced wave solution $(u(t,x),v(t,x))=(\phi(x-ct),\psi(x-ct))$ connecting $(0,0)$ and $(0,0)$, that is,
$\phi(x)>0$ for all $x\in\R$ and $\phi(\pm\infty)=0$.

\item[(2)]
In these three numerical simulations, we used the same space step size $h=0.1$ and the same time step size $\tau=0.002$, which satisfy the numerical stable condition $\frac{\tau}{h^2}<\frac{1}{2}$. Again, we do not give the accuracy analysis of the simulations
in this paper. To see the  reliability of the numerical results,
we also tried different values of $h$ and $\tau$. For example, let $h=0.1$ be fixed, let $\tau=0.001, 0.002, 0.004$ respectively; let $h=0.2$ be fixed, let $\tau=0.01, 0.005, 0.0025$ respectively; let $h=0.05$ be fixed, let $\tau=0.001, 0.0005, 0.00025$ respectively. All the graphs we got do not change much.

\item[(3)]
We also tried to use different initial conditions to simulate the existence of forced wave solutions. For example, let
$$
u_0(x)=\begin{cases}
0 \quad  &{\rm if}\,\,  x\leq -1,\cr
sin(x+1)  \quad  &{\rm if}\,\,  -1<x<\pi-1,\cr
0                   \quad &{\rm if}\,\,  x\geq \pi-1.
\end{cases}
$$
We see similar dynamical scenarios. We then also conjecture that the forced wave solution of \eqref{Keller-Segel-eq0}
is unique and stable in certain parameter region.
\end{itemize}

\end{rk}


\end{document}